\documentclass[12 pt,reqno]{amsart}
\usepackage{enumitem}
\setlist[enumerate,1]{label={(\alph*)}}

\usepackage{pdfpages}
\usepackage{amsmath,amsthm,amssymb,amscd}
\usepackage{setspace}
\usepackage{upgreek}

\usepackage{cite}
\usepackage{url}
\usepackage{mathtools}
\usepackage{fullpage}
\usepackage{float}
\usepackage{comment}
\usepackage{longtable}

\usepackage{slashbox}
\usepackage{diagbox}

\usepackage[bookmarksopen,bookmarksdepth=3]{hyperref}

\DeclareFontFamily{U}{rcjhbltx}{}
\DeclareFontShape{U}{rcjhbltx}{m}{n}{<->rcjhbltx}{}
\DeclareSymbolFont{hebrewletters}{U}{rcjhbltx}{m}{n}

\DeclareMathSymbol{\mem}{\mathord}{hebrewletters}{109}
\DeclareMathSymbol{\memh}{\mathord}{hebrewletters}{110}
\DeclareMathSymbol{\nund}{\mathord}{hebrewletters}{141}
\newcommand{\nundh}{\hat{\nund}}

\usepackage{etoolbox}
\AtBeginEnvironment{tabular}{\onehalfspacing}

\usepackage{footnotebackref}

\input xy
\xyoption{all}

\usepackage{microtype}

\allowdisplaybreaks[1]

\newtheorem{thm}{Theorem}
\newtheorem{prop}{Proposition}[section]
\newtheorem{lm}[prop]{Lemma}
\newtheorem{cor}[prop]{Corollary}

\newtheorem{cl}[prop]{Proposition}

\theoremstyle{definition}
\newtheorem{dfn}[prop]{Definition}

\theoremstyle{remark}
\newtheorem{rem}[prop]{Remark}

\DeclareMathOperator{\Ker}{Ker}
\DeclareMathOperator{\Coker}{Coker}

\DeclareMathOperator{\Id}{Id}

\DeclareMathOperator{\im}{Im}

\DeclareMathOperator{\GW}{GW}
\DeclareMathOperator{\ogw}{OGW}
\DeclareMathOperator{\ogwb}{\overline{\ogw}}

\DeclareMathOperator{\rdim}{rel\,dim}

\newcommand{\fix}{\textit{Fix}}

\newcommand{\lrarr}{\longrightarrow}

\newcommand{\Lrarr}{\Longrightarrow}

\newcommand{\R}{\mathbb{R}}
\newcommand{\C}{\mathbb{C}}
\newcommand{\Z}{\mathbb{Z}}

\newcommand{\M}{\mathcal{M}}
\newcommand{\J}{\mathcal{J}}

\renewcommand{\P}{\mathbb{C}P}
\renewcommand{\L}{\Lambda}

\newcommand{\mI}{\mathcal{I}}
\newcommand{\g}{\Gamma}

\newcommand{\Hh}{\widehat{H}}

\newcommand{\qkl}{\mathfrak{q}_{\,k\!,\:l}}
\newcommand{\q}{\mathfrak{q}}

\newcommand{\m}{\mathfrak{m}}
\renewcommand{\d}{\partial}

\newcommand{\I}{[0,1]}
\newcommand{\at}{\tilde{\alpha}}

\newcommand{\xit}{\tilde{\xi}}
\newcommand{\etat}{\tilde{\eta}}

\newcommand{\Mt}{\widetilde{\M}}

\newcommand{\bt}{\tilde b}
\newcommand{\gt}{\tilde\gamma}

\newcommand{\qt}{\tilde\q}
\newcommand{\ct}{\tilde{c}}

\newcommand{\evbt}{\widetilde{evb}}
\newcommand{\evit}{\widetilde{evi}}
\newcommand{\evt}{\widetilde{ev}}

\renewcommand{\ll}{\langle\!\langle}
\renewcommand{\gg}{\rangle\!\rangle}

\renewcommand{\Im}{\im}

\newcommand{\Oh}{\widehat{\Omega}}
\newcommand{\Ob}{\overline{\Omega}}

\newcommand{\RP}{\mathbb{R}P}

\newcommand{\oFS}{\omega_{FS}}

\newcommand{\sly}{\Pi}
\newcommand{\pr}{\varpi}

\renewcommand{\i}{\mathfrak{i}}

\newcommand{\Qh}{Q}

\newcommand{\qbg}{\q^{\gamma,b}}
\newcommand{\qg}{\q^{\gamma}}
\newcommand{\qtbg}{\qt^{\gt,\bt}}

\newcommand{\xb}{\bar{x}}

\renewcommand{\a}{\alpha}

\newcommand{\mf}{f}
\newcommand{\dmf}{\d_s\mf}

\newcommand{\Ah}{\widehat{A}}

\newcommand{\cone}{C(\i)}
\newcommand{\coneq}{\widehat C(\i)}
\newcommand{\Qhh}{\widehat{\Qh}}
\renewcommand{\k}{\Bbbk}
\newcommand{\Lc}{\L_c}
\newcommand{\mJ}{\mathcal{K}}
\newcommand{\dcone}{d_C}

\newcommand{\memhh}{\hat{\memh}}

\newcommand{\s}{\mathfrak{s}}

\renewcommand{\k}{\Bbbk}

\makeatletter
\DeclareRobustCommand{\cev}[1]{%
  \mathpalette\do@cev{#1}%
}
\newcommand{\do@cev}[2]{%
  \fix@cev{#1}{+}%
  \reflectbox{$\m@th#1\vec{\reflectbox{$\fix@cev{#1}{-}\m@th#1#2\fix@cev{#1}{+}$}}$}%
  \fix@cev{#1}{-}%
}
\newcommand{\fix@cev}[2]{%
  \ifx#1\displaystyle
    \mkern#23mu
  \else
    \ifx#1\textstyle
      \mkern#23mu
    \else
      \ifx#1\scriptstyle
        \mkern#22mu
      \else
        \mkern#22mu
      \fi
    \fi
  \fi
}

\makeatother

\usepackage{xpatch}

\makeatletter
\xpretocmd{\@adminfootnotes}{\let\@makefntext\BHFN@OldMakefntext}{}{}
\renewcommand\@makefntext[1]{%
  \@ifundefined{@makefnmark}
    {}
    {%
     \renewcommand\@makefnmark{%
       \mbox{%
         \textsuperscript{%
           \normalfont
           \hyperref[\BackrefFootnoteTag]{\@thefnmark}%
         }%
       }\,%
     }%
     \BHFN@OldMakefntext{#1}%
  }%
}
\makeatother

\newcommand{\rI}{\cev{I}}
\newcommand{\ga}{\Delta}
\DeclareMathOperator{\spa}{Span}

\usepackage{marginnote}

\begin{document}

\title{Relative quantum cohomology}

\author[J. Solomon]{Jake P. Solomon}
\address{Institute of Mathematics\\ Hebrew University, Givat Ram\\Jerusalem, 91904, Israel } \email{jake@math.huji.ac.il}
\author[S. Tukachinsky]{Sara B. Tukachinsky}
\address{School of Mathematical Sciences\\ Tel Aviv University\\Tel Aviv, 6997801, Israel }\email{sarabt1@gmail.com}

\date{August 2022}

\begin{abstract}
We establish a system of PDE, called open WDVV, that constrains the bulk-deformed superpotential and associated open Gromov-Witten invariants of a Lagrangian submanifold $L \subset X$ with a bounding chain. Simultaneously, we define the quantum cohomology algebra of $X$ relative to $L$ and prove its associativity. We also define the relative quantum connection and prove it is flat. A wall-crossing formula is derived that allows the interchange of point-like boundary constraints and certain interior constraints in open Gromov-Witten invariants. Another result is a vanishing theorem for open Gromov-Witten invariants of homologically non-trivial Lagrangians with more than one point-like boundary constraint. In this case, the open Gromov-Witten invariants with one point-like boundary constraint are shown to recover certain closed invariants. From open WDVV and the wall-crossing formula, a system of recursive relations is derived that entirely determines the open Gromov-Witten invariants of $(X,L) = (\C P^n, \R P^n)$ with $n$ odd, defined in previous work of the authors. Thus, we obtain explicit formulas for enumerative invariants defined using the Fukaya-Oh-Ohta-Ono theory of bounding chains.
\end{abstract}

\maketitle

\tableofcontents

\section{Introduction}

\subsection{Overview}
Let $(X,\omega)$ be a symplectic manifold with $\dim_\R X = 2n$, and let $L \subset X$ be a Lagrangian submanifold. We assume $L$ admits a relative spin structure and fix one, so the standard theory of orientations for moduli spaces of $J$-holomorphic disks~\cite{FOOO,Seidel} applies.

A bounding chain for $L$ is a solution of the Maurer-Cartan equation in the Fukaya $A_\infty$ algebra of $L.$ Bounding chains provide a systematic method to compensate for the disk bubbling phenomenon that generally spoils the invariance of counts of $J$-holomorphic curves with boundary. There is a natural equivalence relation on bounding chains known as gauge equivalence. The superpotential of $L$ is a function on the space of bounding chains modulo gauge equivalence. The superpotential counts $J$-holomorphic disks in $X$ with boundary in $L$ constrained to pass through the given bounding chain. Cycles on $X$ give rise to deformations of the Fukaya $A_\infty$ algebra, bounding chains, and superpotential of $L$, known as bulk deformations. The deformed superpotential counts $J$-holomorphic disks in $X$ with boundary on $L$ with the interior constrained to pass through the cycles in $X$ and the boundary constrained to pass through the deformed bounding chain. If one can invariantly parameterize the space of bounding chains modulo gauge equivalence, the superpotential becomes a generating function for open Gromov-Witten invariants~\cite{Fukaya2,ST2}. Thus, the superpotential is an analog in open Gromov-Witten theory of the genus zero closed Gromov-Witten potential in closed Gromov-Witten theory. Indeed, the closed Gromov-Witten potential is a generating function for genus zero closed Gromov-Witten invariants, which count $J$-holomorphic spheres in $X.$ To invariantly count $J$-holomorphic disks with boundary contractible in $L$, it is natural to define an enhanced superpotential that includes contributions from $J$-holomorphic spheres as well as disks~\cite{Fukaya2,Joyce,PandharipandeSolomonWalcher}. The sphere contributions compensate for the phenomenon of the boundary of the disk collapsing to a point.

We prove a system of quadratic PDE, called the open WDVV equations, satisfied by the bulk-deformed enhanced superpotential of an arbitrary $L$ equipped with a bounding chain for bulk deformations in a Frobenius subalgebra of big quantum cohomology $U \subset QH(X).$ See Theorem~\ref{thm:OWDVV}. The coefficients of open WDVV are given by the partial derivatives of the closed Gromov-Witten potential. Viewing the superpotential as a generating series of open Gromov-Witten invariants, open WDVV gives rise to a system of quadratic equations relating Gromov-Witten invariants of different degrees, both closed and open. Thus, open WDVV is an analog in open Gromov-Witten theory of the WDVV equations in closed Gromov-Witten theory~\cite{KontsevichManin,RuanTian0,Witten2}.

Simultaneously, for a Lagrangian submanifold equipped with a bounding chain for bulk deformations in a Frobenius subalgebra $U \subset QH(X),$ we define the relative quantum cohomology algebra $QH_U(X,L)$ and prove its associativity. See Theorem~\ref{thm:mem}. Denoting by $Q$ the relevant Novikov ring, we have a long exact sequence of $Q$-modules,
\[
\xymatrix{
QH_U(X,L) \ar[rr] && U \ar[dl]^{\int_L} \\
 & Q[n] \ar[ul]^{[1]}
}
\]
where the top arrow is an algebra homomorphism.
The algebra structure on $QH(X,L)$ is given by counting both $J$-holomorphic spheres in $X$ and $J$-holomorphic disks in $X$ with boundary in $L$. When $L$ is a real cohomology sphere, one may consider $U = QH(X)$ and find a bounding chain for all associated bulk-deformations as shown in Theorem~\ref{thm:esphere}. A typical situation in which it is useful to consider $U \subset QH(X)$ a proper subalgebra is when an anti-symplectic involution fixing $L$ is used to construct the bounding chain as in Theorem~\ref{thm:cases}.

To naturally integrate both disk and sphere contributions in the enhanced superpotential, the open WDVV system, and the relative quantum product, we introduce a cone complex $\cone.$ We define a relative potential $\Psi \in \cone$ and a tensor potential $\memh : \cone \to \cone,$ which combine both the enhanced superpotential and the closed Gromov-Witten potential. The open WDVV equations and associativity of relative quantum cohomology follow from a commutation relation for partial derivatives of the tensor potential, which holds up to chain homotopy. See Theorem~\ref{lm:assoc}. We interpret the commutation relation for partial derivatives of $\memh$ as the flatness of the relative quantum connection in Corollary~\ref{cy:flat}.

The chain homotopy underlying the open WDVV equations and associativity is constructed from operators on the Fukaya $A_\infty$ algebra of $L$ associated to moduli spaces of $J$-holomorphic disks with three marked points constrained to a geodesic. These moduli spaces come in three different families with either one, two, or three, of the marked points on the geodesic being in the interior of the disk, while the rest are on the boundary. In addition, the construction of the chain homotopy uses a family of $J$-holomorphic sphere moduli spaces with the cross ratio of four marked points constrained to be real. Bubbling in Gromov converging sequences of $J$-holomorphic disks with three marked points constrained to a geodesic gives rise to $A_\infty$ type relations for the geodesic operators. The Maurer-Cartan equation satisfied by the bounding chain cancels the terms of the geodesic $A_\infty$ relations corresponding to many types of bubbling. The remaining types of bubbling give rise to the open WDVV equations. Figure~\ref{pic:geodbubbling} shows the two types of bubbling of $J$-holomorphic disks with three interior marked points on a geodesic that contribute to the open WDVV equation. In the bubbling depicted on the left, the components of the Gromov limit are a disk without any geodesic constraints and a sphere. This type of bubbling gives rise to products of the enhanced superpotential and the closed Gromov-Witten potential in the open WDVV equations~\eqref{eq:cor1}-\eqref{eq:cor2}. In the bubbling depicted on the right, one of the two disk components of the limit open stable map retains the geodesic constraint. Namely, two of the interior marked points and the node must lie on the geodesic. A priori, it is not clear how to interpret the disk component with the geodesic constraint in terms of the enhanced superpotential, which counts disks without a geodesic constraint. Remarkably, the Maurer-Cartan equation and the interaction of the geodesic constraint with the unit of the Fukaya $A_\infty$ algebra of $L$ nonetheless allow the open WDVV equations to capture this type of bubbling with quadratic expressions in the enhanced superpotential.

\begin{figure}[ht]
\centering
\includegraphics[width=12cm]{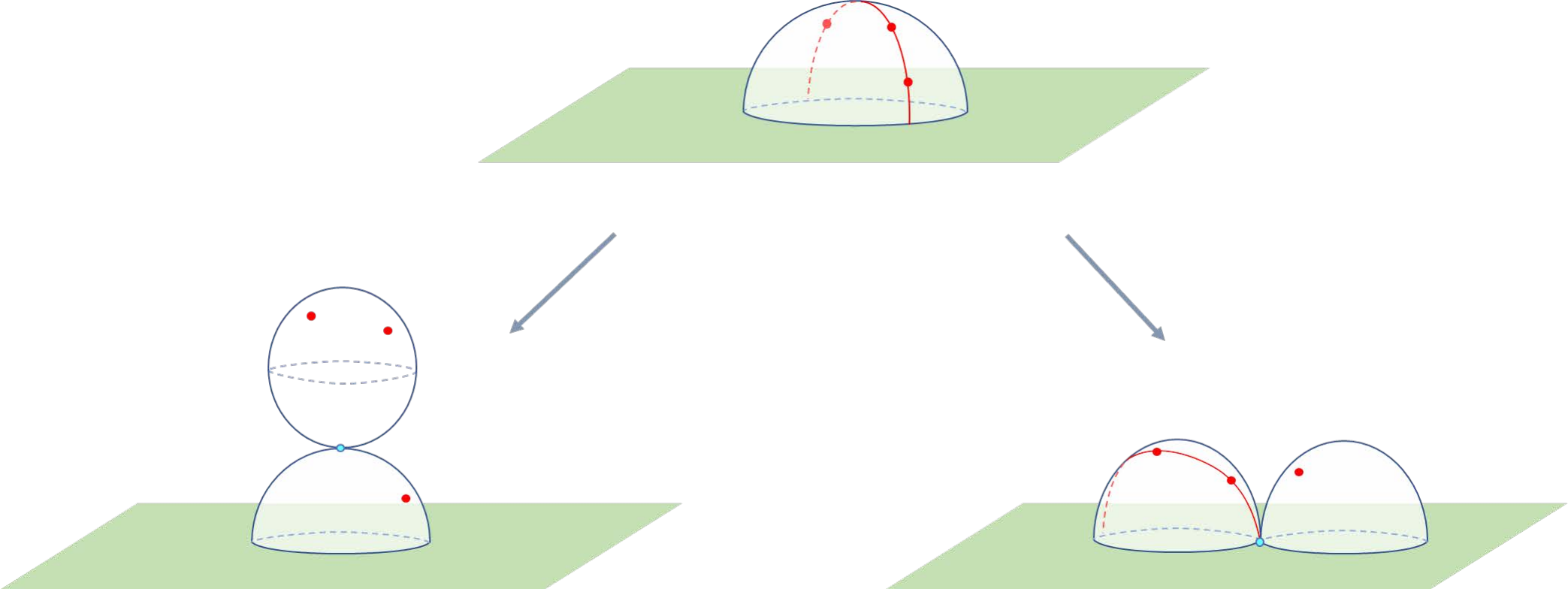}
\caption{The bubbling of $J$-holomorphic disks with three interior points constrained to a geodesic.}\label{pic:geodbubbling}
\end{figure}

Again using the tensor potential $\memh$, we obtain a wall-crossing formula for open Gromov-Witten invariants that allows one to exchange a boundary constraint for a certain type of interior constraint. See Theorem~\ref{prop:wallcross}. Furthermore, from the proof that $\memh$ is a chain map, we derive a vanishing theorem for open Gromov-Witten invariants with more than one boundary constraint in the case that $[L] \neq 0 \in H_n(X).$ In this case, we show the open Gromov-Witten invariants with one boundary constraint recover certain closed invariants. See Theorem~\ref{prop:formforc}.

We apply the open WDVV equations and the wall-crossing formula to calculate the superpotential open Gromov-Witten invariants of~\cite{ST2} for $(X,L) = (\C P^n,\R P^n)$ with $n$ odd. See Corollary~\ref{cor:computable} and sample values in Section~\ref{ssec:sample}. When $n = 3,$ it is shown in~\cite{ST2} that the superpotential invariants recover Welschinger's invariants~\cite{Welschinger2}. Thus, our calculations recover those of~\cite{BrugalleGeorgieva,BrugalleMikhalkin1}. For arbitrary odd $n,$ interior constraints restricted to odd powers of $\omega$, and no boundary constraints, it is shown in~\cite{ST2} that the superpotential invariants recover the invariants of Georgieva~\cite{Georgieva}. Thus, our calculations recover those of~\cite{GZ0,GZ}. When at least one interior constraint is an odd power of $\omega,$ the invariants of Georgieva vanish. However, our calculations show that the superpotential invariants do not vanish. Also, the superpotential invariants for arbitrary odd $n$ allow for boundary constraints that behave like real point constraints for Welschinger's invariants in dimension~$3$. Point-like boundary constraints are not allowed in other constructions for $n > 3.$ Our calculations show the superpotential invariants with point-like boundary constraints are non-trivial for~$n > 3.$

The open WDVV equations of the present work are an extension of the equations introduced in~\cite{SolomonPreprint} in the real setting for $n = 2.$ See also~\cite{HorevSolomon} and~\cite{Chen2d}. We establish the open WDVV equations without regard to real structure and in any dimension. A preliminary version of our results appeared in~\cite{TThesis}. A discussion of the formal properties of open WDVV in arbitrary dimension appeared in~\cite{Alcolado}. The real WDVV equations of~\cite{GZ0} can be obtained from open WDVV by setting certain parameters to zero. Recently, the preprint~\cite{ChenZinger} has appeared, which obtains some of our results in the real setting when $n = 3$ by different methods.

\subsection{Background}
\subsubsection{Notation}

Consider a symplectic manifold $(X,\omega)$ of dimension $2n$ and a connected, Lagrangian submanifold $L$ with relative spin structure $\s.$ Let $J$ be an $\omega$-tame almost complex structure on $X.$ Denote by $\mu:H_2(X,L) \to \Z$ the Maslov index. Denote by $A^*(L)$ the ring of differential forms on $L$ with coefficients in $\R$.
Let $\sly$ be a quotient of $H_2(X,L;\Z)$ by a possibly trivial subgroup $S_L$ contained in the kernel of the homomorphism $\omega \oplus \mu : H_2(X,L;\Z) \to \R \oplus \Z.$ Thus the homomorphisms $\omega,\mu,$ descend to $\sly.$ Denote by $\beta_0$ the zero element of $\sly.$ Let
\begin{equation}\label{eq:varpi}
\pr: H_2(X;\Z) \to \sly
\end{equation}
denote the composition of the natural map $H_2(X;\Z) \to H_2(X,L;\Z)$ with the projection $H_2(X,L;\Z) \to \sly.$

\subsubsection{Coefficient rings}
Define Novikov coefficient rings
\begin{gather*}
\L=\left\{\sum_{i=0}^\infty a_iT^{\beta_i}\bigg|a_i\in\R,\beta_i\in \sly,\omega(\beta_i)\ge 0,\; \lim_{i\to \infty}\omega(\beta_i)=\infty\right\},\\
\Lc:= \left\{\sum_{j=0}^\infty a_jT^{\varpi(\beta_j)}\,| \,a_j\in \R, \beta_j\in H_2(X;\Z), \omega(\beta_j)\ge 0,\lim_{j\to \infty}\omega(\beta_j)=\infty\right\}\leqslant \L.
\end{gather*}
Gradings on $\L,\Lc$ are defined by declaring $T^\beta$ to be of degree $\mu(\beta).$ Define ideals $\L^+ \triangleleft \L$ and $\Lc^+ \triangleleft \Lc$ by
\[
\L^+=\left\{\sum_{i=0}^\infty a_iT^{\beta_i} \in \L \bigg|\;\omega(\beta_i)> 0 \quad\forall i\right\}, \qquad \Lc^+ = \Lc\cap\L^+ .
\]

For a graded real vector space $V$, let $\R[[V]]$ denote the ring of formal functions on the completion of $V$ at the origin and let $m_V \subset \R[[V]]$ denote the unique maximal ideal. More explicitly, let $\{v_i\}_{i \in I_V}$ be a homogeneous basis of $V,$ let $\{v_i^*\}_{i \in I_V}$ be the dual basis of $V^*$, and let $t_i$ be a formal variable of degree $-|v_i|.$ We will often identify $\R[[t_i]]_{i \in I_V} \overset{\sim}\to \R[[V]]$ by the natural isomorphism taking $t_i$ to $v_i^*.$ Under this isomorphism, the ideal $\langle t_i\rangle_{i \in I_V}$ is identified with the ideal $m_V.$ Since each tangent space of $V$ is canonically isomorphic to $V,$ the $\R[[V]]$ module of formal vector fields on $V$ is canonically isomorphic to $V \otimes \R[[V]]$. Each formal vector field $v \in V \otimes\R[[V]]$ gives rise to a derivation $\partial_v : \R[[V]] \to \R[[V]].$  In coordinates, if $v = \sum_i f_i v_i$ with $f_i \in \R[[V]]$, then $\partial_v = \sum_i f_i \d_i.$ For $l \in \Z,$ let $V[l]$ denote the shift of $V$ by $l.$ That is, $V[l]$ is the graded vector space with $V[l]^i = V^{i+l}.$ Let $S$ be another graded real vector space. Write
\begin{gather*}
R_V:=\L \otimes \R[[V[2]\oplus S[1]]],\\
\Qh_V:=\Lc \otimes \R[[V[2]]] \leqslant R_V.
\end{gather*}
The vector space $V$ will be used to parameterize deformations associated with marked points in the interior of a Riemann surface while $S$ will be used to parameterize deformations associated with marked points on the boundary of a Riemann surface. So, the grading of $V$ is shifted by the real dimension of a Riemann surface and the grading of $S$ is shifted by the dimension of the boundary.
Define ideals $\mJ_V \triangleleft R_V$
and $\mI_V \triangleleft Q_V$ by
\[
\mJ_V = \L^+R_V + m_VR_V + m_S R_V,\qquad \mI_V = \Lc^+Q_V + m_V Q_V.
\]
We may drop the subscript $V$ from the notations $Q_V, R_V, \mI_V,\mJ_V,$ in statements that hold for all choices of $V.$
Denote by $\Gamma_V \in V \otimes\Qh_V$ the vector field corresponding to the parity operator $\mathcal P:V \to V$ given by $\mathcal P(v) = (-1)^{\deg v}v.$ That is,
\[
\Gamma_V = \sum_{i \in I_V} (-1)^{|v_i|}v_i \otimes v_i^* = \sum_{i \in I_V} t_i v_i.
\]
Note that $|\Gamma_V| = 2.$

Below, we assume that each index set $I_V$ for a basis of a vector space $V$ is endowed with an order, and we implicitly use this order in every graded commutative product over $i \in I_V.$ We denote by $\rI_V$ the same set with the order reversed.
We reserve the formal variables $\{s_i\}_{i \in I_S}$ for coordinate functions on the vector space $S$. We abbreviate $s^k = \prod_{i \in I_S} s_i^{k_i}$ and $k! = \prod_{i \in I_S}k_i!.$

\subsubsection{Moduli spaces}
Let $\M_{k+1,l}(\beta)$ be the moduli space of genus zero $J$-holomorphic open stable maps $u:(\Sigma,\d \Sigma) \to (X,L)$ of degree $[u_*([\Sigma,\d \Sigma])] = \beta \in \sly$ with one boundary component, $k+1$ boundary marked points, and $l$ interior marked points. The boundary points are labeled according to their cyclic order. The space $\M_{k+1,l}(\beta)$ carries evaluation maps associated to boundary marked points $evb_j^\beta:\M_{k+1,l}(\beta)\to L$, $j=0,\ldots,k$, and evaluation maps associated to interior marked points $evi_j^\beta:\M_{k+1,l}(\beta)\to X$, $j=1,\ldots,l$.

Let $\M_{l+1}(\beta)$ be the moduli space of genus zero $J$-holomorphic stable maps $u: \Sigma \to X$ of degree $u_*([\Sigma]) = \beta\in H_2(X;\Z)$ with $l+1$ marked points. The space $\M_{l+1}(\beta)$ carries evaluations maps $ev_j^\beta:\M_{l+1}(\beta)\to X$, $j = 0,\ldots,l.$

We assume that all $J$-holomorphic genus zero open stable maps with one boundary component are regular, the moduli spaces $\M_{k+1,l}(\beta;J)$ are smooth orbifolds with corners, and the evaluation maps $evb_0^\beta$ are proper submersions. Furthermore, we assume that all the moduli spaces $\M_{l+1}(\beta)$ are smooth orbifolds and $ev_0$ is a submersion.
Examples include $(\P^n,\RP^n)$ with the standard symplectic and complex structures or, more generally, flag varieties, Grassmannians, and products thereof. See Example 1.5 and Remark 1.6 in~\cite{ST1}.
Throughout the paper we fix a connected component $\mathcal{J}$ of the space of $\omega$-tame almost complex structures satisfying our assumptions.  All almost complex structures are taken from $\J.$ The results and arguments of the paper extend to general target manifolds with arbitrary $\omega$-tame almost complex structures if we use the virtual fundamental class techniques of~\cite{Fukaya,Fukaya2,FOOOtoricI,FOOOtoricII,FOOO1}. Alternatively, it should be possible to use the polyfold theory of~\cite{HoferWysockiZehnder,HoferWysockiZehnder1,HoferWysockiZehnder2,HoferWysockiZehnder3,LiWehrheim}.
See Section~\ref{sssec:reg} for a detailed discussion on regularity assumptions.

\subsubsection{Operations}
We encode the geometry of the moduli spaces of open stable maps in operations
\[
\qkl:A^*(L;\L)^{\otimes k}\otimes A^*(X)^{\otimes l}\lrarr A^*(L;\L)
\]
defined by
\begin{multline*}
\qkl(\alpha_1,\ldots,\alpha_k;\eta_1,\ldots,\eta_l):=\\
=\delta_{k,1}\cdot\delta_{l,0}\cdot d\alpha_1+(-1)^{\sum_{j=1}^kj(|\alpha_j|+1)+1}
\sum_{\beta\in\sly} T^{\beta}{evb_0^\beta}_* (\bigwedge_{j=1}^k (evb_j^\beta)^*\alpha_j\wedge \bigwedge_{j=1}^l(evi_j^\beta)^*\eta_j).
\end{multline*}
The push-forward $(evb_0^\beta)_*$ is defined by integration over the fiber; it is well-defined because $evb_0^\beta$ is a proper submersion.
Intuitively, the $\eta_i$ should be thought of as interior constraints, while $\alpha_j$ are boundary constraints. Then the output is a cochain on $L$ that is ``Poincar\'e dual'' to the image of the boundaries of disks that satisfy the given constraints.

We define similar operations using moduli spaces of stable maps,
\[
\q_{\emptyset,l}:A^*(X;\Lc)^{\otimes l}\lrarr A^*(X;\Lc),
\]
as follows. Recall that the relative spin structure $\s$ on $L$ determines a class $w_{\s} \in H^2(X;\Z/2\Z)$ such that $w_2(TL) = i^* w_{\s}$. By abuse of notation we think of $w_{\s}$ as acting on $H_2(X;\Z)$. Set
\begin{equation}\label{eq:qemptyset}
\q_{\emptyset,l}(\eta_1,\ldots,\eta_l):=
\sum_{\beta\in H_2(X;\Z)}
(-1)^{w_\s(\beta)}
T^{\varpi(\beta)}{ev_0^\beta}_* (\bigwedge_{j=1}^l(ev_j^\beta)^*\eta_j).
\end{equation}
The sign $(-1)^{w_\s(\beta)}$ is designed to balance out the sign of gluing spheres as explained in~\cite[Lemma 2.12]{ST1}.
Below, we use the same notation for the linear extensions of these operations to spaces of differential forms with larger coefficient rings.

\subsubsection{Bounding chains}
Consider the subcomplex of differential forms on $X$ consisting of those with trivial integral on $L$,
\[
\Ah^*(X,L):=\left\{\eta\in A^*(X)\;\bigg|\,\int_L\eta=0\right\}.
\]
For an $\R$-algebra $\Upsilon$, write
\[
 \Ah(X,L;\Upsilon) = \Ah(X,L)\otimes\Upsilon, \qquad \Hh^*(X,L;\Upsilon):=H^*(\Ah(X,L;\Upsilon),d).
\]

\begin{dfn}\label{dfn:bdpair}
A pair $(\gamma,b)\in \mI_V A^*(X;\Qh_V)\oplus \mJ_V A^*(L;R_V)$ is called a \textbf{bounding pair} if
$d\gamma=0$, $|\gamma|=2,$
$|b|=1,$ and
there exists $c\in \mJ_V$
such that $|c|=2$ and the Maurer-Cartan equation holds,
\begin{equation}\label{eq:bdchdfn}
\sum_{k,l\ge 0} \frac{1}{l!}\qkl(b^{\otimes k};\gamma^{\otimes l})
= c\cdot 1.
\end{equation}
In this case, we call $b$ a \textbf{bounding chain.}
Let $W \subset \Hh^*(X,L;\R)$ be a graded vector subspace. A \textbf{bounding pair over $W$} is a bounding pair $(\gamma,b) \in \mI_W\Ah^*(X,L;\Qh_W )\oplus \mJ_{W}A^*(L;R_W)$ with $[\gamma] = \Gamma_W.$
We say a bounding chain $b$ is \textbf{separated} if
\[
\int_L b\in \Lambda\otimes \R[[S[1]]] \subset R_W.
\]
A bounding chain is \textbf{point-like} if the vector space $S$ is one-dimensional and $\int_L b = s,$ where $s$ is the single coordinate on $S.$
\end{dfn}
The definition of a bounding chain is due to Fukaya-Oh-Ohta-Ono~\cite{FOOO}. The notion of a point-like bounding chain is due to~\cite{ST2}. In Remark~\ref{rem:plbc} we explain why generically, all open Gromov-Witten invariants can be obtained from point-like bounding chains.

\subsubsection{Gromov-Witten potential}\label{sssec:gwp}
Define a bilinear form on $A^*(X)$ by
\[
\langle\eta,\zeta\rangle_X:=(-1)^{|\zeta|} \int_X \eta\wedge\zeta.
\]
The pairing $\langle\cdot,\cdot\rangle_X$ descends to the Poincar\'e pairing on cohomology, for which we use the same notation. Let $U \subset H^*(X;\R)$ be a linear subspace, and let $\gamma_U \in \mI_U A^*(X;\Qh_U)$ satisfy $d\gamma_U = 0$ and $[\gamma_U] = \Gamma_U.$

Consider the formal function
\begin{equation}\label{eq:Phi}
\Phi_U
=
\sum_{l\ge 0} \frac{1}{(l+1)!}\langle\q_{\emptyset,l}(\gamma_U^{\otimes l}),\gamma_U\rangle_X \in \Qh_U.
\end{equation}
Writing $\gamma_U = \sum_{i \in I_U} t_i\gamma_i,$ we have
\[
\Phi_U = \Phi_U(\{t_i\}_{i \in I_U})
=
\sum_{\substack{\beta\in H_2(X;\Z)\\ r_i\ge 0}}
(-1)^{w_\s(\beta)}\frac{T^{\varpi(\beta)}\prod_{i \in \rI_U}t_i^{r_i}}{\prod_{i \in I_U}r_i!} \GW_{\beta}(\otimes_{i \in I_U}[\gamma_i]^{\otimes r_i}),
\]
where $\GW_\beta(\eta_1,\ldots,\eta_N)$ denotes the closed Gromov-Witten invariant.
The sign $(-1)^{w_\s(\beta)}$ compensates for the sign in equation~\eqref{eq:qemptyset}.
The gradient of $\Phi_U$ with respect to  $\langle\cdot,\cdot\rangle_X$ is given by
\begin{equation}\label{eq:gradphi}
\nabla\Phi_U(\gamma) :=
\sum_{l\ge 0}\frac{1}{l!}[\q_{\emptyset,l}(\gamma_U^{\otimes l})]
\in H^*(X;\Qh_U).
\end{equation}
It is well-known that $\q_{\emptyset,l}(\gamma_U^{\otimes l})$ is closed and that $\Phi_U$ only depends on the cohomology class of~$\gamma_U$.

\subsubsection{Quantum cohomology}\label{sssec:qprod}
The big quantum product
\[
\star_U : H^*(X;Q_U) \otimes H^*(X;Q_U) \to H^*(X;Q_U)
\]
is given by
\[
[\zeta] \star_U [\eta] \to \sum_{l\ge 0}\frac{1}{l!}[\q_{\emptyset,l+2}(\zeta,\eta,\gamma_U^{\otimes l})]
\]
It is well known to be commutative and associative. See Remark~\ref{rem:clst}.
Moreover, the Poincar\'e pairing makes $(H^*(X;Q_U),\star_U)$ a Frobenius algebra,
\[
\langle\eta \star_U \xi,\zeta\rangle_X = \langle\eta,\xi \star_U \zeta\rangle_X.
\]
We denote this Frobenius algebra by $QH_U(X)$.

\subsection{Results}
\subsubsection{Relative potential}\label{sssec:psi}
The usual superpotential of a Lagrangian submanifold $L\subset X$ does not give invariant counts of $J$-holomorphic disks in $X$ with boundary contractible in $L$ and no boundary constraints. The lack of invariance stems from the possibility of the boundary of such disks collapsing to a point in a $1$ parameter family. In order to formulate the open WDVV equations, we define an enhanced superpotential
that gives invariant counts of all types of $J$-holomorphic disks in $X$ with boundary in $L.$ Invariance is achieved by including certain contributions from $J$-holomorphic spheres that cancel the boundary collapse phenomenon. We begin by defining a relative potential that counts both $J$-holomorphic disks and $J$-holomorphic spheres. The natural home for the relative potential is the following cone complex. Let $W$ be a graded real vector space and consider the map of complexes of $Q_W$ modules
\begin{gather*}
\i:A^*(X;\Qh_W)\lrarr R_W[-n],\\
\eta\mapsto (-1)^{n+|\eta|}\int_Li^*\eta,
\end{gather*}
where $R_W[-n]$ is equipped with the trivial differential. The \textbf{cone} $\cone$ is the complex with underlying graded $Q_W$ module $A^*(X;\Qh_W) \oplus R_W[-n-1]$
and differential
\[
\dcone(\eta,\xi):=(d\eta, \i(\eta)-d\xi)=(d\eta,\i(\eta)).
\]

Let $(\gamma,b) \in \mI_W\Ah^*(X,L;\Qh_W )\oplus \mJ_{W}A^*(L;R_W)$ be a bounding pair and define $\psi(\gamma,b) \in \cone$ by
\begin{equation}\label{eq:psidfn}
\psi(\gamma,b):=
\Big(\sum_{l\ge 0}\frac{1}{l!}\q_{\emptyset,l}(\gamma^{\otimes l}), (-1)^{n+1}\big(\sum_{k,l\ge 0}\frac{1}{l!(k+1)}\langle\qkl(b^{\otimes k};\gamma^{\otimes l}),b\rangle
+\sum_{l\ge 0}\frac{1}{l!}\q_{-1,l}(\gamma^{\otimes l})\big)\Big).
\end{equation}
In Section~\ref{ssec:psi}, we show that $\dcone\psi(\gamma,b) = 0$. Thus, we define
\textbf{the relative potential}
$\Psi(\gamma,b)$ to be the cohomology class of $\psi(\gamma,b).$
Definition~\ref{dfn:gequiv} gives a notion of gauge-equivalence between a bounding pair $(\gamma,b)$ with respect to $J$ and a bounding pair $(\gamma',b')$ with respect to $J'$.
We prove the following.
\begin{thm}\label{thm:psiinv}
If $(\gamma,b)$ is gauge equivalent to $(\gamma',b')$, then $\Psi(\gamma,b)=\Psi(\gamma',b')$.
\end{thm}

\subsubsection{The relative potential and the closed potential}\label{sssec:intropsi}
Let
\begin{equation}\label{eq:rho}
\rho : \Hh^*(X,L;\R) \to H^*(X;\R)
\end{equation}
denote the natural map. Let
\[
U \subset H^*(X;\R), \qquad W = \rho^{-1}(U) \subset \Hh^*(X,L;\R),
\]
and suppose $(\gamma_W,b)$ is a bounding pair over $W.$ Denote by $\rho^* : Q_U \to Q_W$ the induced map of rings and let $\pi : H^*(\cone) \to H^*(X,\Qh_{W})$ be the natural map. We show in Lemma~\ref{lm:phipsi} that
\[
\pi(\Psi(\gamma_W,b)) = \rho^*(\nabla\Phi_U).
\]
That is, the relative potential $\Psi(\gamma_W,b)$ lifts the gradient of the closed Gromov-Witten potential $\nabla\Phi_U$ to the cohomology of the cone complex $H^*(\cone).$

\subsubsection{Enhanced superpotential}\label{sssec:esp}
From the long exact sequence of the cone,
\[
\xymatrix{
H^*(\cone) \ar[rr]^\pi && H^*(X;\Qh_W) \ar[dl]_{[1]}^{\i} \\
 & R_W[-n-1] \ar[ul]^x,
}
\]
we obtain an exact sequence,
\begin{equation}\label{eq:xbar}
0 \lrarr \Coker \i \overset{\bar x}{\lrarr} H^*(\cone) \overset{\pi}{\lrarr} H^*(X,\Qh_{W}).
\end{equation}
If $[L] = 0 \in H_n(X;\R),$ then $\Coker \i \simeq R_W[-n-1].$  Otherwise, $\Coker \i \simeq (R_W/Q_W)[-n-1].$ We choose
\[
P :  H^*(\cone) \lrarr \Coker \i,
\]
a left inverse to the map $\bar x$ from the exact sequence~\eqref{eq:xbar} satisfying natural conditions detailed in Section~\ref{ssec:omegabar}. The choice of $P$ is equivalent to the choice of a left inverse
\begin{equation}\label{eq:iPR}
P_\R : \Hh^*(X,L;\R) \to \Coker \i_\R
\end{equation}
to the map $\bar y : \Coker \i_\R \to \Hh^*(X,L;\R)$ induced by the map $y$ from the long exact sequence
\begin{equation}\label{eq:yrho}
\xymatrix{
\Hh^*(X,L;\R) \ar[rr]^\rho && H^*(X;\R) \ar[dl]^{\i_\R} \\
 & \R[-n] \ar[ul]_{[1]}^y.
}
\end{equation}
If $[L] \neq 0,$ then $\Coker \i_\R = 0,$ so $P$ is unique. A geometric interpretation of $P$ is given in Remark~\ref{rem:Pgeom}.

Define the \textbf{enhanced superpotential} $\Ob(\gamma,b) \in \Coker\i$ by
\begin{equation}\label{eq:obdef}
\Ob(\gamma,b) := P\Psi(\gamma,b).
\end{equation}
Define the \textbf{superpotential} $\Omega(\gamma,b) \in R_W$ as follows. Let $D : R_W \to Q_W$ denote the unique $Q_W$ module homomorphism such that $D|_{Q_W} = \Id,$ $D(s^k) = 0$ for $k \in \Z_{\ge 1}$ and $D(T^\beta) = 0$ for $\beta \notin \im \varpi.$ Let $q : R_W \to R_W/Q_W$ denote the quotient map, and let $\bar q : \Coker \i \to R_W/Q_W$ be the induced map. Then $\Omega \in R_W$ is the unique element such that $q(\Omega) = \bar q(\Ob)$ and $D(\Omega) = 0.$ The following is immediate from Theorem~\ref{thm:psiinv}.
\begin{cor}
The enhanced superpotential $\Ob$ and the superpotential $\Omega$ are invariant under gauge equivalence of bounding pairs.
\end{cor}

Assume now that $(\gamma,b) = (\gamma_W,b)$ is a bounding pair over $W \subset \Hh^*(X,L)$. We define the associated open Gromov-Witten invariants
\[
\ogwb_{\beta,k} : W^{\otimes l} \to \R, \qquad l,k \in \Z_{\ge 0}, \quad \beta \in \sly,
\]
to be the coefficients of the series expansion of $\Ob(\gamma_W,b).$ More explicitly, write $\gamma_W = \sum_{i \in I_W} t_i \gamma_i.$ Then, the invariants $\ogwb$ are defined by the equation in $\Coker \i,$
\begin{equation}\label{eq:Ob}
\Ob(\gamma_W,b) = \Ob(s,\{t_i\}_{i \in I_W})=\sum_{\substack{\beta\in\sly\\k\ge 0\\r_i\ge 0}}\frac{T^\beta s^k \prod_{i \in \rI_W}t_i^{r_i}}{k! \prod_{i \in I_W}r_i!} \ogwb_{\beta,k}(\otimes_{i \in I_W}[\gamma_i]^{\otimes r_i}).
\end{equation}
Thus, if $[L] \neq 0 \in H_n(X;\R),$ the invariants $\ogwb_{\beta,k}(\gamma_{i_1},\ldots,\gamma_{i_l})$
are undefined when $k = 0$ and $\beta\in \im \varpi.$ Indeed, $\Ob(\gamma_W,b) \in \Coker \i \simeq (R_W/Q_W)[-n-1],$ and $T^\beta \prod_{i \in \rI_W}t_i^{r_i} \in Q_W$ when $\beta \in \im\varpi.$

Define open Gromov-Witten invariants $\ogw_{\beta,k} : W^{\otimes l} \to \R$ for $l,k \in \Z_{\geq 0}, \beta \in \sly,$ by
\begin{equation}\label{eq:ogw}
\ogw_{\beta,k}(\eta_1,\ldots,\eta_l) =
\begin{cases}
0, & k = 0 \text{ and } \beta \in \im\varpi, \\
\ogwb_{\beta,k}(\eta_1,\ldots,\eta_l), & \text{otherwise.}
\end{cases}
\end{equation}
So, equation~\eqref{eq:Ob} holds in $R_W$ with $\Ob$ replaced by $\Omega$ and $\ogwb$ replaced by $\ogw.$ Lemma~\ref{lm:OP} shows that $\Omega$ and therefore the invariants $\ogw$ are independent of the choice of $P.$ Thus, the choice of $P$ only influences the invariants $\ogwb_{\beta,k}$ for $k = 0$ and $\beta \in \Im \varpi.$
\begin{rem}
Lemma~\ref{lm:spc} shows the superpotential $\Omega$ coincides with superpotential defined in \cite{ST2}.
For $(\gamma,b)$ a bounding pair of the type considered in~\cite{ST2}, the open Gromov-Witten invariants $\ogw$ defined here coincide with those defined there.
\end{rem}

\subsubsection{Vanishing theorem}
In order to formulate the open WDVV equations, we need the following property of the element $c \in R_W$ associated to a bounding pair $(\gamma,b) \in \mI_W\Ah^*(X,L;\Qh_W )\oplus \mJ_{W}A^*(L;R_W)$ by the Maurer-Cartan equation~\eqref{eq:bdchdfn}.
Let $PD([L])\in H^n(X;\R)$ denote the Poincar\'e dual to the fundamental class $[L] \in H_n(X;\R).$
\begin{thm}\label{prop:formforc}
Suppose $[L]\ne 0$ and let $\eta\in H^*(X;\R)$ such that $\int_L\eta=1$. Then
\[
c= \sum_{\substack{\beta \in H_2(X,\Z) \\ l\ge 0}}(-1)^{n+1+w_\s(\beta)}\frac{T^{\varpi(\beta)}}{l!} \GW_\beta(\eta, PD([L]), [\gamma]^{\otimes l}).
\]
In particular, $c\in \Qh_W$.
\end{thm}
\begin{cor}
The product $c \,\Ob \in \Coker \i$ is well-defined.
\end{cor}

Theorem~\ref{prop:formforc} has the following consequences for open Gromov-Witten invariants when $[L] \neq 0.$ See Lemma~\ref{lm:dsomega}. Recall the map $\rho$ from~\eqref{eq:rho}.

\begin{cor}\label{cor:van}
Suppose that $[L] \neq 0$
and $(\gamma_W,b)$ is a bounding pair over $W \subset \Hh^*(X,L)$ with $b$ point-like. Let $\eta_1,\ldots,\eta_l \in W.$ If $k\ge 2$, or if $k\ge 1$ and $\beta \notin \Im \varpi$, then
\[
\ogwb_{\beta,k}(\eta_1,\ldots,\eta_l)=0.
\]
Moreover, for $\beta\in \Im(\varpi)$ and $\eta \in H^*(X;\R)$ such that $\int_L \eta = 1,$
\[
\ogwb_{\beta,1} (\eta_1,\ldots\eta_l)
=
\sum_{\substack{\hat\beta \in H_2(X;\Z)\\  \varpi(\hat\beta)=\beta}}
(-1)^{n+1+w_\s(\hat\beta)}
\GW_{\hat\beta}(\eta,PD([L]),\rho(\eta_1), \ldots,\rho(\eta_l)).
\]
\end{cor}
Theorem~\ref{prop:formforc} does not a priori give information about the invariants $\ogwb_{\beta,0}(\eta_1,\ldots,\eta_l).$

\subsubsection{Open WDVV equations}\label{sssec:owdvv}
Recall the map $\rho$ from~\eqref{eq:rho}.
To formulate the open WDVV equations, we need the following two assumptions:
\begin{enumerate}[leftmargin=2cm,labelsep=.5cm,align=left,label=\text{(A.\arabic*)}]\setcounter{enumi}{0}
\item\label{assump:sfa}
$U \subset H^*(X;\R)$ is a subspace such that $U\otimes Q_U \subset QH_U(X)$ is a Frobenius subalgebra.
\item\label{assump:b}
$(\gamma_W,b)$ is a bounding pair over $W = \rho^{-1}(U)\subset \Hh^*(X,L;\R)$ with $b$ separated.
\end{enumerate}
More explicitly, assumption~\ref{assump:sfa} means that $U$ is a subalgebra with respect to the big quantum product $\star_U$, and the restriction of the Poincar\'e pairing to $U$ is non-degenerate. All point-like bounding chains satisfy assumption~\ref{assump:b}.
Both assumptions are satisfied in the cases discussed in~\cite{ST2} as explained in Section~\ref{ssec:examples} below. The map $P_\R$ from~\eqref{eq:iPR} determines a complement $W' = \Ker(P_\R|_W) \subseteq W$ to the image of the map $y$ from the exact sequence~\eqref{eq:yrho}. In particular, $\rho|_{W'}$ is injective.
Choose index sets $I_{W'} \subseteq I_U$, a basis $\Delta_i \in U, i \in I_U,$ and a basis $\Gamma_i \in W', i \in I_{W'},$ such that $\rho(\Gamma_i) = \Delta_i.$ By abuse of notation, denote by $\partial_i : Q_U \to Q_U$ (resp. $\partial_i : R_W \to R_W$) the derivations corresponding to $\Delta_i, i \in I_U$ (resp. $\Gamma_i, i \in I_{W'}$).
Let
\[
g_{ij}:=\int_X \Delta_i\cup \Delta_j, \qquad i,j \in I_U,
\]
and let $(g^{ij})$ be the inverse matrix to $(g_{ij})$, which exists by assumption~\ref{assump:sfa}. Abbreviate $\Phi = \Phi_U \in Q_U$ and $\Ob = \Ob(\gamma_W,b) \in R_W.$ Let $\rho^* : Q_U \to Q_W$ denote the induced ring homomorphism as in Section~\ref{sssec:intropsi}.
We are now ready to formulate the open WDVV equations.
\begin{thm}[Open WDVV equations]\label{thm:OWDVV}\label{2nd}
Let $c$ be the coefficient of the Maurer-Cartan equation~\eqref{eq:bdchdfn} for the bounding pair $(\gamma_W,b),$ and let $u,v \in W \oplus S, \, w \in W.$
Let $u_W,v_W,$ denote the projections of $u,v,$ to $W,$ and let $\bar{w} = \rho(w), \bar{u} = \rho(u_W), \bar{v} = \rho(v_W).$ Then,
\begin{multline}\label{1st}
\sum_{l \in I_{W'}, m \in I_U}\d_u\d_l\Ob\cdot g^{lm}\cdot
\rho^*\d_m\d_{\bar w}\d_{\bar v}\Phi
-\d_u c\cdot\d_w\d_v\Ob=\\
=
\sum_{l \in I_U, m \in I_{W'}}\rho^*\d_{\bar u}\d_{\bar w}\d_l\Phi \cdot g^{lm}\cdot
\d_m\d_v\Ob
-\d_u\d_w\Ob\cdot\d_vc.
\end{multline}
\end{thm}

\begin{cor}\label{cor:owdvv}
Suppose $b$ is point-like and let $u,v,w \in W.$ Let $\bar{u} = \rho(u),\bar{v} = \rho(v),\bar{w} =\rho(w).$ If $[L]=0,$ then
\begin{multline}\label{eq:cor1}
\sum_{l \in I_{W'}, m \in I_U}\d_u\d_l\Ob\cdot g^{lm}\cdot
\rho^*\d_m\d_{\bar w}\d_{\bar v}\Phi
-\d_{u}\d_s\Ob\cdot\d_w\d_v\Ob=\\
=
\sum_{l \in I_U, m \in I_{W'}}\rho^*\d_{\bar u}\d_{\bar w}\d_l\Phi \cdot g^{lm}\cdot
\d_m\d_v\Ob
-\d_u\d_w\Ob\cdot\d_{v}\d_s\Ob,
\end{multline}
and
\begin{equation}\label{eq:cor2}
\sum_{l \in I_{W'}, m \in I_U}\d_s\d_l\Ob\cdot g^{lm}\cdot
\rho^*\d_m\d_{\bar w}\d_{\bar v}\Phi
-\d_s^2\Ob\cdot\d_w\d_v\Ob=
-\d_s\d_w\Ob\cdot\d_{v}\d_s\Ob.
\end{equation}
If $[L]\ne 0,$ then equation~\eqref{eq:cor1} holds with $\Ob$ replaced by $\Omega.$
\end{cor}

\subsubsection{The tensor potential and the relative quantum connection}
To prove Theorem~\ref{thm:OWDVV}, we construct an invariant $\memh \in End(\cone)$ called the tensor potential. The tensor potential $\memh$ is closely related to the total derivative of the relative potential $\psi.$ The derivative of the tensor potential is the connection $1$-form of the relative quantum connection.

In greater detail, let $(\gamma,b) \in \mI_W\Ah^*(X,L;\Qh_W )\oplus \mJ_{W}A^*(L;R_W)$ be a bounding pair. We define \textbf{the tensor potential}
\footnote{Read as ``noon.''}
\[
\memh = \memh^{\gamma,b}: \cone \lrarr \cone
\]
by
\begin{multline}\label{eq:memhdfn1}
\memh(\eta,\xi)=
\Big(\sum_{l\ge 0}\frac{1}{l!}\q_{\emptyset,l+1}(\eta\otimes\gamma^{\otimes l}), (-1)^{n+1}\big(\sum_{k,l\ge 0}\frac{1}{l!(k+1)}\langle\q_{k,l+1}(b^{\otimes k};\eta\otimes\gamma^{\otimes l}),b\rangle+\\
+\sum_{l\ge 0}\frac{1}{l!}\q_{-1,l+1}(\eta\otimes\gamma^{\otimes l})\big)
-c\cdot \xi\Big).
\end{multline}
\begin{thm}\label{thm:mp}
The tensor potential $\memh$ is a chain map. If the bounding pairs $(\gamma,b)$ and $(\gamma',b')$ are gauge equivalent, then the tensor potentials $\memh^{\gamma,b}$ and $\memh^{\gamma',b'}$ are chain homotopic.
\end{thm}
The open WDVV equations and in fact the closed WDVV equations as well are a consequence of the following theorem. The notation $\d_u\memh$ is explained in detail in Section~\ref{ssec:flatness}.
\begin{thm}\label{lm:assoc}
For all formal vector fields $u,v \in Q_W \otimes W \oplus R_W \otimes S,$ the composition
$\d_u\memh\circ \d_v\memh$ is chain homotopic to $(-1)^{|u||v|}\d_v\memh\circ\d_u\memh.$
\end{thm}
Since $H^*(\cone)$ is a free $Q_W$ module, it can be viewed as the formal sections of a vector bundle over $W.$ The tensor potential $\memh$ induces a map $\nund : H^*(\cone) \to H^*(\cone).$ We define the \textbf{relative quantum connection} $\nabla$ on $H^*(\cone)$ by
\[
\nabla_u(\Upsilon) = \d_u \Upsilon + \d_u\nund(\Upsilon), \qquad u \in Q_W \otimes W, \quad \Upsilon \in H^*(\cone).
\]
A straightforward calculation using Theorem~\ref{lm:assoc} gives the following.
\begin{cor}\label{cy:flat}
The relative quantum connection is flat.
\end{cor}

\subsubsection{Wall-crossing formula}\label{sssec:wallcross}
Recall the long exact sequence~\eqref{eq:yrho}. Let $U \subset H^*(X;\R)$ be a subspace, and let $(\gamma_W,b)$ be a bounding pair over $W = \rho^{-1}(U) \subset \Hh^*(X,L;\R)$. Let
\[
\Gamma_\diamond := y(1) \in W.
\]
The following result relates open Gromov-Witten invariants with boundary constraints to open Gromov-Witten invariants with interior constraints $\Gamma_\diamond$.
\begin{thm}[Wall-crossing]\label{prop:wallcross}
Suppose $[L] = 0$ and $b$ is point-like. The invariants $\ogwb$ satisfy
\[
\ogwb_{\beta,k+1}(\eta_1,\ldots,\eta_l)
=
-\ogwb_{\beta,k}(\Gamma_\diamond,\eta_1,\ldots,\eta_l).
\]
\end{thm}

Geometrically, we can understand Theorem~\ref{prop:wallcross} as follows. Poincar\'e-Lefschetz duality gives an isomorphism
\[
H_{n-1}(X \setminus L;\R) \simeq H^{n+1}(X,L;\R) \simeq \Hh^{n+1}(X,L;\R).
\]
Under this isomorphism, $\Gamma_\diamond \in \Hh^{n+1}(X,L;\R)$ corresponds to the class in $H_{n-1}(X \setminus L;\R)$ of a small $(n-1)$-dimensional sphere $\Sigma$ linked with $L$. As we shrink $\Sigma,$ it converges to a point in $L.$ Thus, at an intuitive level, it makes sense that the interior constraint $\g_\diamond$ can be swapped with a point boundary constraint.
To see why Theorem~\ref{prop:wallcross} is called the wall-crossing formula, consider the following scenario. Let $M$ be an $(n-1)$-dimensional manifold, and let $g : [0,1] \times M \to X$ be a map transverse to $L$ such that $g^{-1}(L)$ is a single point in $(t_L,m_L) \in (0,1) \times M.$ Let $g^t : M \to X$ be given by $g^t(p) = g(t,p).$
For $i = 0,1,$ let $\eta^i_0 \in \Hh^{n+1}(X,L;\R)$ be the class corresponding to $g^i_*([M]) \in H_{n-1}(X\setminus L, \R)$ under Poincar\'e-Lefschetz duality. Then, it is easy to see that $\eta_0^1 - \eta_0^0 = \Gamma_\diamond.$ So,
\[
\ogwb_{\beta,k}(\eta_0^1,\eta_1,\ldots,\eta_l) - \ogwb_{\beta,k}(\eta_0^0,\eta_1,\ldots,\eta_l) = \ogwb_{\beta,k}(\Gamma_\diamond,\eta_1,\ldots,\eta_l).
\]
On the other hand, roughly speaking, the invariant $\ogwb_{\beta,k}(\eta_0^i,\eta_1,\ldots,\eta_l)$ counts disks of degree $\beta$ with boundary constrained to pass through $k$ points in $L$ and interior constrained to pass through $g^i(M)$ as well as Poincar\'e duals of $\eta_1,\ldots,\eta_l.$ To compare the invariants for $i = 0,1,$ consider the one dimensional family of interior constraints $g^t(M).$ At times $t \neq t_L,$ the invariant is constant. At time $t_L,$ the interior constraint $g^{t_L}(M)$ intersects $L$ at the unique point $g_{t_L}(m_L).$ As $t \to t_L$ from below, the interior intersection points with $g_{t}(M)$ of a subset $\mathcal B$ of the disks being counted limit to a boundary point. These disks are no longer counted for $t > t_L.$ So, the number of disks in $\mathcal B$ is $\ogwb_{\beta,k}(\eta_0^1,\eta_1,\ldots,\eta_l) - \ogwb_{\beta,k}(\eta_0^0,\eta_1,\ldots,\eta_l).$ On the other hand, at $t = t_L,$ the boundaries of the disks in $\mathcal B$ pass through $k+1$ points in $L$, one of which is $g^{t_L}(m_L),$ and the interiors of these disks pass through Poincar\'e duals to $\eta_1,\ldots,\eta_l.$ So, the number of disks in $\mathcal B$ is $\ogwb_{\beta,k+1}(\eta_1,\ldots,\eta_l).$

\subsubsection{Relative quantum cohomology}\label{sssec:memintro}
Suppose again that assumptions~\ref{assump:sfa} and~\ref{assump:b} hold. Define $\Qhh_W = \Qh_W \otimes_{\Lc} \L$ and
\[
\coneq:=(\cone/x(m_S))\otimes_{\Qh_W}\Qhh_W.
\]
The map $\memh : \cone \to \cone$ induces a map $\hat \memh : \coneq \to \coneq,$ which in turn induces a map $\nundh : H^*(\coneq) \to H^*(\coneq).$ Since $\coneq$ is the cone of the map $\i : A^*(X;\Qhh_W)\to \Qhh_W$, we have a canonical isomorphism
\[
\Hh^*(X,L;\Qhh_W)\simeq H^*(\coneq).
\]
We define the relative quantum cohomology as a vector space by $QH_U(X,L) = W \otimes \Qhh_W,$ and we define the relative quantum product
\footnote{Read as ``mem''.}
\[
\mem: QH_U(X,L)\otimes QH_U(X,L)\lrarr QH_U(X,L)
\]
by
\[
\mem(w,v)=\d_w\nundh(v).
\]
A more explicit formula for $\mem$ is given in Lemma~\ref{lm:memform}.
\begin{thm}\label{thm:mem}
The relative quantum product $\mem$ is graded commutative, associative, and
depends only on the gauge-equivalence class of the bounding pair $(\gamma_W,b)$.
\end{thm}

\begin{rem}
Theorem~\ref{thm:mem} is a consequence of Theorem~\ref{lm:assoc} for $u,v \in Q_W \otimes W.$
On the face of it, the relative quantum cohomology algebra $QH_U(X,L)$ contains no information about open Gromov-Witten invariants with boundary constraints. Indeed, in the definition of $QH_U(X,L),$ we have quotiented by $m_S,$ the ideal generated by the parameters that keep track of boundary constraints. However, by Corollary~\ref{cor:van}, when $[L] \neq 0 \in H_n(X;\R),$ open Gromov-Witten invariants with at least one boundary constraint contain no information beyond closed Gromov-Witten invariants. By the wall-crossing formula of Theorem~\ref{prop:wallcross}, when $[L] = 0,$ open Gromov-Witten invariants with boundary constraints are equivalent to open Gromov-Witten invariants with interior constraints. Thus, the associativity of $\mem$ is essentially equivalent to Theorem~\ref{lm:assoc}.
\end{rem}

\subsubsection{Examples}\label{ssec:examples}
In this section, we give examples where assumptions~\ref{assump:sfa} and~\ref{assump:b} hold.
The following is an immediate consequence of Theorem 2 of~\cite{ST2}. See Section 5.4 of~\cite{ST2} for further details.
\begin{thm}\label{thm:esphere}
Suppose $H^*(L;\R) \simeq H^*(S^n,\R)$. Let $U = H^*(X)$, so $W = \rho^{-1}(U) =  \Hh^*(X,L).$ Then, there exists a unique up to gauge equivalence bounding pair $(\gamma_W,b)$ over $W$ with $b$ point-like. In particular, assumptions~\ref{assump:sfa} and~\ref{assump:b} hold.
\end{thm}
The open Gromov-Witten invariants $\ogw$ associated with such $(\gamma_W,b)$ coincide with those of~\cite{ST2}.

A \textbf{real setting} is a quadruple $(X,L,\omega,\phi)$ where $\phi:X\to X$ is an anti-symplectic involution such that $L\subset \fix(\phi)$. Whenever we discuss a real setting, we fix a connected subset $\J_\phi \subset \J$
consisting of $J \in \J$ such that $\phi^*J = -J.$ All almost complex structures of a real setting are taken from $\J_\phi.$ If we use virtual fundamental class techniques, we can treat any $\omega$-tame almost complex structure $J$ satisfying $\phi^*J = -J.$ In addition, whenever we discuss a real setting, we take $S_L \subset  H_2(X,L;\Z)$ with $\Im(\Id+\phi_*) \subset S_L,$ so $\phi_*$ acts on $\sly_L = H_2(X,L;Z)/S_L$ as $-\Id.$ Also, the formal variables $t_j$ have even degree. Given a real setting, let $H^{even}_\phi(X)$ (resp. $\Hh^{even}_\phi(X,L)$) denote the direct sum over $k$ of the $(-1)^{k}$-eigenspaces of $\phi^*$ acting on $H^{2k}(X;\R)$ (resp. $\Hh^{2k}(X,L;\R)$).
Extend the action of $\phi^*$ to $\L,Q,R,C,$ and $D,$ by taking
\begin{equation}\label{eq:phi*ext}
\phi^*T^\beta = (-1)^{\mu(\beta)/2}T^\beta, \qquad \phi^* t_i = (-1)^{\deg t_i/2}t_i, \qquad \phi^* s = -s.
\end{equation}
Elements $a \in \L,Q,R,C,D,$ and pairs thereof are called \textbf{real} if
\begin{equation}\label{eq:relt}
\phi^* a = -a.
\end{equation}
The main ingredient in the following is Theorem 3 of~\cite{ST2}. See Appendix~\ref{app:a} for details.
\begin{thm}\label{thm:cases}
Suppose $(X,L,\omega,\phi)$ is a real setting,
$\s$ is a spin structure,
and $n \not \equiv 1 \pmod 4.$ Moreover,
\begin{itemize}
\item
if $n \equiv 3 \pmod 4,$ assume $H^i(L;\R) \simeq H^i(S^n;\R)$ for $i \equiv 0,3 \pmod 4$;
\item
if $n \equiv 2 \pmod 4,$ assume $H^i(L;\R) \simeq H^i(S^n;\R)$ for $i \not \equiv 1 \pmod 4$;
\item
if $n \equiv 0 \pmod 4,$ assume $H^i(L;\R) \simeq H^i(S^n;\R)$ for $i \not \equiv 2 \pmod 4$.
\end{itemize}
Let $U = H^{even}_\phi(X)\subset H^*(X)$. Then,
\begin{enumerate}
\item\label{it:realthma}
$U\otimes Q_U \subset QH_U(X)$ is a Frobenius subalgebra.
\item\label{it:realthmb}
$W = \rho^{-1}(U) = \Hh^{even}_\phi(X,L)$.
\item\label{it:realthmc}
There exists a unique up to gauge equivalence real bounding pair $(\gamma_W,b)$ over $W$ such that $b$ is point-like.
\end{enumerate}
In particular, assumptions~\ref{assump:sfa} and~\ref{assump:b} hold.
\end{thm}
The open Gromov-Witten invariants $\ogw$ associated with such $(\gamma_W,b)$ again coincide with those of~\cite{ST2}.

\subsubsection{Special case: projective space}\label{sssec:cpn}
Consider the special case $(X,L)=(\P^n,\RP^n)$ with $\omega=\oFS$ the Fubini-Study form, $J=J_0$ the standard complex structure, and $n$ odd. We normalize $\omega$ so that $\int_{\C P^1} \omega = 1.$ Take $\sly = H_2(X,L;\Z)$ and identify $\sly$ with $\Z$ in such a way that $\beta \in \sly$ with $\omega(\beta)\geq 0$ are identified with non-negative integers. Similarly, identify $H_2(X;\Z)$ with $\Z$ in such a way that $\hat\beta \in H_2(X;\Z)$ with $\omega(\hat \beta) \geq 0$ are identified with non-negative integers. So, the map $\varpi : H_2(X;\Z) \to \sly$ is given by multiplication by $2.$
By Theorem~\ref{thm:esphere}, our results apply with $W = \Hh^*(X,L;\R).$
Write $\g_j = [\omega^j] \in \Hh^*(X,L;\R)$ and $\ga_j = [\omega^j] \in H^*(X;\R).$ We take the map $P_\R$ of~\eqref{eq:iPR} to be the unique left inverse of $\bar y$ such that $W' = \Ker(P_\R|_W) = \spa\{\g_j\}_{j= 0}^n$ and this determines the map $P.$ The following is a consequence of Corollary~\ref{cor:owdvv}, the Kontsevich-Manin axioms~\cite{KontsevichManin}, and analogous axioms for open Gromov-Witten invariants given by Proposition~\ref{axioms}. The proof is given in Section~\ref{comp}.
\begin{thm}\label{recursion}\label{thm:recursion}
The invariants $\ogwb_{\beta,k}$ satisfy the following recursions.
\begin{enumerate}
	\item\label{reduce_smallest} 
	Let $l\ge 2$ and
let $I:=\{j_3,\cdots, j_l\}$ (possibly $I=\emptyset$).
Then
	\begin{align*}
	\ogwb_{\beta,k}&(\g_{j_1},\ldots,\g_{j_l})=
	\ogwb_{\beta,k}(\g_{j_1-1},\g_{j_2+1},\g_{j_3},\ldots,\g_{j_l})+\\
	&+\sum_{\substack{\varpi(\hat\beta)+\beta_1=\beta\\I_1\sqcup I_2= I}}\sum_{i=0}^n(-1)^{\frac{(n+1)\hat\beta}{2}}\cdot \Big(\GW_{\hat\beta}(\ga_1,\ga_{j_2},\ga_{I_1},\ga_i)\ogwb_{\beta_1,k}(\g_{n-i},\g_{j_1-1},\g_{I_2})-\\ 
	&\hspace{9em}-\GW_{\hat\beta}(\ga_1,\ga_{j_1-1},\ga_{I_1},\ga_i)\ogwb_{\beta_1,k}(\g_{n-i},\g_{j_2},\g_{I_2}) \Big)+\\ 
&+\sum_{\substack{\beta_1+\beta_2=\beta\\k_1+k_2=k\\I_1\sqcup I_2=I}}\binom{k}{k_1}\Big(
\ogwb_{\beta_1,{k_1}}(\g_1,\g_{j_1-1},\g_{I_1}) \ogwb_{\beta_2,{k_2+1}}(\g_{j_2},\g_{I_2})-\\
&\quad-
\ogwb_{\beta_1,{k_1}}(\g_1,\g_{j_2},\g_{I_1}) \ogwb_{\beta_2,{k_2+1}}(\g_{j_1-1},\g_{I_2})
\Big).
	\end{align*}
	\item\label{reduce_bd} 
	Let $k\ge 2$ and let $I:=\{j_1,\ldots,j_l\}$. Then
	\begin{align*}
(-1)^{\frac{n+3}{2}}&
	\ogwb_{\beta,k}(\g_{j_1},\ldots,\g_{j_l})=\\
& =
2\sum_{\substack{\varpi(\hat\beta)+\beta_1=\beta+1\\I_1\sqcup I_2=I}}\sum_{i=0}^n (-1)^{\frac{(n+1)\hat\beta}{2}} \GW_{\hat\beta}(\ga_1,\ga_{n},\ga_{I_1},\ga_i) \ogwb_{\beta_1,{k-1}}(\g_{n-i},\g_{I_2})-\\ 
	&-2\sum_{\substack{\beta_1+\beta_2=\beta+1\\ 2\le\beta_1\le \beta\\ k_1+k_2=k-2\\I_1\sqcup I_2=I}}\binom{k-2}{k_1} \ogwb_{\beta_1,{k_1}}(\g_1,\g_{n},\g_{I_1}) \ogwb_{\beta_2,{k_2+2}}(\g_{I_2})+\\
	&+ 2\sum_{\substack{\beta_1+\beta_2=\beta+1\\ 1\le\beta_1\le \beta\\ k_1+k_2=k-2\\I_1\sqcup I_2=I}}\binom{k-2}{k_1}\ogwb_{\beta_1,{k_1+1}}(\g_1,\g_{I_1}) \ogwb_{\beta_2,{k_2+1}}(\g_{n},\g_{I_2}).
	\end{align*}
\end{enumerate}
From the definition, one computes $\ogw_{1,2}= 2$ for an appropriately chosen relative spin structure.
It then follows from the open WDVV equations that
\[
\ogw_{1,1}(\g_{\frac{n+1}{2}})=0,\quad
\ogw_{1,0}(\g_n)=(-1)^{\frac{n+3}{2}}.
\]
\end{thm}

\begin{cor}\label{cor:computable}
The open Gromov-Witten invariants of $(\P^n,\RP^n)$ are entirely determined by the open WDVV equations, the axioms of $\ogwb$, the wall-crossing formula Theorem~\ref{prop:wallcross}, the closed Gromov-Witten invariants of $\C P^n$, and $\ogw_{1,2}= 2$.
\end{cor}
Moreover, the recursion process readily implies Corollary~\ref{cor:integer}, which says the invariants are rational numbers with denominator a power of $2$. The denominators arise from the divisor axiom.

\subsubsection{Sample values for projective space}\label{ssec:sample}
We continue with the setting and notation of the preceding section. Below, we write $\ogw^n_{\beta,k}$ for invariants of $(\P^n, \R P^n).$

\begin{table}[ht]
\centering
\begin{tabular}{|l| l| l| c}
    $n$ & $\beta$ & $k$  & $\ogw_{\beta,k}^n$ \\
\hline
        $3$ & $3$ & $6$ & $-2$\\
    $3$ & $5$ & $10$ & $90$\\
    $3$ & $7$ & $14$ & $-29178$\\
		    $5$ & $5$ & $8$ & $2$\\
    $5$ & $9$ & $14$ & $1974$\\
    $5$ & $13$ & $20$ & $42781410$\\
		$5$ & $17$ & $26$ & $7024726794150$\\
    $7$ & $7$ & $10$ & $-2$\\
    $7$ & $13$ & $18$ & $35498$\\
    $7$ & $19$ & $26$ & $-40083246650$\\
		$7$ & $25$ & $34$ & $680022893749060370$\\
		$9$ & $9$ & $12$  & $2$\\
		$9$ & $17$ & $22$ & $587334$\\
		$9$ & $25$ & $32$ & $31424766229890$\\
		$9$ & $33$ & $42$ & $49920592599715322910150$\\
		$15$ & $15$ & $18$ & $-2$\\
		$15$ & $29$ & $34$ & $2247512778$
\end{tabular}
\caption{Sample values with only boundary constraints}\label{values}
\end{table}

The invariants $\ogw^3_{\beta,k}$ coincide with the analogous invariants of Welschinger~\cite{Welschinger2} up to a factor of $\pm 2^{1-l}$ by Theorem~5 of~\cite{ST2}. We have verified this for small values of $n,l,\beta,$ by comparing the tables in~\cite{BrugalleMikhalkin1,BrugalleGeorgieva} with computer calculations based on Theorem~\ref{thm:recursion}.

On the other hand, we are not aware of a definition of open Gromov-Witten invariants generalizing Welschinger's invariants with $k > 0$ real point constraints in dimensions $n > 3$ besides the invariants $\ogw_{\beta,k}$. In Table~\ref{values}, we present the results of computer calculations based on Theorem~\ref{thm:recursion}, which show these invariants are non-trivial.

For $i_1,\ldots,i_l,$ odd, the invariants $\ogw^n_{\beta,0}(\Gamma_{i_1},\ldots,\Gamma_{i_l})$ coincide with the analogous invariants of Georgieva~\cite{Georgieva} up to a factor of $2^{1-l}$ by Theorem~6 of~\cite{ST2}. We have verified this for small values of $n,l,\beta,$ by comparing the tables in~\cite{GZ0} with computer calculations based on Theorem~\ref{thm:recursion}.

On the other hand, if one or more of $i_1,\ldots,i_l$ is even, the invariants of~\cite{Georgieva} vanish, while the invariants $\ogw^n_{\beta,0}(\Gamma_{i_1},\ldots,\Gamma_{i_l})$ are often non-vanishing. See Tables~\ref{table:dim5} and~\ref{table:dim7}, which present results of computer computations based on Theorem~\ref{thm:recursion}.

\begin{table}[ht]
\centering
\begin{tabular}{l|c c c c c}
\backslashbox{$l_2$}{$\beta$}& 1 & 3 & 5 & 7 & 9 \\
\hline
0 & $\frac{1}{8}$ & $-\frac{43515}{512}$ & $\frac{601224741985}{32768}$ & $-\frac{116238642273889476915}{2097152}$ & $\frac{140294698313130485254672005681}{134217728}$ \\
1 & $\frac{1}{2}$ & $-\frac{255}{32}$ & $\frac{602002259}{1024}$ & $-\frac{113202029455499631}{131072}$ & $\frac{20113658490667274313737811}{2097152}$ \\
2 & 0 & $-\frac{11}{32}$ & $\frac{42354213}{2048}$ & $-\frac{1929681317485627}{131072}$ & $\frac{797025207129184980573997}{8388608}$ \\
3 & 0 & $\frac{3}{8}$ & $\frac{88007}{128}$ &$-\frac{136049411385}{512}$ & $\frac{526360518604567156539}{524288}$ \\
\end{tabular}
\caption{Values of $\ogw^5_{\beta,0}(\g_2^{\otimes l_1}\otimes \g_4^{\otimes l_2})$. The value of $l_1$ is determined by $\beta$ and $l_2$.}\label{table:dim5}
\end{table}

\begin{table}[ht]
\centering
\begin{tabular}{l|c c c c c}
\backslashbox{$l_2$}{$\beta$}& 1 & 3 & 5 & 7 \\
\hline
0 & $-\frac{1}{32}$ & $-\frac{40990173}{8192}$ & $-\frac{679329355023877929}{2097152}$ & $-\frac{304348375415626014616826203269}{536870912}$ \\
1 & $-\frac{1}{2}$ & $-\frac{23229}{512}$ & $-\frac{32941896518265}{131072}$ & $-\frac{3267223276596315843861141}{33554432}$ \\
2 & 0 & $-\frac{11}{32}$ & $-\frac{2874828463}{8192}$ & $-\frac{54161348711499099171}{2097152}$ \\
3 & 0 & 0 & $-\frac{338633}{512}$ & $-\frac{1318018642164857}{131072}$ \\
\end{tabular}
\caption{Values of $\ogw^7_{\beta,0}(\g_2^{\otimes l_1}\otimes \g_6^{\otimes l_2})$. The value of $l_1$ is determined by $\beta$ and $l_2$. Here we chose to take no constraints in $\g_4$.}\label{table:dim7}
\end{table}

The reliance on general bounding chains is the main difference between the invariants $\ogw_{\beta,k}$ and the invariants of Welschinger and Georgieva. In the situations where the invariants $\ogw_{\beta,k}$ coincide with Welschinger's and Georgieva's invariants, bounding chains become explicit: either zero or an $n$-form with integral $s.$ However, the general construction of bounding chains in Theorems~2 and~3 of~\cite{ST2}, upon which we rely in Theorems~\ref{thm:esphere} and~\ref{thm:cases}, uses an inductive argument based on the obstruction theory of~\cite{FOOO}. It is difficult to give an explicit description of the resulting bounding chain. Nonetheless, the results of this paper allow explicit calculations of the invariants~$\ogw_{\beta,k}.$

Write $n = 2r+1.$ To illustrate the geometric significance of the relative cohomology $\Hh^*(X,L)$ and the wall-crossing formula, we consider the real analog of the classical result that there are $r+1$ complex lines in $\P^n$ through $4$ generic complex subspaces of dimension $r$. Real lines correspond to conjugate pairs of holomorphic disks of degree $1.$ When $\beta = 1$, it is not hard to see that the invariants $\ogw_{\beta,k}$ enumerate disks of degree $1$; the bounding chain plays a role only when $\beta > 1.$ Recall that the class $\Delta_{r+1} = [\omega^{r+1}] \in H^{n+1}(\P^n)$ is Poincar\'e dual to the class of an $r$ plane in $H_{n-1}(\P^n).$ However, the class $\Gamma_{r+1} = [\omega^{r+1}] \in \Hh^{n+1}(\P^n,\R P^n)$ is not Poincar\'e dual to the class of an $r$ plane in $H_{n-1}(\C P^n\setminus \R P^n).$ Rather, the classes
\[
\lambda^{\pm} = \Gamma_{r+1} \mp \frac{1}{2}\Gamma_\diamond
\]
are Poincar\'e dual to two distinct classes of $r$ planes in $H_{n-1}(\C P^n\setminus\R P^n).$ We have
\[
\ogw^n_{1,2} = 2, \qquad \ogw^n_{1,1}(\Gamma_{r+1}) = 0, \qquad \ogw_{1,0}^n(\Gamma_{\!{r+1}},\Gamma_{\!{r+1}}) = -\frac{1}{2},
\]
The first two values are stated in Theorem~\ref{recursion} and the third is a consequence of equation~\ref{reduce_smallest} and the divisor axiom. Applying the wall-crossing formula of Theorem~\ref{prop:wallcross}, we obtain
\[
 \ogw^n_{1,0}(\Gamma_\diamond,\Gamma_\diamond) = 2, \qquad \ogw^n_{1,0}(\Gamma_{r+1},\Gamma_\diamond) = 0.
\]
Thus, it follows by multi-linearity that
\begin{equation*}
 \ogw^n_{1,0}(\lambda^\pm, \lambda^\pm) = 0, \qquad \ogw^n_{1,0}(\lambda^\pm,\lambda^\mp) = -1
\end{equation*}
When $n \equiv 1 \pmod 4$ the classes $\lambda^\pm$ are anti-conjugate, so the Poincar\'e dual of a conjugate pair of complex $r$ planes is $2\Gamma_{r+1} = \lambda^+ + \lambda^-.$ Thus, an invariant count of real lines through two conjugate pairs of complex $r$ planes is half of the corresponding disk count:
\[
\frac{1}{2} \ogw_{1,0}(2\Gamma_{r+1},2\Gamma_{r+1}) = -1.
\]
As expected, this agrees with the complex count mod $2.$ When $n \equiv 3 \pmod 4,$ the classes $\lambda^\pm$ are conjugation invariant, so the Poincar\'e dual of a conjugate pair of complex $r$ planes may be either $2\lambda^+$ or $2\lambda^-.$ Thus, there are four possible invariant counts of real lines through two conjugate pairs of complex $r$ planes:
\[
\frac{1}{2}\ogw_{1,0}(2\lambda^\pm,2\lambda^\pm) = 0 \quad\text{and}\quad \frac{1}{2}\ogw_{1,0}(2\lambda^\pm,2\lambda^\mp) = -2.
\]
Again, these invariants agree with the complex count mod $2.$
In~\cite[Example 12]{Kollar}, Koll\'ar constructs a generic configuration of two conjugate pairs of complex $r$ planes of the same class, such that there is no real line that intersects them. This shows that the vanishing invariant is optimal for such pairs. However, for conjugate pairs of complex $r$ planes of different classes, we obtain a positive lower bound of $2.$

\subsubsection{Regularity assumptions}\label{sssec:reg}
We proceed with the regularity assumptions set in~\cite{ST1}, namely, that moduli spaces are smooth orbifolds with corners and the evaluation maps at the zero point are proper submersions. To that we add in Section~\ref{sec:geod} the assumption that the zero evaluation maps remain submersions after restricting to a subspace of open stable maps where certain marked points are constrained to lie on a geodesic of the hyperbolic metric of the disk.

In~\cite[Example 1.5 and Remark 1.6]{ST1} we show that the regularity assumptions hold
for homogeneous spaces. The additional assumption concerning moduli spaces of open stable maps with geodesic constraints on marked points holds for homogeneous spaces as well. Indeed, suppose $J$ is integrable and suppose there exists a Lie group $G_X$ that acts transitively on $X$ by $J$-holomorphic diffeomorphisms. Furthermore, suppose there exists a subgroup $G_L \subset G_X$ that preserves $L$ and acts transitively on $L.$ Let $\M_{k,l;a,b}(\beta)\subset\M_{k,l}(\beta)$ be a moduli space with a geodesic constraint, as defined in Section~\ref{sec:geod}. Then $G_L$ acts on $\M_{k,l;a,b}(\beta)$ as well, and the evaluation maps are equivariant.
Since $G_L$ acts transitively on $L$, we see that $evb_0$ remains a submersion after restricting to $\M_{k,l;a,b}(\beta).$

In particular, $(\P^n,\RP^n)$ with the standard symplectic and complex structures, or more generally, Grassmannians, flag varieties and products thereof, satisfy our regularity assumptions.
Using the theory of the virtual fundamental class from ~\cite{Fukaya,Fukaya2,FOOOtoricI,FOOOtoricII,FOOO1} or~\cite{HoferWysockiZehnder,HoferWysockiZehnder1,HoferWysockiZehnder2,HoferWysockiZehnder3,LiWehrheim}, our results extend to general target manifolds.

\subsection{Acknowledgments}
The authors would like to thank M. Abouzaid, D. Auroux, V.~Kharlamov, N.~Sheridan, E.~Shustin, and A. Zernik, for helpful conversations.
The authors were partially supported by ERC starting grant 337560. The first author was partially supported by ISF Grant 569/18.
The second author was partially supported by the Canada Research Chairs Program and by NSF grants Nos. DMS-163852 and DMS-1440140.

\section{Background}

\subsection{Integration properties}
In the following, $M,N$ and $P$ are orbifolds with corners. We follow the conventions of~\cite{ST4} concerning smooth maps and orientations of orbifolds with corners except that here we require by definition that a proper submersion satisfies the additional property of strong smoothness. Let $f: M\to N$ be a proper submersion with fiber dimension $\rdim f =r,$ and let $\Upsilon$ be a graded-commutative algebra over $\R$. Denote by
\[
f_* : A^*(M;\Upsilon) \to A^*(N;\Upsilon)[-r]
\]
the push-forward of forms along $f$, that is, integration over the fiber, as defined in~\cite{ST4}. We will need the following properties of $f_*$ proved in~\cite{ST4}. In particular, property~\ref{normalization} of Proposition~\ref{prop:iof} is immediate from Definition~4.14 of~\cite{ST4} while the remaining properties of $f_*$ are given in Theorem 1.
\begin{prop}\leavevmode\label{prop:iof}
\begin{enumerate}
	\item \label{normalization}
	Let $f:M\to pt$ and $\alpha\in A^m(M)\otimes\Upsilon$. Then
	\[
	f_*\alpha=\begin{cases}
	\int_M\alpha,& m=\dim M,\\
	0,&\text{otherwise}.
	\end{cases}
	\]
	\item\label{prop:pushcomp}
		Let $g: P\to M$, $f:M\to N,$ be proper submersions. Then
		\[
		f_*\circ g_*=(f\circ g)_*.
		\]
	\item\label{prop:pushpull}
		Let $f:M\to N$ be a proper submersion, $\alpha\in A^*(N;\Upsilon),$ $\beta\in A^*(M;\Upsilon)$. Then
		\[
		f_*(f^*\alpha\wedge\beta)=\alpha\wedge f_*\beta.
		\]
	\item\label{prop:pushfiberprod}
		Let
		\[
		\xymatrix{
		{M\times_N P}\ar[r]^{\quad p}\ar[d]^{q}&
        {P}\ar[d]^{g}\\
        {M}\ar[r]^{f}&N
		}
		\]
		be a pull-back diagram of smooth maps, where $g$ is a proper submersion. Let $\alpha\in A^*(P).$ Then
		\[
		q_*p^*\alpha=f^*g_*\alpha.
		\]
\end{enumerate}
\end{prop}

\begin{prop}[Stokes' theorem]\label{stokes}
Let $f:M\to N$ be a proper submersion with $\dim M=s$, and let $\xi\in A^t(M;\Upsilon)$. Then
\[
d (f_*\xi)=f_*(d \xi)+(-1)^{s+t}\big(f\big|_{\partial M}\big)_*\xi,
\]
where $\d M$ is understood as the fiberwise boundary with respect to $f.$
\end{prop}

\begin{rem}
Proposition~\ref{stokes} applied to $f:M\to pt$ yields the classical Stokes' theorem up to a sign,
\begin{equation*}
\int_M d\xi = (-1)^{\dim M+|\xi|+1}\int_{\d M} \xi.
\end{equation*}
The sign arises from the possibly non-trivial grading of the coefficient ring. See~\cite[Remark~2.3]{ST1} for an extended discussion.
\end{rem}

The following result is Lemma~5.4 of~\cite{Solomon2018}.
\begin{lm}\label{lm:pulltopush}
Let $f:M\to M$ be a diffeomorphism and let $\a\in A^*(M)$. Then $f^*\a=(-1)^{sgn(f)}f_*^{-1}\a.$
\end{lm}

\subsection{Open stable maps}
Here, we recall definitions and notations for open stable maps and moduli spaces thereof from~\cite[Section 2.2.1]{ST1}. A $J$-holomorphic genus-$0$ \textbf{open stable map} to $(X,L)$ of degree $\beta \in \sly$ with one boundary component, $k+1$ boundary marked points, and $l$ interior marked points, is a quadruple $(\Sigma, u,\vec{z},\vec{w})$ as follows. The domain $\Sigma$ is a genus-$0$ nodal Riemann surface with boundary consisting of one connected component,
\[
u: (\Sigma,\d\Sigma) \to (X,L)
\]
is a continuous map, $J$-holomorphic on each irreducible component of $\Sigma,$ with
\[
u_*([\Sigma,\partial\Sigma]) = \beta,
\]
and
\[
\vec{z} = (z_0,\ldots,z_k), \qquad \vec{w} = (w_1,\ldots,w_l),
\]
with $z_j \in \partial \Sigma, \, w_j \in int(\Sigma),$ distinct. The labeling of the marked points $z_j$ respects the cyclic order given by the orientation of $\partial \Sigma$ induced by the complex orientation of $\Sigma.$
Stability means that
if $\Sigma_i$ is an irreducible component of $\Sigma$, then either $u|_{\Sigma_i}$ is nonconstant or it satisfies the following requirement: If $\Sigma_i$ is a sphere, the number of marked points and nodal points on $\Sigma_i$ is at least 3; if $\Sigma_i$ is a disk, the number of marked and nodal boundary points plus twice the number of marked and nodal interior points is at least $3$.
An open stable map is called \textbf{irreducible} if its domain consists of a single irreducible component. An isomorphism of open stable maps $(\Sigma,u,\vec{z},\vec{w})$ and $(\Sigma',u',\vec{z}',\vec{w}')$ is a homeomorphism $\theta : \Sigma \to \Sigma'$, biholomorphic on each irreducible component, such that
\[
u = u' \circ \theta, \qquad\qquad  z_j' = \theta(z_j), \quad j = 0,\ldots,k, \qquad w_j' = \theta(w_j), \quad j = 1,\ldots,l.
\]

Denote by $\M_{k+1,l}(\beta) = \M_{k+1,l}(\beta;J)$ the moduli space of $J$-holomorphic genus zero open stable maps to $(X,L)$ of degree $\beta$ with one boundary component, $k+1$ boundary marked points, and $l$ internal marked points.
Denote by
\begin{gather*}
evb_j^\beta:\M_{k+1,l}(\beta)\to L, \qquad  \qquad j=0,\ldots,k, \\
evi_j^\beta:\M_{k+1,l}(\beta) \to X, \qquad \qquad j=1,\ldots,l,
\end{gather*}
the evaluation maps given by $evb_j^\beta((\Sigma,u,\vec{z},\vec{w}))=u(z_j)$ and $evi_j^\beta((\Sigma,u,\vec{z},\vec{w}))= u(w_j).$
We may omit the superscript $\beta$ when the omission does not create ambiguity.

\subsection{Structure equations and properties}\label{ssec:prop}

For all $\beta\in\sly$, $k,l\ge 0$, and $(k,l,\beta) \not\in\{ (1,0,\beta_0),(0,0,\beta_0)\}$, define
\[
\qkl^\beta:A^*(L;R)^{\otimes k}\otimes A^*(X;Q)^{\otimes l} \lrarr A^*(L;R)
\]
by
\begin{align*}
\q^{\beta}_{k,l}(\alpha_1\otimes\cdots\otimes\alpha_k; \gamma_1\otimes\cdots\otimes\gamma_l):=
(-1)^{\varepsilon(\alpha)}
(evb_0^\beta)_* \left(\bigwedge_{j=1}^l(evi_j^\beta)^*\gamma_j\wedge \bigwedge_{j=1}^k (evb_j^\beta)^*\alpha_j\right)
\end{align*}
with
\[
\varepsilon(\alpha):=\sum_{j=1}^kj(|\alpha_j|+1)+1.
\]
In addition, define $\q_{1,0}^{\beta_0}(\alpha):=d\alpha$ and $\q_{0,0}^{\beta_0}:=0.$
Set
\begin{align*}
\qkl:=\sum_{\beta\in\sly}T^{\beta}\qkl^{\beta}.
\end{align*}

For $l\ge 0$, $(l,\beta)\neq (1,\beta_0),(0,\beta_0),$ define
\[
\q_{-1,l}^\beta:A^*(X;Q)^{\otimes l}\lrarr Q
\]
by
\begin{gather*}
\q_{-1,l}^\beta(\gamma_1\otimes\cdots\otimes\gamma_l)
:= \int_{\M_{0,l}(\beta)} \bigwedge_{j=1}^l (evi_j^\beta)^*\gamma_j,\\
\q_{-1,l}(\gamma_1\otimes\cdots\otimes\gamma_l):=\sum_{\beta\in\sly}T^{\beta} \q_{-1,l}^\beta(\gamma_1\otimes\cdots\otimes\gamma_l).
\end{gather*}
Define also
$
\q_{-1,1}^{\beta_0}:=0$
and $\q_{-1,0}^{\beta_0}:=0$.
Lastly, define similar operations using spheres,
\[
\q_{\emptyset,l}:A^*(X;Q)^{\otimes l}\lrarr A^*(X;R),
\]
as follows. For $\beta\in H_2(X;\Z)$ let $\M_{l+1}(\beta)$ be the moduli space of genus zero $J$-holomorphic stable maps with $l+1$ marked points indexed from 0 to $l,$ representing the class $\beta$. Denote by $ev_j^\beta:\M_{l+1}(\beta)\to X$ the evaluation map at the $j$-th marked point. Assume that all the moduli spaces $\M_{l+1}(\beta)$ are smooth orbifolds and $ev_0$ is a submersion. Let $\pr: H_2(X;\Z) \to \sly$ denote the projection and let $w_\s\in H^2(X;\Z)$ be the class given by the relative spin structure on $L$. For $l\ge 0$, $(l,\beta)\ne (1,0),(0,0)$, set
\[
\q_{\emptyset,l}^\beta(\gamma_1,\ldots,\gamma_l):=
(-1)^{w_\s(\beta)}
(ev_0^\beta)_*(\wedge_{j=1}^l(ev_j^\beta)^*\gamma_j),
\]
and define
\[
\q_{\emptyset,1}^0:= 0,\qquad \q_{\emptyset,0}^0:= 0,
\]
and
\[
\q_{\emptyset,l}(\gamma_1,\ldots,\gamma_l):=
\sum_{\beta\in H_2(X;\Z)}T^{\pr(\beta)}\q_{\emptyset,l}^\beta(\gamma_1,\ldots,\gamma_l).
\]
Denote by $\langle\;,\,\rangle: A^*(L;R) \otimes A^*(L;R) \to R$ the signed Poincar\'e pairing on $L$:
\begin{equation}\label{eq:pairing}
\langle\xi,\eta\rangle=(-1)^{|\eta|}\int_L\xi\wedge\eta.
\end{equation}
It satisfies
\begin{equation}\label{eq:psgn}
\langle\xi,\eta\rangle:=(-1)^{|\eta|}\int_L\xi\wedge\eta
=(-1)^{|\eta|+|\eta|\cdot|\xi|}\int_L\eta\wedge\xi
=(-1)^{(|\eta|+1)(|\xi|+1)+1}\langle\eta,\xi\rangle.
\end{equation}

\begin{cl}[Structure equations for {$k\ge 0$},~{\cite[Proposition 2.4]{ST1}}]\label{q_rel}
For any fixed $\alpha=(\alpha_1,\ldots,\alpha_k)$, $\gamma=(\gamma_1,\ldots,\gamma_l)$,
\begin{align*}
&0=\sum_{i=1}^l(-1)^{1+\sum_{j=1}^{i-1}|\gamma_j|}\qkl(\alpha;
\otimes_{j=1}^{i-1}\gamma_j\otimes d\gamma_i\otimes\otimes_{j=i+1}^l\gamma_j)+\mbox{}\\
&+\sum_{\substack{k_1+k_2=k+1\\1\le i\le k_1\\I\sqcup J=\{1,\ldots,l\}}}
(-1)^{\iota(\alpha,\gamma;i,I)}
\q_{k_1,|I|}(\otimes_{j=1}^{i-1}\alpha_j\otimes \q_{k_2,|J|}(\otimes_{j=1}^{k_2}\alpha_{j+i-1};\otimes_{j \in J}\gamma_j)\otimes\otimes_{j=i+k_2}^k\alpha_j;\otimes_{j \in I}\gamma_j),
\end{align*}
where
\[
\iota(\alpha,\gamma;i,I)
=(1+\sum_{j\in J}|\gamma_j|)\cdot\sum_{j=1}^{i-1}(|\alpha_j|+1)+\sum_{j\in I}|\gamma_j| +\sum_{\substack{m\in I,j\in J\\ j<m}}|\gamma_m|\cdot|\gamma_j|.
\]
\end{cl}

\begin{cl}[Structure equation for {$k=-1$},~{\cite[Proposition 2.5]{ST1}}]\label{q-1_rel}
For any fixed $\gamma=(\gamma_1,\ldots,\gamma_l)$,
\begin{align*}
0=&\sum_{i=1}^l(-1)^{1+\sum_{j=1}^{i-1}|\gamma_j|} \q_{-1,l}(\otimes_{j=1}^{i-1}\gamma_j\otimes d\gamma_i\otimes \otimes_{j=i+1}^l\gamma_j)+\\
&+\frac{1}{2}\sum_{I\sqcup J=\{1,\ldots,l\}}
(-1)^{\iota(\gamma;I)}
\langle\q_{0,|I|}(\otimes_{j\in I}\gamma_j),\q_{0,|J|}( \otimes_{j\in J}\gamma_j)\rangle
+(-1)^{|\gamma|+1}\int_L i^* \q_{\emptyset,l}(\gamma).
\end{align*}
\end{cl}

\begin{lm}[Linearity, {\cite[Proposition 3.1]{ST1}}]\label{lm:qlinear}
The $\q$ operators are multilinear, in the sense that for $a \in R$ we have
\begin{multline*}
\qquad\q_{k,l}(\a_1,\ldots,\a_{i-1},a\cdot\a_i,\ldots,\a_k;\gamma_1,\ldots,\gamma_l)=\\
		=(-1)^{|a|\cdot\big(i+\sum_{j=1}^{i-1}|\a_j|+\sum_{j=1}^l|\gamma_j|\big)}
		a\cdot\q_{k,l}(\a_1,\ldots,\a_k;\gamma_1,\ldots,\gamma_l),
\end{multline*}
and for $a \in Q$ we have
\[
\q_{k,l}(\a_1,\ldots,\a_k; \gamma_1,\ldots,a\cdot\gamma_i,\ldots,\gamma_l) =(-1)^{|a|\cdot\sum_{j=1}^{i-1}|\gamma_j|}
		a\cdot\q_{k,l}(\a_1,\ldots,\a_k;\gamma_1,\ldots,\gamma_l),
\]
and
\[
\q_{\emptyset,l}(\gamma_1,\ldots,a\cdot\gamma_i,\ldots,\gamma_l)=
(-1)^{|a|\cdot\sum_{j=1}^{i-1}|\gamma_j|}a\cdot\q_{\emptyset,l}(\gamma_1,\ldots,\gamma_l).
\]
In addition, the pairing $\langle\;,\,\rangle$ defined by~\eqref{eq:pairing} is $R$-bilinear in the sense that
\[
\langle a\a_1,\a_2\rangle = a\langle \a_1,\a_2\rangle,
\quad
\langle \a_1,a\a_2\rangle=(-1)^{|a|\cdot (1+|\a_1|)}a\cdot\langle\a_1,\a_2\rangle,
\quad
\forall a\in R, \quad \a_1,\a_2\in A^*(L;R).
\]
\end{lm}

\begin{lm}[Cyclic structure, {\cite[Proposition 3.3]{ST1}}]\label{cl:cyclic}
For any $\alpha_1,\ldots,\alpha_{k+1}\in A^*(L;R)$ and $\gamma_1,\ldots,\gamma_l\in A^*(X;Q)$,
\begin{align*}
\langle\qkl(\alpha_1,\ldots&,\alpha_k;\gamma_1,\ldots\gamma_l),\alpha_{k+1}\rangle=\\
&(-1)^{(|\alpha_{k+1}|+1)\sum_{j=1}^{k}(|\alpha_j|+1)}\cdot
\langle \qkl(\alpha_{k+1},\alpha_1,\ldots,\alpha_{k-1};\gamma_1,\ldots,\gamma_l),\alpha_k\rangle.
\end{align*}
\end{lm}

The following is a straightforward generalization of~\cite[Proposition 3.5]{ST1} and~\cite[Lemma~5.8]{ST2}.
\begin{lm}[Degree]\label{deg_str_map}
For all $k,l\ge 0$ and $\beta\in \sly$, the map
\[
\qkl:A^*(L;R)^{\otimes k}\otimes A^*(X;Q)^{\otimes l} \lrarr A^*(L;R)
\]
is of degree $2-k-2l$. For $l \geq 0$ and $\beta \in \sly,$ the map
\[
\q_{-1,l} : A^*(X;Q)^{\otimes l} \lrarr R
\]
is of degree $3-n - 2l.$
\end{lm}

For a permutation $\sigma \in S_l$ and $\gamma = (\gamma_1,\ldots,\gamma_l)$ with $\gamma_j \in A^*(X;Q),$ define
\begin{equation}\label{eq:sgnsigmagamma}
s_\sigma(\gamma):=
\sum_{\substack{i<j\\ \sigma^{-1}(i)>\sigma^{-1}(j)}}|\gamma_i|\cdot|\gamma_j|
=
\sum_{\substack{i>j\\ \sigma(i)<\sigma(j)}}|\gamma_{\sigma(i)}|\cdot|\gamma_{\sigma(j)}|\pmod 2.
\end{equation}

\begin{lm}[Symmetry, {\cite[Proposition 3.6]{ST1}}]\label{cl:symmetry}
Let $k\ge -1$. For any permutation $\sigma\in S_l,$
\[
\qkl(\alpha_1,\ldots,\alpha_k;\gamma_1,\ldots,\gamma_l)=
(-1)^{s_\sigma(\gamma)}\qkl(\alpha_1,\ldots,\alpha_k;\gamma_{\sigma(1)},\ldots,\gamma_{\sigma(l)}).
\]
\end{lm}

\begin{lm}[Zero energy, {\cite[Proposition 3.8]{ST1}}]\label{lm:qzero}
For $k\ge 0,$
\[
\qkl^{\beta_0}(\alpha_1,\ldots,\alpha_k;\gamma_1,\ldots,\gamma_l)=
\begin{cases}
d\alpha_1, & (k,l)=(1,0),\\
(-1)^{|\alpha_1|}\alpha_1\wedge\alpha_2, & (k,l)=(2,0),\\
-\gamma_1|_L, & (k,l)=(0,1),\\
0, & \text{otherwise}.
\end{cases}
\]
Furthermore,
\[
\q_{-1,l}^{\beta_0}(\gamma_1,\ldots,\gamma_l)=0.
\]
\end{lm}

The following lemma is an immediate consequence of~\cite[Proposition 3.12]{ST1}.
\begin{lm}\label{no_top_deg}
Suppose $(k,l,\beta)\not\in\{(1,0,\beta_0),(0,1,\beta_0),(2,0,\beta_0)\}$.
Then $\int_L\qkl^\beta(\alpha;\gamma)=0$ for all lists $\alpha,\gamma$.
\end{lm}

\begin{lm}[Chain map, {\cite[Proposition 3.13]{ST1}}]\label{lm:qecm}
The operator
\[
\q_{\emptyset}=\oplus_{l\ge 0} \q_{\emptyset,l}
: \bigoplus_{l\ge 0}A^*(X;Q)^{\otimes l} \lrarr A^*(X;Q)
\]
is a chain map.
\end{lm}

The following lemmas are well known.
\begin{lm}[Closed unit]\label{lm:cunit}
For $\gamma_1,\ldots,\gamma_{l-1} \in A^*(X),$
\[
\q_{\emptyset,l}^\beta(1,\gamma_1,\ldots,\gamma_{l-1}) =
\begin{cases}
\gamma_1, & \beta = \beta_0 \text{ and } l = 2, \\
0, & \text{otherwise}.
\end{cases}
\]
\end{lm}

\begin{lm}[Closed degree]\label{lm:cdeg}
For $\gamma = (\gamma_1,\ldots,\gamma_l),$ and $\beta \in H_2(X;\Z),$
\[
|\q^\beta_{\emptyset,l}(\gamma)| = |\gamma| - 2l - 2c_1(\beta) + 4.
\]
Equivalently, the map
\[
\q_{\emptyset,l} : A^*(X;Q)^{\otimes l} \to A^*(X;Q)
\]
is of degree $4-2l.$
\end{lm}

\begin{lm}[Closed symmetry]\label{lm:csymm}
For $\gamma_1,\ldots,\gamma_l\in A^*(X)$ and a permutation $\sigma\in S_l$,
\[
\q_{\emptyset,l}(\gamma_1,\ldots,\gamma_l)
=(-1)^{s_\sigma(\gamma)}
\q_{\emptyset,l}(\gamma_{\sigma(1)},\ldots,\gamma_{\sigma(l)})
\]
with $s_\sigma(\gamma)$ defined by equation~\eqref{eq:sgnsigmagamma}.
\end{lm}

\begin{lm}[Closed zero energy]\label{lm:czero}
For $\gamma_1,\ldots,\gamma_l \in A^*(X),$
\[
\q_{\emptyset,l}^0(\gamma_1,\ldots,\gamma_l) =
\begin{cases}
\gamma_1 \wedge \gamma_2, & l = 2, \\
0,  & \text{otherwise.}
\end{cases}
\]
\end{lm}

\begin{lm}[Closed divisor]\label{lm:cdiv}
For $\gamma_1,\ldots,\gamma_l \in A^*(X)$ with $d\gamma_1 = 0, \,|\gamma_1| = 2,$ and $\beta \in H_2(X;\Z),$ we have
\[
\q_{\emptyset,l}^\beta(\gamma_1,\ldots,\gamma_l) = \left(\int_\beta \gamma_1\right) \cdot\q_{\emptyset,l-1}^\beta(\gamma_2,\ldots,\gamma_l).
\]
\end{lm}

\subsection{Pseudoisotopies}\label{ssec:isot}
Let $I=[0,1]$ and
let $\{J_t\}_{t\in I}$ be a path in $\J$ from $J = J_0$ to $J'=J_1.$
For each $\beta, k, l,$ set
\[
\Mt_{k+1,l}(\beta):=\{(t,u)\,|\,u\in\M_{k+1,l}(\beta;J_t)\}.
\]
We have evaluation maps
\begin{gather*}
\evbt_j:\Mt_{k+1,l}(\beta)\lrarr I\times L, \quad j\in\{0,\ldots,k\},\\
\evbt_j(t,[u,\vec{z},\vec{w}]):=(t,u(z_j)).
\end{gather*}
and
\begin{gather*}
\evit_j:\Mt_{k+1,l}(\beta)\lrarr I\times X, \quad j\in\{1,\ldots,l\}\\
\evit_j(t,[u,\vec{z},\vec{w}]):=(t,u(w_j)).
\end{gather*}
It follows from the assumption on $\J$ that all $\Mt_{k+1,l}(\beta)$ are smooth orbifolds with corners, and $\evbt_0$ is a proper submersion.
Let
\[
p_I:I\times L\lrarr I,\qquad p_{\M}:\Mt_{k+1,l}(\beta)\lrarr I,
\]
be the projections.
For $k,l\ge 0,$ define
\[
\qt_{k,l}^{\beta}:A^*(I\times L;R)^{\otimes k}\otimes A^*(I\times X;Q)^{\otimes l}\lrarr A^*(I\times L;R)
\]
by
\begin{gather*}
\qt_{1,0}^{\beta_0=0}(\at)=d\at,\quad \qt_{k,l}^{\beta}(\otimes_{j=1}^k\at_j;\otimes_{j=1}^l\gt_j):= (-1)^{\varepsilon(\at)}(\evbt_0)_*(\wedge_{j=1}^k\evbt_j^*\at_j\wedge \wedge_{j=1}^l\evit_j^*\gt_j),\\
\quad \at,\at_j\in A^*(I\times L),\qquad\gt_j\in A^*(I\times X).
\end{gather*}
Define also
\[
\qt_{-1,l}^{\beta}:A^*(I\times X;Q)^{\otimes l}\lrarr A^*(I;Q),\quad l\ge 0,
\]
by
\[
\qt_{-1,l}^{\beta}(\otimes_{j=1}^l\gt_j):= (p_\M)_*\wedge_{j=1}^l\evit_j^*\gt_j.
\]
As before, denote the sum over $\beta$ by
\begin{gather*}
\qt_{k,l}(\otimes_{j=1}^k\at_j;\otimes_{j=1}^l\gt_j):=
\sum_{\beta\in \sly}
T^{\beta}\qt_{k,l}^{\beta}(\otimes_{j=1}^k\at_j;\otimes_{j=1}^l\gt_j),\\
\qt_{-1,l}(\otimes_{j=1}^l\gt_j):=\sum_{\beta\in \sly}T^{\beta}\qt^\beta_{-1,l}(\otimes_{j=1}^l\gt_j).
\end{gather*}
Lastly, define similar operations using spheres,
\[
\qt_{\emptyset,l}:A^*(I\times X;Q)^{\otimes l}\lrarr A^*(I\times X;R),
\]
as follows. For $\beta\in H_2(X;\Z)$ let
\[
\Mt_{l+1}(\beta):=\{(t,u)\;|\,u\in \M_{l+1}(\beta;J_t)\}.
\]
For $j=0,\ldots,l,$ let
\begin{gather*}
\evt_j^\beta:\Mt_{l+1}(\beta)\to I\times X,\\
\evt_j^\beta(t,[u,\vec{w}]):=(t,u(w_j)),
\end{gather*}
be the evaluation maps. Assume that all the moduli spaces $\Mt_{l+1}(\beta)$ are smooth orbifolds and $\evt_0$ is a submersion. For $l\ge 0$, $(l,\beta)\ne (1,0),(0,0)$, set
\begin{gather*}
\qt_{\emptyset,l}^\beta(\gt_1,\ldots,\gt_l):=
(-1)^{w_\s(\beta)}
(\evt_0^\beta)_*(\wedge_{j=1}^l(\evt_j^\beta)^*\gt_j),\\
\qt_{\emptyset,l}(\gt_1,\ldots,\gt_l):=
\sum_{\beta\in H_2(X)}T^{\pr(\beta)}\qt_{\emptyset,l}^\beta(\gt_1,\ldots,\gt_l),
\end{gather*}
and define
\[
\qt_{\emptyset,1}^0:= 0,\qquad \qt_{\emptyset,0}^0:= 0.
\]

\begin{cl}[Structure equations for {$k\ge 0$}, {\cite[Proposition 4.3]{ST1}}]\label{qt_rel}
For any fixed $\at=(\at_1,\ldots,\at_k)$, $\gt=(\gt_1,\ldots,\gt_l)$,
\begin{align*}
&0=\sum_{i=1}^l(-1)^{1+\sum_{j=1}^{i-1}|\gt_j|}\qt_{k,l}(\at;
\otimes_{j=1}^{i-1}\gt_j\otimes d\gt_i\otimes\otimes_{j=i+1}^l\gt_j)+\mbox{}\\
&+\sum_{\substack{k_1+k_2=k+1\\1\le i\le k_1\\I\sqcup J=\{1,\ldots,l\}}}
(-1)^{\iota(\at,\gt;i,I)}
\qt_{k_1,|I|}(\otimes_{j=1}^{i-1}\at_j\otimes \qt_{k_2,|J|}(\otimes_{j=1}^{k_2}\at_{j+i-1};\otimes_{j \in J}\gt_j)\otimes\otimes_{j=i+k_2}^k\at_j;\otimes_{j \in I}\gt_j).
\end{align*}
\end{cl}

To formulate the next results, define
\[
\ll \xit,\etat\gg :=
(-1)^{|\etat|}(p_I)_*(\xit\wedge\etat),
\qquad
\xit,\etat\in A^*(I\times L;R).
\]
\begin{cl}[Structure equations for {$k=-1$}, {\cite[Proposition 4.4]{ST1}}]\label{cl:qt_-1}
For any fixed $\gt=(\gt_1,\ldots,\gt_l)$,
\begin{align*}
-d\qt_{-1,l}(\gt)=&\sum_{i=1}^l(-1)^{1+\sum_{j=1}^l |\gt_j|} \qt_{-1,l}(\otimes_{j=1}^{i-1}\gt_j\otimes d\gt_i\otimes \otimes_{j=i+1}^l\gt_j)+\\
&+\frac{1}{2}\sum_{I\sqcup J=\{1,\ldots,l\}}
(-1)^{\iota(\gt;I)}
\ll\qt_{0,|I|}(\otimes_{j\in I}\gt_j),\qt_{0,|J|}(\otimes_{j\in J}\gt_j)\gg
+ (-1)^{|\gt|+1}(p_I)_* i^* \qt_{\emptyset,l}(\gt).
\end{align*}
\end{cl}

\begin{lm}[{\cite[Proposition 4.18]{ST1}}]\label{lm:mt0}
For all lists $\gt=(\gt_1,\ldots,\gt_l)$, we have
\[
\ll\qt_{0,l}(\gt),1\gg=
\begin{cases}
0, & l\ge 1,\\
-(p_I)_*(\gt_1|_{I\times L}), & l=1.
\end{cases}
\]
\end{lm}

\begin{lm}[Chain map, {\cite[Proposition 4.19]{ST1}}]\label{lm:qtchain}
The operator
\[
\qt_{\emptyset}=\oplus_{l\ge 0} \qt_{\emptyset,l}
: \bigoplus_{l\ge 0}A^*(I\times X;Q)^{\otimes l} \lrarr A^*(I\times X;Q)
\]
is a chain map.
\end{lm}

Recall the notion of bounding pairs from Definition~\ref{dfn:bdpair}.
\begin{dfn}\label{dfn:gequiv}
We say a bounding pair $(\gamma,b)$ with respect to $J$ is \textbf{gauge-equivalent} to a bounding pair $(\gamma',b')$ with respect to $J'$, if there exist
$\gt\in\mI A^*(X;Q)$ and $\bt\in\mJ A^*(L;R)$ such that
\begin{gather*}
j_0^*\gt=\gamma,\quad j_1^*\gt=\gamma',\quad j_0^*\bt=b,\quad j_1^*\bt=b',\quad d\gt=0,\quad |\gt|=2,\\
\sum_{k,l\ge 0}\qt_{k,l}(\bt^{\otimes k};\gt^{\otimes l})=\ct\cdot 1,\qquad \ct\in \mJ,\quad |\ct| = 2.
\end{gather*}
In this case, we say that $(\gt,\bt)$ is a pseudo-isotopy from $(\gamma,b)$ to $(\gamma',b')$ and write $(\gamma,b)\sim(\gamma',b')$.
\end{dfn}

\section{Geodesic conditions}\label{sec:geod}

\subsection{Geodesic operators}\label{ssec:geoddfn}
Let $a,e\in \Z_{\ge 0}$ such that $a+e=3$, and let $\M_{k+1,l;a,e}(\beta)\subset \M_{k+1,l}(\beta)$ be the closure of the subspace consisting of one-component maps such that $a$ of the boundary points and the first $e$ of the interior points lie on a common geodesic in the domain with respect to the hyperbolic metric.
When we need to specify which of the boundary points are taken to lie on a geodesic, we add their labels as sub-indices to $a$, in which case the order of the indices indicates the order in which the points appear on the geodesic. If not indicated explicitly, the points are assumed to appear according to their labeling order. For example, $\M_{k+1,l;2_{0,k},1}(\beta)$ is the space of stable disks with $k+1$ boundary and $l$ marked points, such that the first interior point lies on the geodesic between the zeroth and last boundary points. In $\M_{k+1,l;1_0,2}(\beta)$, the geodesic starts at the zeroth boundary point and passes through the first and second interior points, in that order. As mentioned in Section~\ref{sssec:reg}, we assume that $evb_0|_{\M_{k+1,l;a,e}(\beta)}$ is a proper submersion.

To determine the orientation on $\M_{k+1,l;a,e}(\beta)$, it is useful to identify it with a fiber product of oriented orbifolds, as follows.
Denote by $v_1,v_2,v_3\in\{z_0,\ldots,z_{k},w_1,\ldots,w_l\}$ the marked points that lie on the geodesic, labeled according to the order in which they appear on the geodesic.
Given a nodal Riemann surface with boundary $\Sigma$ with complex structure $j,$ denote by $\overline \Sigma$ a copy of $\Sigma$ with the opposite complex structure $-j.$ For a point $v \in \Sigma,$ let $\bar v \in \overline \Sigma$ denote the corresponding point. The complex double $\Sigma_\C = \Sigma\coprod_{\d\Sigma} \overline{\Sigma}$ is a closed nodal Riemann surface, so it is possible to define the cross ratio of four points on $\Sigma_\C$ as in~\cite[Appendix~D.4]{MS}.
We define
\begin{gather*}
\chi:\M_{k+1,l}(\beta)\lrarr \C,\\
\chi(u,\Sigma,\vec{z},\vec{w})
:=(\bar{v}_2,v_1,v_3,v_2).
\end{gather*}
On the irreducible locus of $\M_{k+1,l}(\beta)$, the domain $\Sigma$ can be identified with the upper half plane and $\chi$ has the explicit formula
\[
\chi(u,\vec{z},\vec{w})
=\frac{(v_3-\bar{v}_2)(v_2-v_1)}{(v_3-v_1)(v_2-\bar{v}_2)}.
\]
Note that the second marked point on the geodesic is necessarily an interior point, so $\bar{v}_2\ne v_2$, and thus $\chi$ is well defined.
Then the condition that $v_1,v_2,v_3,$ lie on a geodesic is equivalent to the condition $\chi(u,\vec{z},\vec{w})\in [0,1]$.
Thus, we have
\begin{equation*}
\M_{k+1,l;a,e}(\beta)= \M_{k+1,l;\chi}=
\xymatrix{
[0,1]\times_{\C}\M_{k+1,l}(\beta)\ar[r]\ar[d] & \M_{k,l}(\beta)\ar[d]^{\chi}\\
[0,1]\ar@{^{(}->}[r]& \C ,
}
\end{equation*}
and the fiber product identification determines orientation, as in~\cite{ST4}.

Denote by
\[
\q_{k,l;\chi}^\beta=\q_{k,l;a,e}^\beta: A^*(L; R)^{\otimes k}\otimes A^*(X;Q)^{\otimes l}\lrarr A^*(L;R)
\]
the operators defined analogously to $\qkl^\beta$  with $\M_{k+1,l;a,e}(\beta)$ in place of $\M_{k+1,l}(\beta)$. Explicitly,
\begin{align*}
\q_{k,l;a,e}^\beta(\alpha_1\otimes\cdots\otimes\alpha_k; \gamma_1\otimes\cdots\otimes\gamma_l):=
(-1)^{\varepsilon_\chi(\alpha)}
(evb_0^\beta)_* \left(\bigwedge_{j=1}^l(evi_j^\beta)^*\gamma_j\wedge \bigwedge_{j=1}^k (evb_j^\beta)^*\alpha_j\right)
\end{align*}
with
\[
\varepsilon_\chi(\alpha):=\varepsilon(\alpha) +|\alpha|+k
=\sum_{j=1}^k(j+1)(|\alpha_j|+1)+1.
\]
In addition, define
\[
\q_{-1,l;0,e}^\beta:(A^*(X;Q))^{\otimes l}\lrarr R
\]
by
\[
\q_{-1,l;0,e}^\beta(\gamma_1\otimes\cdots\otimes\gamma_l)
:=
-\int_{\M_{0,l;0,e}(\beta)} \bigwedge_{j=1}^l (evi_j^\beta)^*\gamma_j.
\]
Set
\begin{align*}
\q_{k,l;\chi}=\q_{k,l;a,e}:=\sum_{\beta\in\sly}T^{\beta}\q_{k,l;a,e}^{\beta}.
\end{align*}
Again, specifying boundary points and the order of the points on the geodesic can be done by adding a sub-index to $a$ and $e$.

Lastly,
consider the moduli space $\M_{l+1}(\beta)$ of spheres with $l+1$ marked points $w_0,\ldots,w_l$.
Let
\begin{gather*}
\chi_0:\M_{l+1}(\beta)\lrarr \C,\\
\chi_0(u,\Sigma,\vec{w})
:=(w_0,w_1,w_3,w_2),
\end{gather*}
be the cross ratio map. Let $\M_{l+1;\chi_0}(\beta)=[0,1]\times_{\C}\M_{l+1}(\beta)$ be the associated geodesic moduli space
and denote by
\[
\q_{\emptyset,l;\chi_0},\q^\beta_{\emptyset,l;\chi_0}: A^*(X;Q)^{\otimes l}\lrarr A^*(X;Q)
\]
the associated operators.
Explicitly,
\[
\q^\beta_{\emptyset,l;\chi_0} (\gamma_1\otimes\cdots\otimes \gamma_l)=
(-1)^{w_\s(\beta)}
(ev_0^\beta)_*\Big(\bigwedge_{j=1}^l(ev_j^\beta)^*\gamma_j\Big),
\]
and
\[
\q_{\emptyset,l;\chi_0}
:=\sum_{\beta\in H_2(X;\Z)}T^{\varpi(\beta)}
\q^\beta_{\emptyset,l;\chi_0}.
\]

\subsection{Structure equations}\label{ssec:geo_a_infty}
Here we formulate structure equation for the geodesic $\q$ operations, similarly to the structure equations that govern the usual $\q$ operators.
The proofs are similar to those of~\cite[Propositions 2.4-2.5]{ST1}, and are based on Stokes' theorem, Proposition~\ref{stokes}.

For ordered lists $B = (b_1,\ldots,b_k)$ and $I=(i_1,\ldots,i_l),$ denote by $B\circ I$ the ordered list resulting from concatenation, i.e., $B\circ I:=(b_1,\ldots,b_k,i_1,\ldots,i_l).$

Let $B$ be a list of indices and let $\eta=(\eta_j)_{j\in B}\subset A^*(X;\Qh)$. For a sublist $I\subset B$, denote by $\eta^I$ the list $(\eta_j)_{j\in I}$.
For a partition $I\sqcup J$ of $B$ into two ordered sub-lists, define
$sgn(\sigma^\eta_{I\sqcup J})$ by the equation
\[
\bigwedge_{i \in I} \eta_i \wedge \bigwedge_{j \in J} \eta_j = (-1)^{sgn(\sigma^\eta_{I\sqcup J})}\bigwedge_{k \in B} \eta_k,
\]
where the wedge products are taken in the order of the respective lists.
Explicitly,
\[
sgn(\sigma^\eta_{I\sqcup J})\equiv
\sum_{\substack{i\in I,j\in J\\ j<i}}|\eta_i|\cdot|\eta_j|\pmod 2.
\]

Throughout, let $\alpha=(\alpha_1,\ldots,\alpha_k)$ and $\gamma=(\gamma_1,\ldots,\gamma_l)$ be lists with $\alpha_j\in A^*(L;R)$ and $\gamma_j\in \Ah^*(X,L;\Qh)$.

Use the following notation for signs, modulo $2$:

\begin{gather*}
\iota_1 
:=\iota_1(\a,\gamma;i,I)
:=(|\gamma^J|+1)\sum_{j=1}^{i-1}(|\alpha_j|+1) +sgn(\sigma^\gamma_{I\sqcup J})+|\gamma|,\\
\iota_2 
:=\iota_2(\a,\gamma;i,I)
:=|\gamma^J|\sum_{j=1}^{i-1}(|\alpha_j|+1)+|\gamma^I| +sgn(\sigma^\gamma_{I\sqcup J})+1,\\
\iota_3 
:=\iota_3(\a,\gamma;I)
:=|\a|+k+|\gamma|+n+1+sgn(\sigma^\gamma_{J\sqcup I}).
\end{gather*}

\begin{prop}[The spaces $\M_{k+1,l;2_{0,m},1}(\beta)$]\label{prop:21}
\begin{align*}
&0=\sum_{i=1}^l
(-1)^{\sum_{j=1}^{i-1}|\gamma_j|}
\q_{k,l;2_{0,m},1}(\otimes_{j=1}^k\alpha_j; \otimes_{j=1}^{i-1}\gamma_j\otimes d\gamma_i\otimes \otimes_{j=i+1}^l\gamma_j)+\tag{AI}\label{AI}\\
&+\hspace{-1em}\sum_{\substack{k_1+k_2=k+1\\m+1\le i\le k_1\\I\sqcup J=\{2,\ldots,l\}}}\hspace{-1em}
(-1)^{\iota_1(\alpha,\gamma;i, (1)\circ I)} \q_{k_1,l_1;2_{0,m},1}(\otimes_{j=1}^{i-1}\alpha_j\otimes \q_{k_2,l_2}(\otimes_{j=1}^{k_2}\alpha_{j+i-1};\otimes_{j\in J}\gamma_j)\otimes\\
&\hspace{22em}
\otimes \otimes_{j=i+1}^{k_1}\alpha_{j+k_2-1}; \gamma_1\otimes\otimes_{j\in I}\gamma_j)+\tag{AII}\label{AII}\\
&+\hspace{-0.75em}\sum_{\substack{k_1+k_2=k+1\\1\le i\le m-k_2\\I\sqcup J=\{2,\ldots,l\}}}\hspace{-1em}
(-1)^{\iota_1(\alpha,\gamma;i,(1)\circ I)} \q_{k_1,l_1;2_{0,m-k_2+1},1}(\otimes_{j=1}^{i-1}\alpha_j \otimes \q_{k_2,l_2}(\otimes_{j=1}^{k_2}\alpha_{j+i-1};\otimes_{j\in J}\gamma_j)\otimes \\
&\hspace{21em}
\otimes\otimes_{j=i+1}^{k_1}\alpha_{j+k_2-1};\gamma_1\otimes\otimes_{j\in I}\gamma_j)+\tag{AIII}\label{AIII}\\
&+\hspace{-1em}\sum_{\substack{k_1+k_2=k+1\\ m-k_2+1 \le i\le m \\I\sqcup J=\{2,\ldots,l\}}}\hspace{-1em}
(-1)^{\iota_2(\alpha,\gamma;i,I)} \q_{k_1,l_1}(\otimes_{j=1}^{i-1}\alpha_j\otimes \q_{k_2,l_2;2_{0,m-i+1},1}(\otimes_{j=1}^{k_2}\alpha_{j+i-1}; \gamma_1\otimes\otimes_{j\in J}\gamma_j)\otimes \\
&\hspace{23em}
\otimes\otimes_{j=i+1}^{k_1}\alpha_{j+k_2-1};\otimes_{j\in I}\gamma_j)+\tag{AIV}\label{AIV}\\
&+\hspace{-0.5em}\sum_{\substack{k_1+k_2=k+1\\ m-k_2+1 \le i\le m\\I\sqcup J=\{2,\ldots,l\}}}\hspace{-1em}
(-1)^{\iota_1(\alpha,\gamma;i,(1)\circ I)} \q_{k_1,l_1;2_{0,i},1}(\otimes_{j=1}^{i-1}\alpha_j\otimes \q_{k_2,l_2}(\otimes_{j=1}^{k_2}\alpha_{j+i-1};\otimes_{j\in J}\gamma_j)\otimes\\
&\hspace{23em}
\cdot\otimes_{j=i+1}^{k_1}\alpha_{j+k_2-1}; \gamma_1\otimes\otimes_{j\in I}\gamma_j).\tag{AV}\label{AV}
\end{align*}
\end{prop}
See Figure~\ref{pic:21}.
\begin{figure}[ht]
\centering
\includegraphics{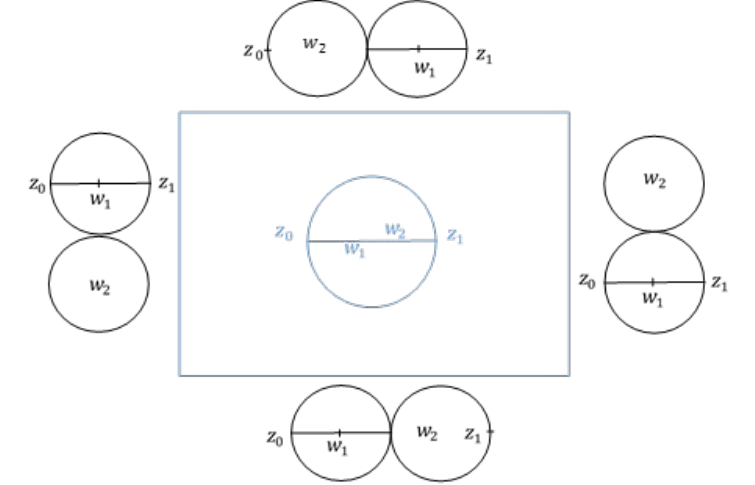}
\caption{Boundary components of $\M_{2,2;2_{0,1},1}(\beta_0)$}\label{pic:21}
\end{figure}

\begin{prop}[The spaces $\M_{k+1,l;1_{0},2}(\beta)$]\label{prop:12}
\begin{align*}
&0=\sum_{i=1}^l
(-1)^{\sum_{j=1}^{i-1}|\gamma_j|} \q_{k,l;1_0,2}(\otimes_{j=1}^k\alpha_j; \otimes_{j=1}^{i-1}\gamma_j\otimes d\gamma_i\otimes \otimes_{j=i+1}^l\gamma_j)+\tag{BI}\label{BI}\\
& +\hspace{-1em}\sum_{\substack{k_1+k_2=k+1\\ 1\le i\le k_1\\I\sqcup J=\{3,\ldots,l\}}}\hspace{-1em}
(-1)^{\iota_1(\alpha,\gamma;i, (1,2)\circ I)} \q_{k_1,l_1;1_0,2} (\otimes_{j=1}^{i-1}\alpha_j\otimes \q_{k_2,l_2}(\otimes_{j=i}^{i+k_2-1}\alpha_j;\otimes_{j\in J}\gamma_j)\otimes \notag\\
&\hspace{20em}\otimes \otimes_{j=i+k_2}^k\alpha_j;\gamma_1\otimes\gamma_2\otimes\otimes_{j\in I}\gamma_j)+\tag{BII}\label{BII}\\
& +\hspace{-1em}\sum_{\substack{k_1+k_2=k+1\\ 1\le i\le k_1\\I\sqcup J=\{3,\ldots,l\}}}\hspace{-1em}
(-1)^{\iota_2(\alpha,\gamma;i,I)} \q_{k_1,l_1}( \otimes_{j=1}^{i-1}\alpha_j\otimes \q_{k_2,l_2;1_0,2}(\otimes_{j=i}^{i+k_2-1}\alpha_j; \gamma_1\otimes\gamma_2\otimes\otimes_{j\in J}\gamma_j)\otimes\notag\\
&\hspace{23em}\otimes \otimes_{j=i+k_2}^k\alpha_j;\otimes_{j\in I}\gamma_j)+\tag{BIII}\label{BIII}\\
& +\hspace{-1em}\sum_{\substack{k_1+k_2=k+1\\ 1\le i\le k_1\\I\sqcup J=\{3,\ldots,l\}}}\hspace{-1em}
(-1)^{\iota_1(\alpha,\gamma;i, (1)\circ I)}\q_{k_1,l_1;2_{0,i},1} (\otimes_{j=1}^{i-1}\alpha_j\otimes \q_{k_2,l_2}(\otimes_{j=i}^{i+k_2-1}\alpha_j; \gamma_2\otimes\otimes_{j\in J}\gamma_j)\otimes\notag\\
&\hspace{23em}\otimes \otimes_{j=i+k_2}^k\alpha_j; \gamma_1\otimes\otimes_{j\in I}\gamma_j)+\tag{BIV}\label{BIV}\\
&+\hspace{-1em}\sum_{I\sqcup J=\{3,\ldots,l\}}\hspace{-1em}
(-1)^{\iota_3(\alpha,\gamma;I)} \q_{k,l_1}(\alpha; \q_{\emptyset,l_2}(\gamma_1\otimes\gamma_2\otimes\otimes_{j\in J}\gamma_j)\otimes\otimes_{j\in I}\gamma_j).\tag{BV}\label{BV}
\end{align*}
\end{prop}
See Figure~\ref{pic:12}.
\begin{figure}[ht]
\centering
\includegraphics{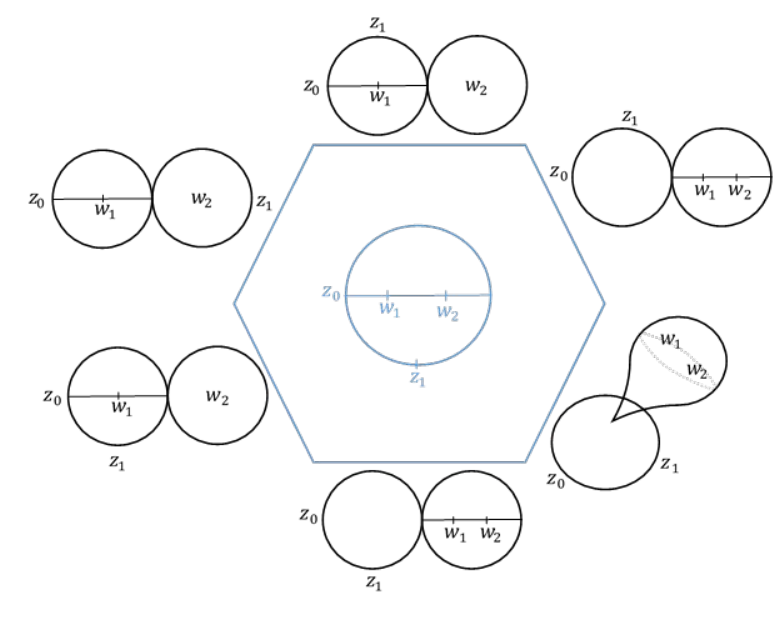}
\caption{Boundary components of $\M_{1,2;1_0,2}(\beta_0)$}\label{pic:12}
\end{figure}

\begin{prop}[The spaces $\M_{k+1,l;0,3}(\beta)$]\label{prop:03}
For $k\ge 0,$
\begin{align*}
&0=\sum_{i=1}^l(-1)^{\sum_{j=1}^{i-1}|\gamma_j|}
\q_{k,l;0,3}(\otimes_{j=1}^k\alpha_j; \otimes_{j=1}^{i-1}\gamma_j\otimes d\gamma_i\otimes \otimes_{j=i+1}^l\gamma_j)+\tag{CI}\label{CI}\\
& +\hspace{-1em}\sum_{\substack{k_1+k_2=k+1\\ 1\le i\le k_1\\I\sqcup J=\{4,\ldots,l\}}}\hspace{-1em}
(-1)^{\iota_1(\alpha,\gamma;i,(1,2,3)\circ I)} \q_{k_1,l_1;0,3}(\otimes_{j=1}^{i-1}\alpha_j\otimes
\q_{k_2,l_2}(\otimes_{j=i}^{i+k_2-1}\alpha_j;\otimes_{j\in J}\gamma_j)\otimes \\
&\hspace{20em}
\otimes \otimes_{j=i+k_2}^k\alpha_j; \gamma_1\otimes\gamma_2\otimes\gamma_3\otimes\otimes_{j\in I}\gamma_j)+\tag{CII}\label{CII}\\
& + \hspace{-1em}\sum_{\substack{k_1+k_2=k+1\\ 1\le i\le k_1\\I\sqcup J=\{4,\ldots,l\}}}\hspace{-1em}
(-1)^{\iota_2(\alpha,\gamma;i,I)} \q_{k_1,l_1}( \otimes_{j=1}^{i-1}\alpha_j\otimes \q_{k_2,l_2;0,3}( \otimes_{j=i}^{i+k_2-1}\alpha_j; \gamma_1\otimes\gamma_2\otimes\gamma_3\otimes\otimes_{j\in J}\gamma_j)\otimes\\
&\hspace{27em}
\otimes \otimes_{j=i+k_2}^k\alpha_j;\otimes_{j\in I}\gamma_j)+\tag{CIII}\label{CII2}\\
& +\hspace{-1em}\sum_{\substack{k_1+k_2=k+1\\ 1\le i\le k_1\\I\sqcup J=\{4,\ldots,l\}}}\hspace{-1em}
(-1)^{\iota_2(\alpha,\gamma;i,(1)\circ I)} \q_{k_1,l_1}(\otimes_{j=1}^{i-1}\alpha_j\otimes \q_{k_2,l_2;1_0,2}(\otimes_{j=i}^{i+k_2-1}\alpha_j; \gamma_2\otimes\gamma_3\otimes\otimes_{j\in J}\gamma_j)\otimes\\
&\hspace{25em}\otimes
\otimes_{j=i+k_2}^k\alpha_j; \gamma_1\otimes\otimes_{j\in I}\gamma_j)+\tag{CIV}\label{CIII}\\
&+\hspace{-1em}\sum_{\substack{k_1+k_2=k+1\\ 1\le i\le k_1\\I\sqcup J=\{4,\ldots,l\}}}\hspace{-1em}
(-1)^{\iota_1(\alpha,\gamma;i,(2,3)\circ I)}
\q_{k_1,l_1;1_i,2}(\otimes_{j=1}^{i-1}\a_j\otimes
\q_{k_2,l_2}(\otimes_{j=i}^{i+k_2-1}\a_j;\gamma_1\otimes\otimes_{j\in J}\gamma_j) \otimes\\
&\hspace{23em}\otimes
\otimes_{j=i+k_2}^k\a_j;\gamma_2\otimes\gamma_3\otimes\otimes_{j\in I} \gamma_j)+\tag{CV}\label{CIII2}\\
& -\hspace{-1em}\sum_{\substack{k_1+k_2=k+1\\ 1\le i\le k_1\\I\sqcup J=\{4,\ldots,l\}}}\hspace{-1em}
(-1)^{\iota_1(\alpha,\gamma;i,(2,1)\circ I)} \q_{k_1,l_1;1_i,2}(\otimes_{j=1}^{i-1}\alpha_j\otimes \q_{k_2,l_2}(\otimes_{j=i}^{i+k_2-1}\alpha_j; \gamma_3\otimes\otimes_{j\in J}\gamma_j)\otimes\\
&\hspace{23em}\otimes
\otimes_{j=i+k_2}^k\alpha_j; \gamma_2\otimes\gamma_1\otimes\otimes_{j\in I}\gamma_j)-\tag{CVI}\label{CIV}\\
& -\hspace{-1em}\sum_{\substack{k_1+k_2=k+1\\ 1\le i\le k_1\\I\sqcup J=\{4,\ldots,l\}}}\hspace{-1em}
(-1)^{\iota_2(\a,\gamma;i,(2,1)\circ I)}
\q_{k_1,l_1}(\otimes_{j=1}^{i-1}\a_j\otimes
\q_{k_2,l_2;1_0,2}(\otimes_{j=i}^{i+k_2-1}\a_j;\gamma_2\otimes \gamma_1\otimes \otimes_{j\in J} \gamma_j)\otimes\\
&\hspace{25em}\otimes
\otimes_{j=i+k_2}^k\a_j;\gamma_3\otimes \otimes_{j\in I} \gamma_j)-\tag{CVII}\label{CIV2}\\
&-\hspace{-1em}\sum_{I\sqcup J=\{4,\ldots,l\}}\hspace{-1em}
(-1)^{\iota_3(\gamma;(3)\circ I)}\q_{k,l_1}(\alpha; \q_{\emptyset,l_2}(\gamma_1\otimes\gamma_2\otimes\otimes_{j\in J}\gamma_j)\otimes\gamma_3\otimes\otimes_{j\in I}\gamma_j)+\tag{CVIII}\label{CV}\\
&+\hspace{-1em}\sum_{I\sqcup J=\{4,\ldots,l\}}\hspace{-1em}
(-1)^{\iota_3(\gamma;(1)\circ I)}\q_{k,l_1}(\alpha; \q_{\emptyset,l_2}(\gamma_2\otimes\gamma_3\otimes\otimes_{j\in J}\gamma_j)\otimes\gamma_1\otimes\otimes_{j\in I}\gamma_j)\tag{CIX}\label{CVI}.
\end{align*}
\end{prop}
See Figure~\ref{pic:03}.
\begin{figure}[ht]
\centering
\includegraphics{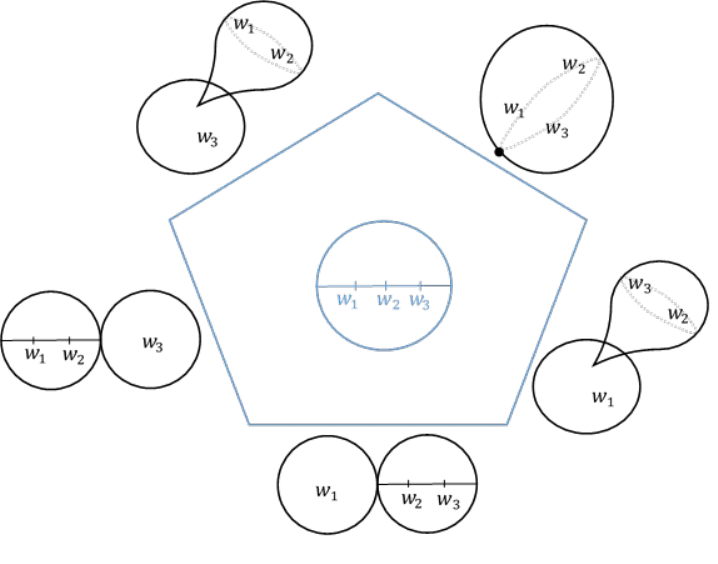}
\caption{Boundary components of $\M_{0,3;0,3}(\beta_0)$}\label{pic:03}
\end{figure}

\begin{prop}[The spaces $\M_{0,l;0,3}(\beta)$]\label{prop:03k-1}
\begin{align*}
&0=\sum_{i=1}^l(-1)^{\sum_{j=1}^{i-1}|\gamma_j|} \q_{-1,l;0,3}( \otimes_{j=1}^{i-1}\gamma_j\otimes d\gamma_i\otimes \otimes_{j=i+1}^l\gamma_j)+\tag{DI}\label{DI}\\
&+\sum_{I\sqcup J=\{4,\ldots,l\}}
(-1)^{\iota_1(\gamma;(1,2,3)\circ I)}\langle\q_{0,|I|;0,3} (\gamma_1\otimes\gamma_2\otimes\gamma_3\otimes\otimes_{j\in I}\gamma_j),\q_{0,|J|}(\otimes_{j\in J}\gamma_j)\rangle+\tag{DII}\label{DII}\\
&+\sum_{I\sqcup J=\{4,\ldots,l\}}
(-1)^{\iota_2(\gamma;(1)\circ I)}
\langle\q_{0,|I|} ( \gamma_1\otimes\otimes_{j\in I}\gamma_j),\q_{0,|J|;1_0,2}( \gamma_2\otimes\gamma_3\otimes\otimes_{j\in J}\gamma_j)\rangle+\tag{DIII}\label{DIII}\\
&-\sum_{I\sqcup J=\{4,\ldots,l\}}
(-1)^{\iota_1(\gamma;(2,1)\circ I)}
\langle\q_{0,|I|;1_0,2} ( \gamma_2\otimes\gamma_1\otimes\otimes_{j\in I}\gamma_j),\q_{0,|J|}(\gamma_3\otimes\otimes_{j\in J}\gamma_j)\rangle+\tag{DIV}\label{DIV}\\
&-\sum_{I\sqcup J=\{4,\ldots,l\}}
(-1)^{\iota_3(\gamma;(3)\circ I)}
\q_{-1,|I|}( \q_{\emptyset,|J|}(\gamma_1\otimes \gamma_2\otimes\otimes_{j\in J}\gamma_j)\otimes\gamma_3\otimes\otimes_{j\in I}\gamma_j)+\tag{DV}\label{DV}\\
&+\sum_{I\sqcup J=\{4,\ldots,l\}}
(-1)^{\iota_3(\gamma;(1)\circ I)}
\q_{-1,|I|}( \q_{\emptyset,|J|}(\gamma_2\otimes\gamma_3\otimes\otimes_{j\in J}\gamma_j)\otimes\gamma_1\otimes\otimes_{j\in I}\gamma_j)+\tag{DVI}\label{DVI}\\
&+(-1)^{\iota_3(\gamma;I=\emptyset)}\int_L i^*(\q_{\emptyset,l;\chi_0} (\gamma_{1}\otimes\gamma_{2}\otimes\gamma_{3}\otimes \otimes_{j=4}^l \gamma_j)).\tag{DVII}\label{DVII}
\end{align*}
\end{prop}

\begin{prop}\label{lm:clst}
\begin{align*}
d\q_{\emptyset,l;\chi_0}&(\gamma_1,\ldots,\gamma_l)=\\
=&
\sum_{i=1}^l(-1)^{\sum_{j=1}^{i-1}|\gamma_j|} \q_{\emptyset,l;\chi_0}( \otimes_{j=1}^{i-1}\gamma_j\otimes d\gamma_i\otimes \otimes_{j=i+1}^l\gamma_j)+\\
&+(-1)^{|\gamma|+1+ sgn(\sigma^\gamma_{(1,2,3)\circ J\sqcup I})}\hspace{-.75em}\sum_{I\sqcup J=\{4,\ldots,l\}}
\hspace{-.75em}\q_{\emptyset,|I|+2}(\gamma_1\otimes \q_{\emptyset,|J|+2}(\gamma_2\otimes\gamma_3\otimes \otimes_{j \in J} \gamma_j) \otimes\otimes_{j \in I}\gamma_j)-\\
&-(-1)^{|\gamma|+1 + sgn(\sigma^\gamma_{(1,2) \circ J\sqcup (3) \circ I})}\hspace{-.75em}\sum_{I\sqcup J=\{4,\ldots,l\}}\hspace{-.75em}
\q_{\emptyset,|I|+2}(\q_{\emptyset,|J|+2}(\gamma_1\otimes\gamma_2\otimes \otimes_{j \in J}\gamma_j)\otimes\gamma_3 \otimes\otimes_{j \in I}\gamma_j).
\end{align*}
\end{prop}
See Figure~\ref{fig:bdclgeod}.
\begin{figure}[ht]
\centering
\includegraphics[width=12cm]{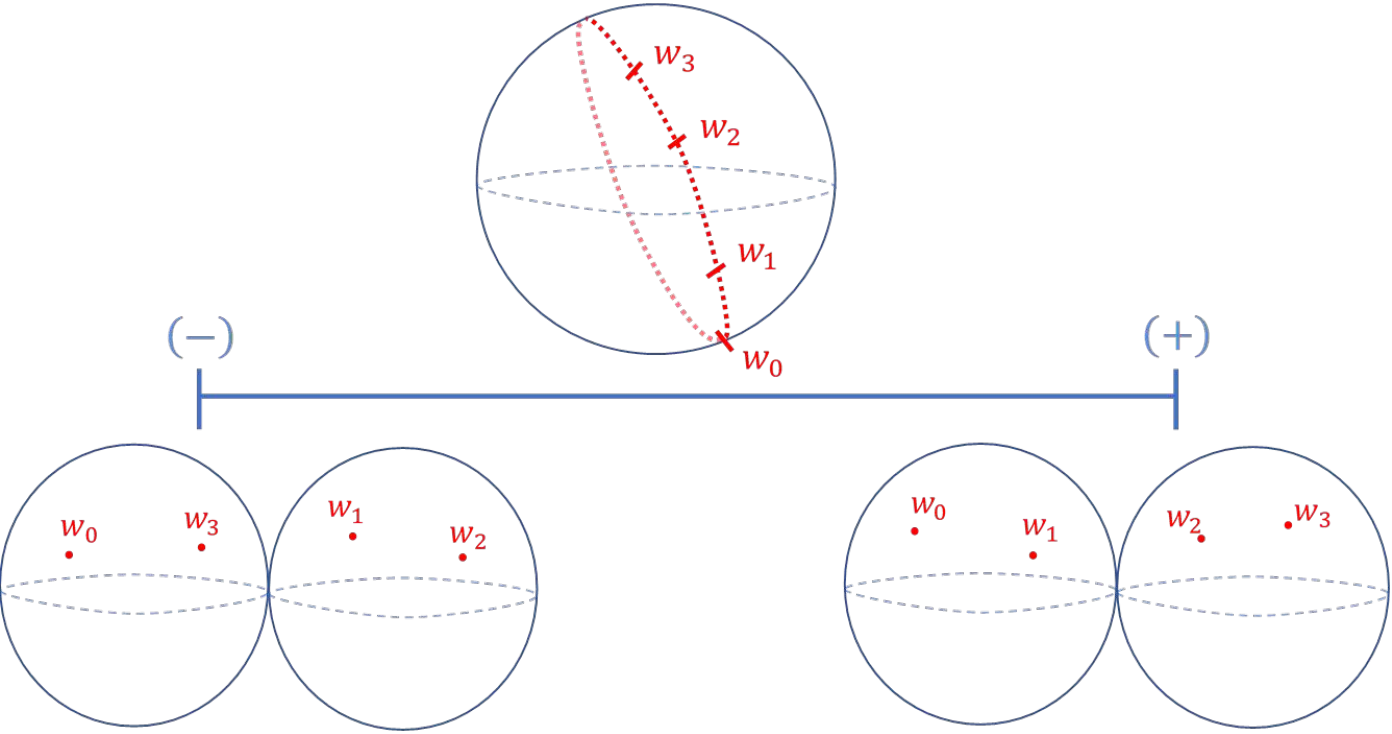}
\caption{Boundary components of $\M_{4;\chi_0}(\beta_0)$}\label{fig:bdclgeod}
\end{figure}
\begin{rem}\label{rem:clst}
Proposition~\ref{lm:clst} implies the WDVV equation for the closed Gromov-Witten potential~\eqref{eq:Phi}. Equivalently, the quantum product $\star_U$ of Section~\ref{sssec:qprod} is associative.
\end{rem}

\subsection{Properties}\label{ssec:prop_geod}

\subsubsection{Degree}
\begin{lm}\label{lm:deggeod}
For any $k,l\ge 0$, and a geodesic condition $\chi$, the map
\[
\q_{k,l;\chi}:A^*(L;R)^{\otimes k}\otimes A^*(X;Q)^{\otimes l} \lrarr A^*(L;R)
\]
is of degree $3-k-2l$. For all $l\ge 0$ and $\beta\in H_2(X;\Z)$, the map
\[
\q_{\emptyset,l;\chi_0}:A^*(X;Q)^{\otimes l} \lrarr A^*(X;Q)
\]
is of degree $5-2l$.
\end{lm}

The proof is similar to that of Proposition 3.5 in~\cite{ST1}.

\subsubsection{Linearity}
The following is a direct analog of the linearity properties of the usual $\q$ operators.
The signs reflect the fact that the map of shifted complexes
\[
\q_{k,l;\chi}: A^*(L; R)[1]^{\otimes k}\otimes A^*(X;Q)^{\otimes l}\lrarr A^*(L;R)[1]
\]
has degree $0 \pmod 2$.

\begin{lm}\label{lm:geodqlinear}
The geodesic $\q$ operators are multilinear, in the sense that for $a \in R$ we have
\begin{multline*}
\qquad\q_{k,l;\chi}(\a_1,\ldots,\a_{i-1},a\cdot\a_i,\ldots,\a_k;\gamma_1,\ldots,\gamma_l)=\\
		=(-1)^{|a|\cdot\big(\sum_{j=1}^{i-1}(|\a_j|+1)+\sum_{j=1}^l|\gamma_j|\big)}
		a\cdot\q_{k,l;\chi}(\a_1,\ldots,\a_k;\gamma_1,\ldots,\gamma_l),
\end{multline*}
and for $a \in Q$ we have
\[
\q_{k,l;\chi}(\a_1,\ldots,\a_k; \gamma_1,\ldots,a\cdot\gamma_i,\ldots,\gamma_l) =(-1)^{|a|\cdot\sum_{j=1}^{i-1}|\gamma_j|}
		a\cdot\q_{k,l;\chi}(\a_1,\ldots,\a_k;\gamma_1,\ldots,\gamma_l),
\]
and
\[
\q_{\emptyset,l;\chi_0}(\gamma_1,\ldots,a\cdot\gamma_i,\ldots,\gamma_l)=
(-1)^{|a|\cdot\sum_{j=1}^{i-1}|\gamma_j|}a\cdot\q_{\emptyset,l;\chi_0}(\gamma_1,\ldots,\gamma_l).
\]
\end{lm}

\subsubsection{Unit on the geodesic}
The followings lemmas concern geodesic operators where the unit is fed to one of the inputs constrained to the geodesic. They have no direct analog for the usual $\q$ operators.
\begin{lm}\label{lm:geodunit-1} 
\[
\q_{-1,l}(\gamma)
= (-1)^{n}\langle\q_{0,l;1_0,2}(\gamma),1\rangle.
\]
\end{lm}
The proofs of this lemma and the next follow after Lemma~\ref{lm:pchii}.
\begin{lm}\label{lm:geodunit_pos}
Set $\a=(\a_1,\ldots,\a_k)$ and $\a'=(\a_2,\ldots,\a_{k}).$
Whenever applicable,
\begin{align*}
\langle\q_{k-1,l}(\a';\gamma),\a_1\rangle
=&
\sum_{i = 1}^k (-1)^{n + \sum_{j = 2}^i (|\alpha_j|+1)}
\langle\q_{k,l;1_i,2}(\a_2,\ldots,\a_{i},1,\a_{i+1},\ldots,\a_k;\gamma),\alpha_1\rangle\\
=&
\sum_{i = 1}^k (-1)^{n + \sum_{j = 2}^i (|\alpha_j|+1)}
\langle\q_{k,l;2_{i,0},1}(\a_2,\ldots,\a_{i},1,\a_{i+1},\ldots,\a_k;\gamma),\a_1\rangle.
\end{align*}
\end{lm}

Let $\chi_i:\M_{k+1,l}(\beta)\to \C$ be the
cross ratio map such that the condition $\chi_i(u)\in [0,1]$ constrains the marked points $v_1=z_i, v_2, v_3,$ to lie on a geodesic in that order.
Let
\[
p_i:\M_{k+1,l;\chi_i}(\beta)\lrarr \M_{k,l}(\beta)
\]
denote the forgetful map given by forgetting $z_i$,
shifting the labels of $z_{i+1},\ldots,z_k,$ down by one,
and stabilizing the resulting open stable map.
\begin{lm}\label{lm:pchii}
The forgetful map $p_i$ restricts to a diffeomorphism from the irreducible locus to an open subset of the irreducible locus
that changes orientation by $sgn(p_i)=n+i+1$.
\end{lm}

\begin{figure}[ht]
\centering
\includegraphics[width=0.75\textwidth]{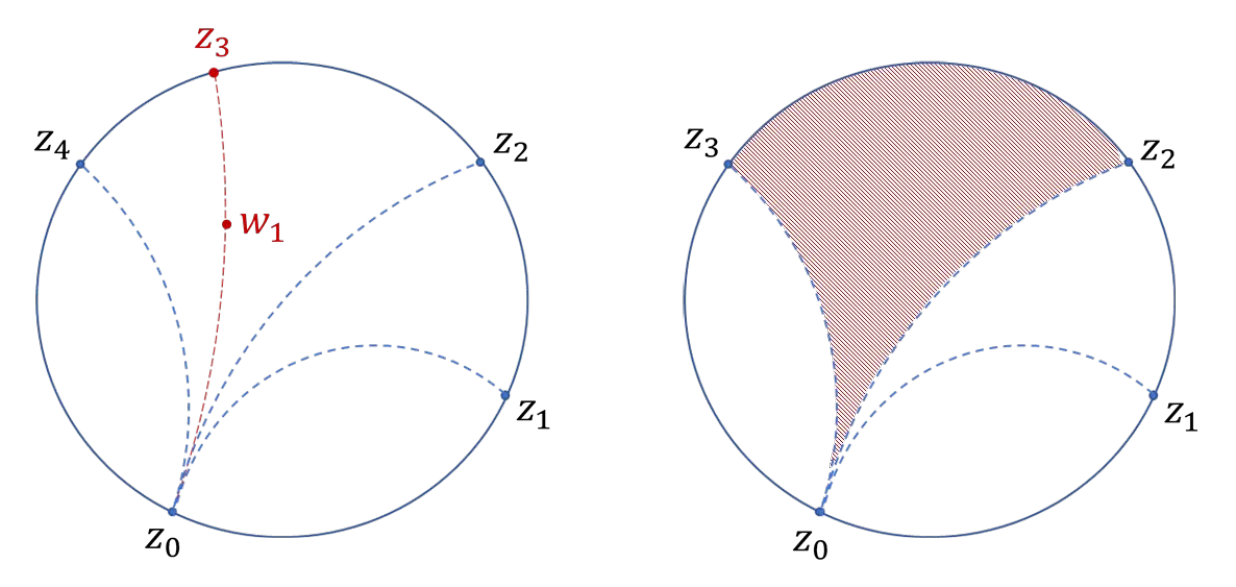}
\caption{On the left is the domain of an irreducible stable map in $\M_{5,1;2_{3,0},1}(\beta).$
On the right is the domain of an irreducible stable map in the image of the forgetful map
$p_3:\M_{5,1;2_{3,0},1}(\beta)\to \M_{4,1}(\beta)$ with the possible locations of $w_1$ shaded. The dotted lines are geodesics.}\label{pic:gunit}
\end{figure}

\begin{proof}[Proof of Lemma~\ref{lm:geodunit-1}]
Let $p=p_0:\M_{1,l;1_0,2}(\beta)\lrarr \M_{0,l}(\beta)$ be the forgetful map as in Lemma~\ref{lm:pchii}. Denote by $evi^1_j$ the evaluation maps at interior marked points on $\M_{1,l;1_0,2}(\beta)$, and by $evi^0_j$ the evaluation maps at interior marked points on $\M_{0,l}(\beta)$. In particular,
$
evi_j^0\circ p =evi_{j+1}^{1}.
$
Set
\[
\xi':=\bigwedge_{j=1}^{l}(evi_j^{0})^*\gamma_{j+1}
\qquad
\mbox{ and }
\qquad
\xi:=\bigwedge_{j=1}^l(evi_j^{1})^*\gamma_j
=p^*\xi'.
\]
For any space, denote by $pt$ the map from it to a point.
In the following calculation we use the fact that
\[
pt_*(p^*\xi')=pt_*p_*(p^*\xi'\wedge 1)
=pt_*(\xi'\wedge p_*1)=(-1)^{sgn(p)}pt_*\xi'\quad \in R.
\]
So,
\begin{align*}
\langle\q_{0,l;1_0,2}(\gamma),1\rangle
&=
pt_*(\q_{0,l;1_0,2}(\gamma))\\
&=
(-1)^{\varepsilon_\chi(\emptyset)} pt_*(evb_0)_*\xi\\
&=
(-1)^{\varepsilon_\chi(\emptyset)}
\int_{\M_{1,l;1_0,2}(\beta)}p^*\xi'\\
&=
(-1)^{\varepsilon_\chi(\emptyset)+sgn(p)}
\int_{\M_{0,l}(\beta)}\xi'\\
&=
(-1)^{\varepsilon_\chi(\emptyset)+sgn(p)}
\int_{\M_{0,l}(\beta)} \bigwedge_{j=0}^{l-1}(evi_j^{0})^*\gamma_{j+1}\\
&=
(-1)^{\varepsilon_\chi(\emptyset)+sgn(p)}
\q_{-1,l}(\gamma).
\end{align*}
The total sign is therefore
\[
\varepsilon_\chi(\emptyset)+sgn(p)
=
1+n+1\equiv n\pmod 2.
\]
\end{proof}

\begin{proof}[Proof of Lemma~\ref{lm:geodunit_pos}]

Let $evi_j, evb_j,$ be the evaluation maps on $\M_{k,l}(\beta)$, and set
\[
\xi':=\bigwedge_{j=1}^levi_j^*\gamma_j\wedge \bigwedge_{j=0}^{k-1}(evb_{j})^*\a_{j+1}.
\]
Then
\begin{align}\label{eq:ptxip}
\begin{split}
pt_*\xi'
=&
pt_*(evb_0)_*\big(
\bigwedge_{j=1}^levi_j^*\gamma_j\wedge (evb_0)^*\a_1\wedge\bigwedge_{j=1}^{k-1}(evb_j)^*\a_{j+1}
\big)\\
=&
(-1)^{|\a_1||\gamma|}
pt_*(evb_0)_*\big(
(evb_0)^*\a_1\wedge
\bigwedge_{j=1}^levi_j^*\gamma_j\wedge \bigwedge_{j=1}^{k-1}(evb_j)^*\a_{j+1}
\big)\\
=&
(-1)^{|\a_1||\gamma|}
pt_*\Big(
\a_1\wedge
(evb_0)_*\big(\bigwedge_{j=1}^levi_j^*\gamma_j\wedge \bigwedge_{j=1}^{k-1}(evb_j)^*\a_{j+1}
\big)\Big)\\
=&
(-1)^{|\a_1||\gamma|+ |\a_1|(|\gamma|+|\a'|-\rdim(evb_0))}
pt_*\Big(
(evb_0)_*\big(\bigwedge_{j=1}^levi_j^*\gamma_j\wedge \bigwedge_{j=1}^{k-1}(evb_j)^*\a_{j+1}
\big)\wedge\a_1\Big)\\
=&
(-1)^{|\a_1|(|\a'|+k-1)+|\a_1|+\varepsilon(\a')}
\langle\q_{k-1,l}(\a';\gamma),\a_1\rangle.
\end{split}
\end{align}

Let $p_i:\M_{k+1,l;\chi_i}(\beta)\to \M_{k,l}(\beta)$ be the forgetful map from Lemma~\ref{lm:pchii} with $v_2,v_3,$ taken to be $w_1,w_2,$ for the first equation, or $w_1,z_0,$ for the second equation.
Varying $i,$ we get a diffeomorphism onto an open dense subset of the irreducible stratum:
\[
p = \coprod_{1 \leq i \leq k} p_i :\coprod_{1\le i\le k}\M_{k+1,l;\chi_i}(\beta)\lrarr \M_{k,l}(\beta).
\]
See Figure~\ref{pic:gunit}.

Let $evi_j, evb_j,$ be the evaluation maps on $\M_{k,l}(\beta)$, and let $evi_j^i, evb_j^i,$ be the evaluation maps on $\M_{k+1,l;\chi_i}(\beta).$
In particular,
\[
evb_j^i=
\begin{cases}
evb_j\circ p_i, & j< i,\\
evb_{j-1}\circ p_i, & j> i.
\end{cases}
\]
Set
\[
\xi_i:=p_i^*\xi'.
\]
Then, by Proposition~\ref{prop:iof}\ref{prop:pushpull},
\begin{equation}\label{eq:pxipxip}
pt_*\xi'
=\sum_{i=1}^k (-1)^{sgn(p_i)} pt_*(\xi'\wedge (p_i)_*1)
=\sum_{i=1}^k (-1)^{sgn(p_i)} pt_*(p_i)_*(p_i^*\xi')
=\sum_{i=1}^k (-1)^{sgn(p_i)} pt_*\xi_i.
\end{equation}

By Proposition~\ref{prop:iof}\ref{prop:pushcomp}-~\ref{prop:pushpull}
we have
\begin{align*}
pt_*\xi_i
=&
pt_*\big(
\bigwedge_{j=1}^levi_j^*\gamma_j\wedge
\bigwedge_{j=1}^{i}(evb^i_{j-1})^*\a_{j}
\wedge
(evb_i^i)^*1
\wedge
\bigwedge_{j=i+1}^{k}(evb^i_j)^*\a_j
\big)\\
=&
(-1)^{|\a_1||\gamma|}
pt_*(evb_0)_*\big(evb_0^*\a_1\wedge
\bigwedge_{j=1}^levi_j^*\gamma_j\wedge
\bigwedge_{j=2}^{i}(evb^i_{j-1})^*\a_{j}
\wedge
(evb_i^i)^*1
\wedge
\bigwedge_{j=i+1}^{k}(evb^i_j)^*\a_j
\big)\\
=&
(-1)^{|\a_1|(|\a'|+\rdim evb_0)}\cdot\\
&\quad\cdot
pt_*\big((evb_0)_*\big(
\bigwedge_{j=1}^levi_j^*\gamma_j\wedge
\bigwedge_{j=2}^{i}(evb^i_{j-1})^*\a_{j}
\wedge
(evb_i^i)^*1
\wedge
\bigwedge_{j=i+1}^{k}(evb^i_j)^*\a_j
\big)\wedge \a_1\big)\\
=&
(-1)^{\varepsilon_{\chi}(\a_2,\ldots,\a_i,1,\a_{i+1},\ldots,\a_k) +|\a_1|
+|\a_1|(|\a'|+k-1)}
\langle\q_{k,l;\chi_i}(\a_2,\ldots,\a_i,1,\a_{i+1},\ldots,\a_k), \a_1\rangle.
\end{align*}
Thus, using equations~\eqref{eq:ptxip},~\eqref{eq:pxipxip}, and Lemma~\ref{lm:pchii}, we obtain
\[
\langle\q_{k-1,l}(\a';\gamma),\a_1\rangle = \sum_{i = 1}^k (-1)^* \langle\q_{k,l;\chi_i}(\a_2,\ldots,\a_i,1,\a_{i+1},\ldots,\a_k), \a_1\rangle
\]
with
\begin{align*}
* &\equiv \varepsilon_{\chi}(\a_2,\ldots,\a_i,1,\a_{i+1},\ldots,\a_k) +\varepsilon(\a') + n+i+1 \\
& = |\alpha'| + k + \varepsilon(\a_2,\ldots,\a_i,1,\a_{i+1},\ldots,\a_k) + \varepsilon(\a') + n + i + 1 \\
& \equiv |\alpha'| + k +  i + \sum_{j = i + 1}^k (|\alpha_j| + 1) + n + i + 1 \\
& \equiv \sum_{j = 2}^i (|\alpha_j| + 1) + n \pmod 2.
\end{align*}
\end{proof}

\subsubsection{Reversing the geodesic}
We consider the effect on the geodesic operators of reversing the order of the marked points constrained to the geodesic.
\begin{lm}\label{lm:geoddirection}
For all lists $\a,\gamma,$ and elements $\eta_1,\eta_2,\eta_3,$ we have
\begin{equation*}
\q_{k,l;2_{a,e},1}(\a;\gamma)
=
-\q_{k,l;2_{e,a},1}(\a;\gamma),
\end{equation*}
and
\begin{multline*}
\q_{k,l+3;0,3}(\a; \eta_1\otimes\eta_2\otimes\eta_3\otimes\otimes_{j=1}^l\gamma_j)
= \\
= (-1)^{1+|\eta_1||\eta_2|+|\eta_1||\eta_3|+|\eta_2||\eta_3|} \q_{k,l+3;0,3}(\a; \eta_3\otimes\eta_2\otimes\eta_1\otimes\otimes_{j=1}^l\gamma_j).
\end{multline*}
\end{lm}

\begin{proof}
Let $\chi,\chi':\M_{k+1,l}(\beta)\to \C$ be the
cross ratio maps such that the condition that $\chi(u)\in [0,1]$ (resp. $\chi'(u)\in [0,1]$) constrains the marked points $v_1, v_2, v_3,$ (resp. $v_3,v_2,v_1$) to lie on a geodesic in that order.
Note that
\[
\chi'=(\bar v_2, v_3,v_1,v_2) = 1- (\bar v_2,v_1,v_3,v_2)=1-\chi.
\]
Then we have the following diagram, where the front and back are pullback diagrams.
\begin{equation*}
\xymatrix{
\M_{k+1,l;\chi}(\beta)\ar[rr]\ar[dd]\ar@{-->}[dr]^\varphi & & \M_{k+1,l}(\beta)\ar'[d]^{\chi}[dd]\ar[dr]^{\Id} &\\
& \M_{k+1,l;\chi'}(\beta) \ar[rr]\ar[dd] & & \M_{k+1,l}(\beta)\ar[dd]^{\chi'} \\
[0,1]
\ar@{^{(}->}'[r]^{}[rr] \ar[dr]^{\theta\mapsto 1-\theta}& & \C \ar[dr]^{z\mapsto 1-z}& \\
&[0,1]\ar@{^{(}->}[rr]& &\C .\\
}
\end{equation*}
By~\cite[Lemma 8.2.3(4)]{FOOO}, the induced map $\varphi$ is an orientation-reversing isomorphism.
The result follows from the definition of the geodesic $\q$ operators.
\end{proof}

\subsubsection{Unit}
The following is a direct analog of the unit property of the usual $\q$ operators.

\begin{lm}\label{lm:unitgeod}
Let $\alpha_1,\ldots,\alpha_{k}\in A^*(L;R),$ and $\gamma_1,\ldots,\gamma_l \in A^*(X;Q).$
If $z_i$ is not constrained to lie on the geodesic, then
\[
\q_{\,k+1\!,l;a,e}(\alpha_1,\ldots,\alpha_{i-1},1,\alpha_{i},\ldots,\alpha_k ;\otimes_{r=1}^l\gamma_r)= 0.
\]
Moreover,
if $z_0$ is not constrained to lie on a geodesic, then
\[
\langle\q_{k,l;a,e}(\otimes_{j=1}^k\alpha_j;\otimes_{j=1}^l\gamma_j),1\rangle =0.
\]
\end{lm}
The proof is similar to that of Propositions 3.2 and 3.12 in~\cite{ST1}.

\subsubsection{Symmetry}

\begin{lm}\label{cl:gsymmetry}
Let $k\ge -1$ and let $\chi$ be a geodesic condition. For any permutation $\sigma\in S_l,$
\[
\q_{k,l;\chi}(\alpha_1,\ldots,\alpha_k;\gamma_1,\ldots,\gamma_l)=
(-1)^{s_\sigma(\gamma)}\q_{k,l;\chi_\sigma}(\alpha_1,\ldots,\alpha_k;\gamma_{\sigma(1)},\ldots,\gamma_{\sigma(l)}),
\]
where $s_\sigma(\gamma)$ is defined by equation~\eqref{eq:sgnsigmagamma}, and $\chi_\sigma$ is given by $\chi$ composed with the diffeomorphism of $\M_{k+1,l}(\beta)$ induced by $\sigma^{-1}$ acting on the labels of interior marked points.
\end{lm}
The proof is similar to that of Proposition 3.6 in \cite{ST1}.

\subsubsection{Cyclic symmetry}
The following is a direct analog of the cyclic symmetry property of the usual $\q$ operators.
\begin{lm}\label{cl:cyc}
Let $\chi$ be a cross ratio map, and let $\chi'$ be the map obtained from $\chi$ by shifting boundary indices up by one, modulo $k+1$.
For all $\alpha_1,\ldots,\alpha_{k+1}\in A^*(L;R)$ and $\gamma_1,\ldots,\gamma_l\in A^*(X;Q)$, we have
\begin{align*}
\langle\q_{k,l;\chi}(\alpha_1,\ldots&,\alpha_k;\gamma_1,\ldots\gamma_l),\alpha_{k+1}\rangle=\\
&(-1)^{(|\alpha_{k+1}|+1)\cdot\sum_{j=1}^k(|\alpha_j|+1)}\cdot
\langle \q_{k,l;\chi'}(\alpha_{k+1},\alpha_1,\ldots,\alpha_{k-1};\gamma_1,\ldots,\gamma_l),\alpha_k\rangle.
\end{align*}
\end{lm}
The proof is similar to that of Proposition 3.3 in~\cite{ST1}.

For a list $\a=(\a_1,\ldots,\a_k)$ and a cyclic permutation $\sigma\in \Z/k\Z$, denote by $\a_\sigma$ the list $(\a_{\sigma(1)},\ldots,\a_{\sigma(k)})$.

\begin{cor}\label{cor:geodunit}
Set $\a=(\a_1,\ldots,\a_k)$ and $\a'=(\a_2,\ldots,\a_{k}).$
Whenever applicable,
\begin{align*}
\langle\q_{k-1,l}(\a';\gamma),\a_1\rangle
=&
\sum_{\sigma\in \Z/k\Z} (-1)^{n + |\alpha|+k +(|\a_1|+1)\cdot (| \alpha'|+k-1) +  s_{\sigma}^{[1]}(\a)}
\langle\q_{k,l;1_0,2}(\a_\sigma;\gamma),1\rangle\\
=&
\sum_{\sigma\in\Z/k\Z} (-1)^{n + |\alpha|+k +(|\a_1|+1)\cdot (| \alpha'|+k-1) +  s_{\sigma}^{[1]}(\a)}
\langle \q_{k,l;2_{0,\sigma^{-1}(1)},1} (\a_{\sigma};\gamma),1\rangle,
\end{align*}
with $s_{\sigma}^{[1]}(\a):=\sum_{\substack{j<m\\ \sigma(j)>\sigma(m)}}(|\a_{\sigma(j)}|+1) \cdot (|\a_{\sigma(m)}|+1)$.
\end{cor}
\begin{proof}
Apply Lemma~\ref{cl:cyc} repeatedly to the expressions in Lemma~\ref{lm:geodunit_pos} to move the input $1$ to the right hand side of the pairing $\langle \cdot,\cdot \rangle.$ Moving $\alpha_{i+1},\ldots,\alpha_k,\alpha_1,$ past $1$ contributes $(|\alpha_1|+1) + \sum_{j = i+1}^k(|\a_j|+1)$ to the sign, which combines with $\sum_{j = 1}^i(|\a_j|+1)$ in the sign of Lemma~\ref{lm:geodunit_pos} to give $|\a|+k.$ Moving $\a_1$ past $\a'$ contributes $(|\a_1|+1) \cdot(|\a'|+k-1)$ to the sign. The remaining cyclic permutation of $\a$ contributes $s_\sigma^{[1]}(\a).$
\end{proof}

\subsection{Deformed \texorpdfstring{$\q$}{q} operators}
Let $\gamma \in \mI A^*(X;Q)$ such that $|\gamma|=2$ and $d\gamma=0$. Let $b\in \mJ A^*(L;R)$ such that $|b|=1$. Define
\[
\qbg_{-1,l}(\otimes_{j=1}^l\eta_j):=
\sum_{k,t}\frac{1}{t!(k+1)}
\langle\q_{k,l+t}(b^{\otimes k};\otimes_{j=1}^l\eta_j\otimes \gamma^{\otimes t}),b\rangle+\sum_t\frac{1}{t!}\q_{-1,t+l}(\otimes_{j=1}^l\eta_j\otimes\gamma^{\otimes t}).
\]
For $k\ge 0,$ define
\begin{multline*}
\qbg_{k,l}(\alpha_1,\ldots,\alpha_k;\delta_1,\ldots,\delta_l)=\\
\sum_{s,t}\frac{1}{(t-l)!}\sum_{\substack{1\le i_1<\cdots\;\\ \;\cdots<i_k\le s}}\!
\q_{s,t}(b^{\otimes i_1-1}\otimes\alpha_1\otimes b^{\otimes i_2-i_1-1}\otimes \cdots\otimes b^{\otimes i_k-i_{k-1}-1}\otimes\alpha_k\otimes b^{\otimes s-i_k};\otimes_{j=1}^l\delta_j\otimes\gamma^{\otimes t-l}).
\end{multline*}
Define also
\[
\qg_{\emptyset,l}(\otimes_{j=1}^l\eta_j)=
\sum_t\frac{1}{t!}\q_{\emptyset,l+t}(\otimes_{j=1}^l\eta_j\otimes\gamma^{\otimes t}).
\]
For $j>0$, the deformed geodesic $\q$ operators are given by
\begin{multline*}
\qbg_{k,l;1_j,2}(\alpha_1,\ldots,\alpha_k;\delta_1,\ldots,\delta_l)=\\
=\sum_{s,t}\frac{1}{(t-l)!}\!\sum_{1\le i_1<\cdots<i_k\le s}\!
\q_{s,t;1_{i_j},2}(b^{\otimes i_1-1}\otimes\alpha_1\otimes b^{\otimes i_2-i_1-1}\otimes \cdots\\
\cdots\otimes b^{\otimes i_k-i_{k-1}-1}\otimes\alpha_k\otimes b^{\otimes s-i_k};\otimes_{j=1}^l\delta_j\otimes\gamma^{\otimes t-l}),
\end{multline*}
\begin{multline*}
\qbg_{k,l;2_{0,j},1}(\alpha_1,\ldots,\alpha_k;\delta_1,\ldots,\delta_l)=\\
=\sum_{s,t}\frac{1}{(t-l)!}\sum_{1\le i_1<\cdots<i_k\le s}
\q_{s,t;2_{0,i_j},1}(b^{\otimes i_1-1}\otimes\alpha_1\otimes b^{\otimes i_2-i_1-1}\otimes \cdots\\
\cdots\otimes b^{\otimes i_k-i_{k-1}-1}\otimes\alpha_k\otimes b^{\otimes s-i_k};\otimes_{j=1}^l\delta_j\otimes\gamma^{\otimes t-l}),
\end{multline*}
and similarly for other geodesic operations.

Let $\gt \in \mI A^*(I\times X;Q)$ such that $|\gt|=2$ and $d\gt=0$. Let $\bt\in \mJ A^*(I\times L;R)$ such that $|\bt|=1$. Define deformed operations on the product,
\[
\qtbg_{k,l;a,e}:A^*(I\times L;R)^{\otimes k}\otimes A^*(I\times X;\Qh)^{\otimes l}\lrarr A^*(I\times L;R)
\]
for $k\ge 0$ and
\[
\qtbg_{-1,l;a,e}:A^*(I\times X;\Qh)^{\otimes l}\lrarr A^*(I;R),
\]
by the formulas above with $\qt$ instead of $\q$.

\begin{lm}\label{lm:deformprop}
The structure equations and all properties of Sections~\ref{ssec:prop} and~\ref{ssec:isot} hold for $\qbg_{k,l}$ and $\qtbg_{k,l}$ for all $k\ge -1, l\ge 0$.
The structure equations and all properties of Sections~\ref{ssec:geo_a_infty} and~\ref{ssec:prop_geod} hold for $\qbg_{k,l;a,e}$ and $\qtbg_{k,l;a,e}$ for all $k\ge -1, l\ge 0$.
\end{lm}
From now on, we may implicitly use Lemma~\ref{lm:deformprop}, referring to the usual properties when working with deformed operators.

\section{Tensor potential and relative potential}

\subsection{Tensor potential}\label{ssec:mem}

Let $(\gamma,b)\in \mI_W A^*(X;\Qh_W)\oplus \mJ_W A^*(L;R_W)$ be a bounding pair as in Definition~\ref{dfn:bdpair}.
Recall that in equation~\eqref{eq:memhdfn1} we have defined the operator
\[
\memh: \cone\lrarr \cone
\]
by
\begin{equation}\label{eq:memhdfn}
\memh(\eta,\xi)=(\qg_{\emptyset,1}(\eta),
(-1)^{n+1}\qbg_{-1,1}(\eta) - c\cdot\xi).
\end{equation}

\begin{lm}\label{lm:ceq}
For any $\eta\in A^*(X;\Qh_W),$
we have
\[
(-1)^{n+1}\qbg_{-1,1}(d\eta) -c\cdot \i(\eta)=
\i(\qg_{\emptyset,1}(\eta)).
\]
\end{lm}
\begin{proof}
By Proposition~\ref{q-1_rel},
\begin{align*}
\qbg_{-1,1}(d\eta)
&=
\frac{1}{2}\big( (-1)^{|\eta|}\langle\qbg_{0,1}(\eta),\qbg_{0,0}\rangle+
\langle\qbg_{0,0},\qbg_{0,1}(\eta)\rangle\big) +(-1)^{|\eta|+1}\int_Li^*\qg_{\emptyset,1}(\eta)\\
&=
(-1)^{|\eta|}\langle\qbg_{0,1}(\eta),\qbg_{0,0}\rangle +(-1)^{|\eta|+1}\int_Li^*\qg_{\emptyset,1}(\eta)\\
&=
(-1)^{|\eta|}c\cdot\langle \qbg_{0,1}(\eta),1\rangle +(-1)^{|\eta|+1}\int_Li^*\qg_{\emptyset,1}(\eta)\\
&=
(-1)^{|\eta|}c\cdot\int_L\qbg_{0,1}(\eta) +(-1)^{|\eta|+1}\int_Li^*\qg_{\emptyset,1}(\eta).\\
\shortintertext{By Lemmas~\ref{no_top_deg}~ and~\ref{lm:qzero}, }
&=
(-1)^{|\eta|+1}c\cdot\int_L i^*\eta  +(-1)^{|\eta|+1}\int_Li^*\qg_{\emptyset,1}(\eta).
\end{align*}
This proves
\[
(-1)^{|\eta|+1}\qbg_{-1,1}(d\eta) -c\cdot \int_L i^*\eta=
\int_Li^*\qg_{\emptyset,1}(\eta).
\]
Equivalently,
\[
(-1)^{n+1}\qbg_{-1,1}(d\eta) -c\cdot \i(\eta)=
\i(\qg_{\emptyset,1}(\eta)).
\]
\end{proof}

\begin{cor}\label{cor:cinQ}
For closed $\eta,$ we have \:
$c\cdot\int_Li^*\eta = -\int_Li^*\qg_{\emptyset,1}(\eta).$
\end{cor}

\begin{lm}
For all $\a\in A^*(L;R_W), \eta\in A^*(X;R_W)$,
we have
\[
\langle \eta,i_*\alpha\rangle_X
=(-1)^{n}\langle i^*\eta,\a\rangle.
\]
\end{lm}

\begin{proof}
For any space, denote by $pt$ the map from it to a point.
By Proposition~\ref{prop:iof}~\ref{prop:pushpull}, we have
\begin{multline*}
\langle i^*\eta,\a\rangle=
(-1)^{|\a|}pt_*(i^*\eta\wedge\a)
=(-1)^{|\a|}pt_*i_*(i^*\eta\wedge\a)=\\
=(-1)^{|\a|}pt_*(\eta\wedge i_*\a)
=(-1)^{|\a|+|\a|-n}\langle \eta,i_*\a\rangle_X.
\end{multline*}
\end{proof}

\begin{proof}[Proof of Theorem~\ref{prop:formforc}]
By assumption, $\int_Li^*\eta=1$. Therefore, by Corollary~\ref{cor:cinQ}, we have
\begin{align*}
c=&-\int_L i^*\qg_{\emptyset,1}(\eta)
= -\langle i^*\qg_{\emptyset,1}(\eta),1\rangle
= (-1)^{n+1}\langle\qg_{\emptyset,1}(\eta), i_*1\rangle_X\\
=&
\sum_{\substack{l\ge 0\\ \beta\in H_2(X;\Z)}}
(-1)^{n+1+w_\s(\beta)}T^{\varpi(\beta)}
\frac{1}{l!}
GW_{\beta}(\eta, i_*1, \gamma^{\otimes l}).
\end{align*}
\end{proof}

\begin{lm}\label{lm:nunchain}
The map $\memh$ is a chain map.
\end{lm}

\begin{proof}
We need to show that $\memh$ commutes with $\dcone$. That is,
\[
\dcone(\memh(\eta,\xi))=\memh(d\eta,\i(\eta))
\]
or, in other words,
\[
(d\qg_{\emptyset,1}(\eta),\i(\qg_{\emptyset,1}(\eta)))
=
(\qg_{\emptyset,1}(d\eta),(-1)^{n+1}\qbg_{-1,1}(d\eta)-c\cdot \i(\eta)).
\]
Indeed, equality of the first component follows from Lemma~\ref{lm:qecm}.
Equality of the second component follows from Lemma~\ref{lm:ceq}.
\end{proof}

We denote by $\nund$ the induced map on cohomology,
\[
\nund:H^*(\cone)\lrarr H^*(\cone).
\]

\begin{lm}\label{lm:sfp}
Consider the fiber product
\begin{equation*}
\xymatrix{
{L\prescript{}{i}\times_{p_X} (I\times X)}\ar[r]\ar[d] &
{I\times X}\ar[d]^{p_X}\\
{L}\ar[r]^i& X.
}
\end{equation*}
The natural diffeomorphism
\[
L\prescript{}{i}\times_{p_X} (I\times X)\lrarr I\times L
\]
has sign $(-1)^n$.
\end{lm}
\begin{proof}
By~\cite[Lemma 8.2.3(3)]{FOOO},
\begin{align*}
L\prescript{}{i}\times_{p_X} (I\times X)
=&
L\prescript{}{i\times pt}\times_{\Id_X\times pt} (X\times I)
=
(L\prescript{}{i}\times_{\Id_X} X) \prescript{}{pt}\times_{pt} I
=
L\times I
=
(-1)^n I\times L.
\end{align*}
\end{proof}

\begin{lm}\label{lm:fibprodsgn}
Consider the diagram
\begin{equation*}
\xymatrix{
I\times L \ar[r]^{ i}\ar[d]^{p_L}&
{I\times X}\ar[d]^{p_X}\\
{L}\ar[r]^i& X.
}
\end{equation*}
Let $\alpha \in A^*(I\times X).$ Then
\[
(p_L)_*i^* \alpha = (-1)^ni^*(p_X)_* \alpha.
\]
\end{lm}
\begin{proof}
This is an immediate consequence of Lemma~\ref{lm:sfp} and Proposition~\ref{prop:iof}~\ref{prop:pushfiberprod}.
\end{proof}

\begin{lm}\label{lm:memhinvt}
If $(\gamma,b)$ is gauge equivalent to $(\gamma',b')$, then $\memh^{\gamma,b}$ and $\memh^{\gamma',b'}$ are chain homotopic.
\end{lm}

\begin{proof}
For each $\eta$, set $\etat:=p_X^*\eta$. Define
\[
H:\cone\lrarr \cone
\]
by
\[
H(\eta,\xi)
=(-1)^{|\eta|+1}\cdot((p_X)_*\qt^{\gt}_{\emptyset,1}(\etat), pt_*(\qt^{\gt,\bt}_{-1,1}(\etat))).
\]
We show
\begin{equation}\label{eq:ch}
\dcone\circ H+H \circ \dcone = \memh^{\gamma',b'}-\memh^{\gamma,b}.
\end{equation}
On the one hand,
\begin{align*}
(\dcone\circ H+&H \circ \dcone)(\eta,\xi)=\\
&=
\dcone((-1)^{|\eta|+1}(p_X)_*\qt^{\gt}_{\emptyset,1}(\etat), (-1)^{|\eta|+1} pt_*(\qt^{\gt,\bt}_{-1,1}(\etat)))
+H(d\eta,\i(\eta))\\
&=
((-1)^{|\eta|+1}d(p_X)_*\qt^{\gt}_{\emptyset,1}(\etat), (-1)^{|\eta|+1}\i((p_X)_*\qt^{\gt}_{\emptyset,1}(\etat)))+\\
&\qquad+((-1)^{|\eta|}(p_X)_*\qt^{\gt}_{\emptyset,1}(d\etat), (-1)^{|\eta|} pt_*(\qt^{\gt,\bt}_{-1,1}(d\etat)))\\
&=
((-1)^{|\eta|+1}(d(p_X)_*\qt^{\gt}_{\emptyset,1}(\etat) -(p_X)_*\qt^{\gt}_{\emptyset,1}(d\etat)),\\
&\hspace{10em}
(-1)^{|\eta|+1}\i((p_X)_*\qt^{\gt}_{\emptyset,1}(\etat)) + (-1)^{|\eta|} pt_*(\qt^{\gt,\bt}_{-1,1}(d\etat))).
\end{align*}
On the other hand,
\begin{align}
(\memh^{\gamma',b'}&-\memh^{\gamma,b})(\eta,\xi)
=
(\q^{\gamma'}_{\emptyset,1}(\eta), (-1)^{n+1}\q^{\gamma',b'}_{-1,1}(\eta)-c\cdot\xi)-
(\qg_{\emptyset,1}(\eta), (-1)^{n+1}\qbg_{-1,1}(\eta) -c\cdot\xi)\notag\\
&=
(\q^{\gamma'}_{\emptyset,1}(\eta)-\qg_{\emptyset,1}(\eta),
(-1)^{n+1} \big(\q^{\gamma',b'}_{-1,1}(\eta)-\qbg_{-1,1}(\eta)\big)).\label{eq:rhsch}
\end{align}
By Lemma~\ref{lm:qtchain} and Stokes theorem, Proposition~\ref{stokes}, applied to $p_X$, we have
\[
(-1)^{|\eta|+1}(d(p_X)_*\qt^{\gt}_{\emptyset,1}(\etat) -(p_X)_*\qt^{\gt}_{\emptyset,1}(d\etat)) =
(j_1^*-j_0^*) (\qt^{\gt}_{\emptyset,1}(\etat)) =\q^{\gamma'}_{\emptyset,1}(\eta)-\qg_{\emptyset,1}(\eta),
\]
which proves the first component of equation~\eqref{eq:ch}.

To prove the second component of equation~\eqref{eq:ch}, we proceed as follows. By Stokes' theorem and Proposition~\ref{cl:qt_-1},
\begin{align}\label{eq:memdiff}
(-1)^{n+1}\big(\q^{\gamma',b'}_{-1,1}(\eta)-\qbg_{-1,1}(\eta)\big) &= (-1)^{|\etat|}pt_*(d\qt_{-1,1}^{\gt}(\etat)) \\
 &= (-1)^{|\etat|} pt_*
\qt^{\gt}_{-1,1}(d\etat) + pt_* (p_I)_* i^* \qt_{\emptyset,1}(\etat) + \notag\\
& \qquad  +\frac{(-1)^{|\etat|+1}}{2}pt_* \big(
(-1)^{|\eta|}
\ll\qt^{\gt}_{0,1}(\etat),\qt^{\gt}_{0,0}\gg
+\ll\qt^{\gt}_{0,0},\qt^{\gt}_{0,1}(\etat)\gg\big) . \notag
\end{align}
By the symmetry of the cyclic structure,
\begin{align*}
\frac{1}{2}\cdot \Big((-1)^{|\eta|}\ll\qt^{\gt}_{0,1}(\etat)&, \qt^{\gt}_{0,0}\gg
+\ll\qt^{\gt}_{0,0},\qt^{\gt}_{0,1}(\etat)\gg\Big)=\\
&=
\frac{1}{2}\cdot\left((-1)^{|\eta|} \ll\qt^{\gt}_{0,1}(\etat),\qt^{\gt}_{0,0}\gg
+(-1)^{|\eta|}
\ll\qt^{\gt}_{0,1}(\etat),\qt^{\gt}_{0,0}\gg\right)\\
&=
(-1)^{|\eta|} \ll\qt^{\gt}_{0,1}(\etat),c\cdot 1\gg\\
&=
(-1)^{|\eta|}c\cdot \ll\qt^{\gt}_{0,1}(\etat), 1\gg\\
&=
(-1)^{|\eta|+1}c\cdot (p_I)_*i^*\etat,
\end{align*}
where the last equality is by Lemma~\ref{lm:mt0}.
Furthermore, by Proposition~\ref{prop:iof}~\ref{prop:pushpull},
\begin{align*}
pt_*(p_I)_*i^*\etat
=&
pt_*(p_I)_*i^*p_X^*\eta\\
=&
pt_*p_L^*i^*\eta\\
=&
pt_*(p_L)_*(p_L^*i^*\eta\wedge 1)\\
=&
pt_*(i^*\eta\wedge (p_L)_*1)
=
0.
\end{align*}
By Lemma~\ref{lm:fibprodsgn},
\begin{align*}
pt_*i^* \qt_{\emptyset,1}(\etat)
&=
pt_*(p_L)_*i^* \qt_{\emptyset,1}(\etat)\\
&=
(-1)^{n}pt_*i^*(p_X)_*(\qt_{\emptyset,1}(\etat))\\
&=
(-1)^{|\eta|+1}\i((p_X)_*(\qt_{\emptyset,1}(\etat))).
\end{align*}
Plugging the preceding calculations into~\eqref{eq:memdiff} above, we get
\[
(-1)^{n+1}\big(\q^{\gamma',b'}_{-1,1}(\eta)-\qbg_{-1,1}(\eta)\big) =
(-1)^{|\etat|} pt_*
\qt^{\gt}_{-1,1}(d\etat) +(-1)^{|\eta|+1}\i((p_X)_*(\qt_{\emptyset,1}(\etat))),
\]
which gives the second component of~\eqref{eq:ch}.
\end{proof}

\begin{proof}[Proof of Theorem~\ref{thm:mp}]
The first part is given by Lemma~\ref{lm:nunchain}. The second part is given by Lemma~\ref{lm:memhinvt}.
\end{proof}

\subsection{Flatness relation}\label{ssec:flatness}

The objective of this section is to prove Theorem~\ref{lm:assoc}.
We begin by explaining the notation in greater detail.

Let $W$ and $S$ be graded real vector spaces.
In this section, it will be useful to note that
elements of $W$ and $S$ define derivations on $\cone=A^*(X;\Qh_W)\oplus R_W[-n-1]$, as follows.
For $u\in W$, the derivation $\d_u:\R[[W]]\to \R[[W]]$ induces a derivation on $\Qh_W=\L_c\otimes \R[[W]]$, which in turn induces a derivation on $A^*(X;\Qh_W)=A^*(X;\R)\otimes \Qh_W$. In addition, the derivation on $\Qh_W$ extends trivially to $R_W$. In total, we get
\[
\d_u:\cone\lrarr \cone.
\]
For $s\in S$, the derivation $\d_s$ extends trivially from $\R[[S]]$ to $R_W$, and acts trivially on $A^*(X;\Qh_W)$. This defines
\[
\d_s:\cone\lrarr \cone.
\]
It is immediate from definition that $\d_u,\d_s,$ are chain maps, and so descend to cohomology.
Furthermore,
for $u\in \Qh_W\otimes W$ of the form $u=r\otimes w$, we define $\d_u= r\cdot\d_w$,
and similarly for $u\in \Qh_W\otimes S$. Moreover, since $\d_s$ for $s\in S$ acts as zero on the first component of $\cone$, we can in fact define $\d_u$ for $u\in R_W\otimes S$, and it is a chain map.
In other words, there are chain map derivations
\[
\d_u: \cone\lrarr \cone,
\quad
\forall u\in (\Qh_W\otimes W)\oplus (R_W\otimes S).
\]
Finally, for a chain map
\[
\Theta: \cone\lrarr \cone,
\]
the derivative operator $\d_u\Theta$ is defined by
\[
(\d_u\Theta)(\eta)=\d_u(\Theta(\eta))-\Theta(\d_u(\eta)).
\]

\begin{proof}[Proof of Theorem~\ref{lm:assoc}]

Let $u,v \in(\Qh_W\otimes W)\oplus (R_W\otimes S)$.
Define
\[
H_{uv}:\cone\lrarr \cone
\]
by
\begin{multline*}
H_{uv}(\eta,\xi)=
((-1)^{1+|\eta|+|u|+|v|+|\eta||v|} \qg_{\emptyset,3;\chi_0}(\d_u\gamma,\eta,\d_v\gamma), \\
(-1)^{|\eta|+|u|+|v|+|\eta||v|} \qbg_{-1,3;0,3}(\d_u\gamma,\eta,\d_v\gamma)
+(-1)^{(|\eta|+1)(|u|+|v|+1)} \langle\qbg_{0,2;1_0,2}(\eta,\d_u\gamma),\d_vb\rangle+\\
+(-1)^{|\eta|+(|\eta|+1)(|u|+|v|)+|u||v|} \langle\qbg_{0,2;1_0,2}(\eta,\d_v\gamma),\d_ub\rangle+\\
+(-1)^{|\eta|+(|\eta|+1)(|u|+|v|)+|u||v|} \langle\qbg_{1,1;2_{0,1},1}(\d_vb;\eta),\d_ub\rangle
).
\end{multline*}
The four summands of the second component of $H_{uv}$ correspond to the four pictures in Figure~\ref{pic:H}. We show that
\begin{equation}\label{eq:memhhomotop}
\d_u\memh\circ \d_v\memh-(-1)^{|u||v|}\d_v\memh\circ\d_u\memh
=\dcone\circ H_{uv}+(-1)^{|u|+|v|}H_{uv}\circ \dcone.
\end{equation}
The heart of the proof is based on an analysis of the boundaries of the moduli spaces of stable disks with geodesic constraints shown in Figure~\ref{pic:H} using the geodesic structure equations. In preparation for this analysis, we expand both sides of equation~\eqref{eq:memhhomotop} into their constituent parts.

\begin{figure}[ht]
\centering
\includegraphics[width=12cm]{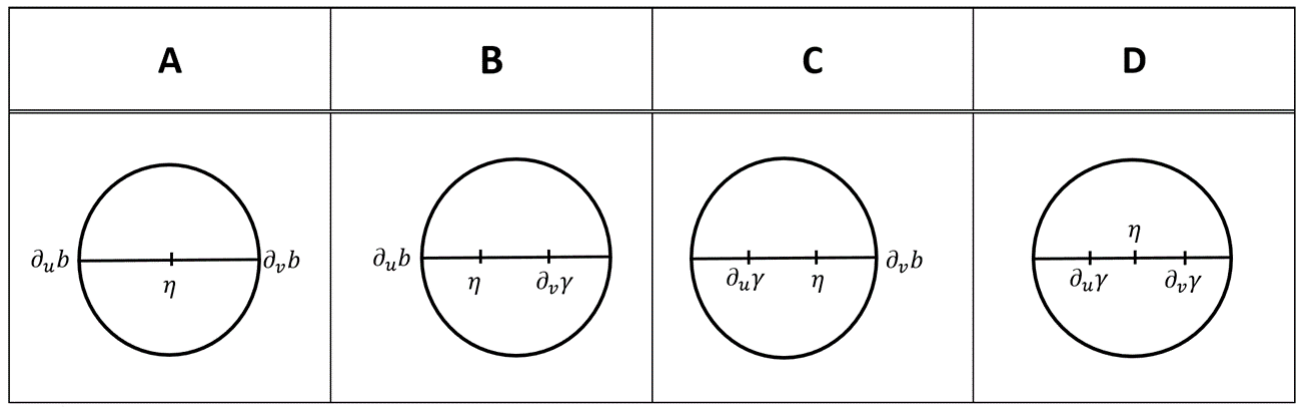}
\caption{The four summands in the second component of $H_{uv}$}\label{pic:H}
\end{figure}

Let $(\eta,\xi)\in \cone$.
By Lemma~\ref{lm:csymm} the derivative $\memh$ operator is given by
\begin{equation}\label{eq:dmem}
\d_v\memh(\eta,\xi)=
(\qg_{\emptyset,2}(\d_v\gamma, \eta),
(-1)^{n+1}\d_v\qbg_{-1,1}(\eta) - \d_vc\cdot\xi).
\end{equation}
To compute the left-hand side of~\eqref{eq:memhhomotop}, first calculate
\begin{align*}
\d_u&\memh\circ \d_v\memh(\eta,\xi)=\\
=&\d_u\memh((
\qg_{\emptyset,2}(\d_v\gamma,\eta),
(-1)^{n+1}\d_v\qbg_{-1,1}(\eta)
-\d_vc\cdot\xi))\\
=&
(\qg_{\emptyset,2}(\d_u\gamma, \qg_{\emptyset,2}(\d_v\gamma,\eta) ),
(-1)^{n+1}\d_u\qbg_{-1,1} (\qg_{\emptyset,2}(\d_v\gamma,\eta))
-\d_uc\cdot((-1)^{n+1}\d_v\qbg_{-1,1}(\eta) -\d_vc\cdot\xi))\\
=&
(\qg_{\emptyset,2}(\d_u\gamma, \qg_{\emptyset,2}(\d_v\gamma,\eta) ),
(-1)^{n+1}(\d_u\qbg_{-1,1} (\qg_{\emptyset,2}(\d_v\gamma,\eta)) -\d_uc\cdot\d_v\qbg_{-1,1}(\eta)) +\d_uc\cdot\d_vc\cdot\xi).
\end{align*}
Symmetrizing, we get
\begin{align*}
\d_u\memh\circ \d_v\memh &(\eta,\xi)
-(-1)^{|u||v|}\d_v\memh\circ \d_u\memh(\eta,\xi)=\\
=&
\Big(\qg_{\emptyset,2}(\d_u\gamma, \qg_{\emptyset,2}(\d_v\gamma,\eta))
-(-1)^{|u||v|}\qg_{\emptyset,2}(\d_v\gamma, \qg_{\emptyset,2}(\d_u\gamma,\eta)),\\
&\hspace{5em}
(-1)^{n+1}( \d_u\qbg_{-1,1}(\qg_{\emptyset,2}(\d_v\gamma,\eta)) -\d_uc\cdot\d_v\qbg_{-1,1}(\eta)-\\
&\hspace{10em}
- (-1)^{|v||u|}\d_v\qbg_{-1,1}(\qg_{\emptyset,2}(\d_u\gamma,\eta))
+(-1)^{|v||u|}\d_vc\cdot\d_u\qbg_{-1,1}(\eta))\Big).
\end{align*}
Compute the right hand side of equation~\eqref{eq:memhhomotop}:
\begin{align*}
\dcone (H_{uv}&(\eta,\xi))=
(-1)^{1+|\eta|+|u|+|v|+|\eta||v|} \big( d\qg_{\emptyset,3;\chi_0}(\d_u\gamma,\eta,\d_v\gamma), \i(\qg_{\emptyset,3;\chi_0}(\d_u\gamma,\eta,\d_v\gamma))\big),
\shortintertext{and}
(-1)^{|u|+|v|}H_{uv}(\dcone &(\eta,\xi))=
(-1)^{|u|+|v|}H_{uv}(d\eta,\i(\eta))=\\
=&
\big(
(-1)^{|v|+|\eta|+|\eta||v|} \qg_{\emptyset,3;\chi_0}(\d_u\gamma,d\eta,\d_v\gamma),\\
&\hspace{5em}+(-1)^{1+|v|+|\eta|+|\eta||v|} \qbg_{-1,3;0,3}(\d_u\gamma,d\eta,\d_v\gamma)+\\
&\hspace{5em}+(-1)^{|u|+|v|+|\eta|+|\eta|(|u|+|v|)} \langle\qbg_{0,2;1_0,2}(d\eta,\d_u\gamma),\d_vb\rangle+\\
&\hspace{5em}+(-1)^{1+|\eta|+|u|+|v|+|\eta|(|u|+|v|)+|u||v|} \langle\qbg_{0,2;1_0,2}(d\eta,\d_v\gamma),\d_ub\rangle+\\
&\hspace{5em}+(-1)^{1+|\eta|+|u|+|v|+|\eta|(|u|+|v|)+|u||v|} \langle\qbg_{1,1;2_{0,1},1}(\d_vb;d\eta),\d_ub\rangle \big).
\end{align*}
Equality in the first component of~\eqref{eq:memhhomotop} follows from Proposition~\ref{lm:clst}.
Equality in the second component reads
\begin{align}
(-&1)^{n+1}( \d_u\qbg_{-1,1}(\qg_{\emptyset,2}(\d_v\gamma,\eta)) -\d_uc\cdot\d_v\qbg_{-1,1}(\eta)-\nonumber\\
&\hspace{10em}
- (-1)^{|v||u|}\d_v\qbg_{-1,1}(\qg_{\emptyset,2}(\d_u\gamma,\eta))
+(-1)^{|v||u|}\d_vc\cdot\d_u\qbg_{-1,1}(\eta))=\nonumber\\
=&
(-1)^{1+|\eta|+|u|+|v|+|\eta||v|}\i(\qg_{\emptyset, 3;\chi_0}(\d_u\gamma,\eta,\d_v\gamma))+\nonumber\\
&+(-1)^{1+|v|+|\eta|+|\eta||v|} \qbg_{-1,3;0,3}(\d_u\gamma,d\eta,\d_v\gamma)
+(-1)^{|u|+|v|+|\eta|+|\eta|(|u|+|v|)} \langle\qbg_{0,2;1_0,2}(d\eta,\d_u\gamma),\d_vb\rangle+\nonumber\\
&+(-1)^{1+|\eta|+|u|+|v|+|\eta|(|u|+|v|)+|u||v|} \langle\qbg_{0,2;1_0,2}(d\eta,\d_v\gamma),\d_ub\rangle+\nonumber\\
&\hspace{5em}+(-1)^{1+|\eta|+|u|+|v|+|\eta|(|u|+|v|)+|u||v|} \langle\qbg_{1,1;2_{0,1},1}(\d_vb;d\eta),\d_ub\rangle. \label{eq:secc}
\end{align}
To show this, we use the deformed versions of the geodesic structure equations.
Figure~\ref{pic:table} depicts the boundary components of the moduli spaces A,B,C,D, in Figure~\ref{pic:H} that contribute to our calculation.

\begin{figure}[ht]
\centering
\includegraphics[width=16.5cm]{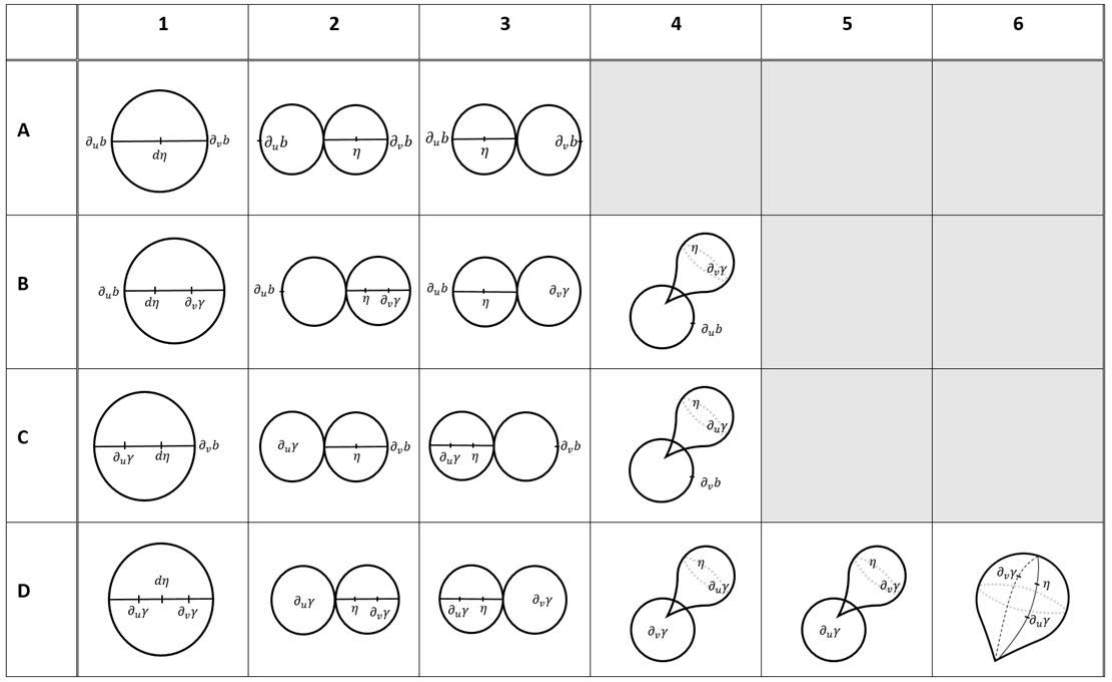}
\caption{Boundary components of the moduli spaces from Figure~\ref{pic:H}}\label{pic:table}
\end{figure}

First, keeping in mind line D of Figure~\ref{pic:table}, consider the contribution from Proposition~\ref{prop:03k-1}:
\begin{align*}
0=&(-1)^{|u|}\qbg_{-1,3;0,3}(\d_u\gamma, d\eta, \d_v\gamma)
+(-1)^{|\eta|+|u|+|v|} \langle\qbg_{0,3;0,3}(\d_u\gamma,\eta,\d_v\gamma), \qbg_{0,0}\rangle+\\
&-(-1)^{|u|}\langle\qbg_{0,1} (\d_u\gamma), \qbg_{0,2;1_0,2}(\eta,\d_v\gamma)\rangle
-(-1)^{|\eta|+|u|+|v|+|u||\eta|}\langle\qbg_{0,2;1_0,2} (\eta, \d_u\gamma), \qbg_{0,1}(\d_v\gamma)\rangle-\\
&-(-1)^{|u|+|v|+|\eta|+n+1}\qbg_{-1,2}( \qg_{\emptyset,2}(\d_u\gamma, \eta),\d_v\gamma) \\
&+(-1)^{|u|(|\eta|+|v|)+|u|+|v|+|\eta|+n+1}\qbg_{-1,2}( \qg_{\emptyset,2}(\eta,\d_v\gamma),\d_u\gamma)+\\
&+(-1)^{n + |u| + |\eta| + |v| + 1}\int_L\qg_{\emptyset,3;(0,1,2,3)\in(1,\infty)} (\d_u\gamma,\eta,\d_v\gamma).
\end{align*}
Lemma~\ref{lm:deggeod} gives $|\qg_{\emptyset,3;(0,1,2,3)\in(1,\infty)} (\d_u\gamma,\eta,\d_v\gamma)| \equiv |u| +|\eta| + |v| + 1 \pmod 2,$ so
\[
(-1)^{n + |u| + |\eta| + |v| + 1}\int_L\qg_{\emptyset,3;(0,1,2,3)\in(1,\infty)} (\d_u\gamma,\eta,\d_v\gamma)
= \i(\qg_{\emptyset,3;(0,1,2,3)\in(1,\infty)} (\d_u\gamma,\eta,\d_v\gamma)).
\]
Using the Maurer-Cartan equation, $\qbg_{0,0}=c\cdot 1$, we can apply the unit property, Lemma~\ref{lm:unitgeod}, to get
\begin{align*}
0=&(-1)^{|u|}\qbg_{-1,3;0,3}(\d_u\gamma, d\eta, \d_v\gamma)+\\
&-(-1)^{|u|}\langle\qbg_{0,1} (\d_u\gamma), \qbg_{0,2;1_0,2}(\eta,\d_v\gamma)\rangle
-(-1)^{|\eta|+|u|+|v|+|u||\eta|}\langle\qbg_{0,2;1_0,2} (\eta,\d_u\gamma), \qbg_{0,1}(\d_v\gamma)\rangle+\\
&+(-1)^{|u||\eta|+|u|+|v|+|\eta|+n+1}\big(-\qbg_{-1,2}( \qg_{\emptyset,2}(\eta,\d_u\gamma),\d_v\gamma)
+(-1)^{|u||v|}\qbg_{-1,2}( \qg_{\emptyset,2}(\eta, \d_v\gamma),\d_u\gamma)\big)+\\
&+\i(\qg_{\emptyset,3;(0,1,2,3)\in(1,\infty)} (\d_u\gamma,\eta,\d_v\gamma)).
\end{align*}
Multiply the equation by $(-1)^{1+|\eta|+|u|+|v|+|v||\eta|}$ to get
\begin{align}\label{eq:def36}
0=&(-1)^{1+|\eta|+|v|+|v||\eta|} \qbg_{-1,3;0,3}(\d_u\gamma, d\eta, \d_v\gamma)+\nonumber\\
&+(-1)^{|\eta|+|v|+|v||\eta|}\langle\qbg_{0,1} (\d_u\gamma), \qbg_{0,2;1_0,2}(\eta,\d_v\gamma)\rangle
+(-1)^{|u||\eta|+|v||\eta|}\langle\qbg_{0,2;1_0,2} (\eta,\d_u\gamma), \qbg_{0,1}(\d_v\gamma)\rangle+\nonumber\\
&+(-1)^{|u||\eta|+|v||\eta|+n} \big(-\qbg_{-1,2}( \qg_{\emptyset,2}( \eta,\d_u\gamma),\d_v\gamma)
+(-1)^{|u||v|}\qbg_{-1,2}( \qg_{\emptyset,2}(\eta, \d_v\gamma),\d_u\gamma)\big)+\\
&+(-1)^{1+|\eta|+|u|+|v|+|v||\eta|} \i(\qg_{\emptyset,3;(0,1,2,3)\in(1,\infty)} (\d_u\gamma,\eta,\d_v\gamma)).\nonumber
\end{align}
The summands of equation~\eqref{eq:def36} correspond to the pictures in line D of Figure~\ref{pic:table}.

Next, keeping in mind line $B$ of Figure~\ref{pic:table}, we compute the contribution of Proposition~\ref{prop:12}, using again the Maurer-Cartan equation, $\qbg_{0,0}=c\cdot 1,$ and the unit property, Lemma~\ref{lm:unitgeod}:
\begin{align*}
0=&
\qbg_{0,2;1_0,2}(d\eta,\d_v\gamma)
+(-1)^{|\eta|+|v|}\qbg_{1,2;1_0,2}( \qbg_{0,0};\eta,\d_v\gamma)
-\qbg_{1,0}(\qbg_{0,2;1_0,2}(\eta,\d_v\gamma))+\\
&+(-1)^{|\eta|+|v|}\qbg_{1,1;2_{0,1},1}(\qbg_{0,1}(\d_v\gamma);\eta)
+(-1)^{|\eta|+|v|+n+1}\qbg_{0,1}(\qg_{\emptyset,2}(\eta,\d_v\gamma))\\
=&
\qbg_{0,2;1_0,2}(d\eta,\d_v\gamma)
-\qbg_{1,0}(\qbg_{0,2;1_0,2}(\eta,\d_v\gamma))+\\
&+(-1)^{|\eta|+|v|}\qbg_{1,1;2_{0,1},1}(\qbg_{0,1}(\d_v\gamma);\eta)
+(-1)^{|\eta|+|v|+n+1} \qbg_{0,1}(\qg_{\emptyset,2}(\eta,\d_v\gamma)).
\end{align*}
Pairing with $\d_ub$ and using the cyclic properties of Lemmas~\ref{cl:cyclic} and~\ref{cl:cyc} together with the degree properties of Lemmas~\ref{deg_str_map} and~\ref{lm:deggeod},
\begin{align*}
0=&
\langle\qbg_{0,2;1_0,2}(d\eta,\d_v\gamma),\d_ub\rangle
-\langle\qbg_{1,0}(\qbg_{0,2;1_0,2}(\eta,\d_v\gamma)), \d_ub\rangle +\\
&+(-1)^{|\eta|+|v|} \langle\qbg_{1,1;2_{0,1},1}(\qbg_{0,1}(\d_v\gamma);\eta), \d_ub\rangle
+(-1)^{|\eta|+|v|+n+1} \langle\qbg_{0,1}(\qg_{\emptyset,2}(\eta,\d_v\gamma)), \d_ub\rangle\\
=&
\langle\qbg_{0,2;1_0,2}(d\eta,\d_v\gamma),\d_ub\rangle
-(-1)^{|u|(|\eta|+|v|)}
\langle\qbg_{1,0}(\d_ub), \qbg_{0,2;1_0,2}(\eta,\d_v\gamma)\rangle+\\
&+(-1)^{|\eta|+|v|+|u|(|v|+1)}
\langle\qbg_{1,1;2_{1,0},1}(\d_ub;\eta), \qbg_{0,1}(\d_v\gamma)\rangle
+(-1)^{|\eta|+|v|+n+1} \langle\qbg_{0,1}(\qg_{\emptyset,2}(\eta,\d_v\gamma)), \d_ub\rangle,\\
\shortintertext{and applying Lemma~\ref{lm:geoddirection} to the third summand,}
=&
\langle\qbg_{0,2;1_0,2}(d\eta,\d_v\gamma),\d_ub\rangle
-(-1)^{|u|(|\eta|+|v|)} \langle\qbg_{1,0}(\d_{u}b),\qbg_{0,2;1_0,2}(\eta,\d_v\gamma)\rangle+\\
&\hspace{-7pt}+(-1)^{|\eta|+|v|+|u|(|v|+1)+1} \langle\qbg_{1,1;2_{0,1},1}(\d_ub;\eta), \qbg_{0,1}(\d_v\gamma)\rangle
+(-1)^{|\eta|+|v|+n+1} \langle\qbg_{0,1}(\qg_{\emptyset,2}(\eta,\d_v\gamma)), \d_ub\rangle.
\end{align*}
To this equation with sign $(-1)^{(|\eta|+1)(|u|+|v|+1)+|u||v|}$, add the corresponding equation with $u,v,$ switched and sign $(-1)^{1+(|\eta|+1)(|u|+|v|+1)}$ to get
\begin{align}\label{eq:def35}
0=&(-1)^{1+|\eta|+|u|+|v|+ |\eta||u|+|\eta||v|+|u||v|} \langle\qbg_{0,2;1_0,2}(d\eta,\d_v\gamma),\d_ub\rangle
+\nonumber\\
&\hspace{10em}+
(-1)^{|\eta|+|u|+|v|+|\eta||u|+|\eta||v|} \langle\qbg_{0,2;1_0,2}(d\eta,\d_u\gamma),\d_vb\rangle+\nonumber\\
&+
(-1)^{|\eta|+|u|+|v|+ |\eta||v|} \langle\qbg_{1,0}(\d_ub),\qbg_{0,2;1_0,2}(\eta,\d_v\gamma)\rangle+\nonumber\\
&\hspace{10em}+
(-1)^{1+|\eta|+|u|+|v|+|\eta||u| + |u||v|} \langle\qbg_{1,0}(\d_vb),\qbg_{0,2;1_0,2}(\eta,\d_u\gamma)\rangle+\nonumber\\
&+(-1)^{|\eta||u|+|\eta||v|} \langle\qbg_{1,1;2_{0,1},1}(\d_ub;\eta), \qbg_{0,1}(\d_v\gamma)\rangle+\\
&\hspace{10em}
+(-1)^{1+|\eta||u|+|\eta||v|+|u||v|}
\langle\qbg_{1,1;2_{0,1},1}(\d_vb;\eta), \qbg_{0,1}(\d_u\gamma)\rangle+\nonumber\\
&+
(-1)^{|u|+|\eta||u|+|\eta||v| +|u||v|+n} \langle\qbg_{0,1}(\qg_{\emptyset,2}(\eta,\d_v\gamma)), \d_ub\rangle+\nonumber\\
&\hspace{10em}
+(-1)^{1+|v|+|\eta||u|+|\eta||v|+n} \langle\qbg_{0,1}(\qg_{\emptyset,2}(\eta,\d_u\gamma)), \d_vb\rangle .\nonumber
\end{align}
The summands of equation~\eqref{eq:def35} correspond to the pictures in lines B and C of Figure~\ref{pic:table}.

Lastly, keeping in mind line A of Figure~\ref{pic:table}, compute the contribution of Proposition~\ref{prop:21}, again using $\qbg_{0,0}=c\cdot 1$ and Lemma~\ref{lm:unitgeod}:
\begin{align*}
0=&
\qbg_{1,1;2_{0,1},1}(\d_vb; d\eta)
+(-1)^{|\eta|+ |v|}\qbg_{2,1;2_{0,1},1}(\d_vb, \qbg_{0,0};\eta)
+(-1)^{|\eta|}\qbg_{2,1;2_{0,2},1} (\qbg_{0,0},\d_vb;\eta)+\\
&-\qbg_{1,0}( \qbg_{1,1;2_{0,1},1}(\d_vb;\eta))
+(-1)^{|\eta|}\qbg_{1,1;2_{0,1},1}( \qbg_{1,0}(\d_vb);\eta)\\
=&
\qbg_{1,1;2_{0,1},1}(\d_vb; d\eta)
-\qbg_{1,0}( \qbg_{1,1;2_{0,1},1}(\d_vb;\eta))
+(-1)^{|\eta|}\qbg_{1,1;2_{0,1},1}( \qbg_{1,0}(\d_vb);\eta).
\end{align*}
Pairing with $\d_ub$ and using the cyclic properties of Lemmas~\ref{cl:cyclic} and~\ref{cl:cyc} together with the degree properties of Lemmas~\ref{deg_str_map} and~\ref{lm:deggeod}, we get
\begin{align*}
0=&
\langle \qbg_{1,1;2_{0,1},1}(\d_vb; d\eta), \d_ub\rangle
-\langle\qbg_{1,0}(\qbg_{1,1;2_{0,1},1}(\d_vb;\eta)), \d_ub\rangle
+(-1)^{|\eta|}\langle\qbg_{1,1;2_{0,1},1}( \qbg_{1,0}(\d_vb);\eta), \d_ub\rangle\\
=&
\langle \qbg_{1,1;2_{0,1},1}(\d_vb; d\eta), \d_ub\rangle-\\
&
-(-1)^{|u|(|v|+|\eta|)} \langle\qbg_{1,0}(\d_ub), \qbg_{1,1;2_{0,1},1}(\d_vb;\eta)\rangle
+(-1)^{|\eta|+|u|(|v|+1)}
\langle\qbg_{1,1;2_{1,0},1}(\d_ub ;\eta), \qbg_{1,0}(\d_vb)\rangle.
\end{align*}
Apply Lemma~\ref{lm:geoddirection} to the last summand of the right hand side of the preceding equation and multiply the entire equation by $(-1)^{(|\eta|+1)(|u|+|v|+1)+|u||v|}$ to get
\begin{align}\label{eq:def34}
0=&
(-1)^{(|\eta|+1)(|u|+|v|+1)+|u||v|}
\langle \qbg_{1,1;2_{0,1},1}(\d_vb; d\eta), \d_ub\rangle+\notag\\
&+(-1)^{|\eta|+|u|+|v|+|v||\eta|}
\langle\qbg_{1,0}(\d_ub), \qbg_{1,1;2_{0,1},1}(\d_vb;\eta)\rangle +\\
&+(-1)^{|v|+|\eta||u|+|\eta||v|} \langle\qbg_{1,1;2_{0,1},1}(\d_ub ;\eta), \qbg_{1,0}(\d_vb)\rangle.\notag
\end{align}
The summands of equation~\eqref{eq:def34} correspond to the pictures in line A of Figure~\ref{pic:table}.

In the following, we add up equations~\eqref{eq:def36},~\eqref{eq:def35}, and~\eqref{eq:def34}, in stages.
First, consider all the summands that involve $\qg_{\emptyset,l}$ as an interior input. By Lemmas~\ref{lm:qlinear}, ~\ref{cl:cyclic}, and~\ref{cl:symmetry},
we have
\begin{equation}\label{eq:dqbg}
\d_u \qbg_{-1,1}(\xi) = (-1)^{|\xi||u|}\qbg_{-1,2}(\xi,\d_u\gamma) +(-1)^{|\xi||u|+|u|}\langle \qbg_{0,1}(\xi),\d_ub\rangle,
\end{equation}
and thus
\begin{align*}
(-1)^{|u||\eta|+|v||\eta|+n}& \big(-\qbg_{-1,2}( \qg_{\emptyset,2}( \eta,\d_u\gamma),\d_v\gamma)
+(-1)^{|u||v|}\qbg_{-1,2}( \qg_{\emptyset,2}(\eta, \d_v\gamma),\d_u\gamma)\big)+\\
+
(-1)&^{|u|+|\eta||u|+|\eta||v| +|u||v|+n}\langle\qbg_{0,1}(\qg_{\emptyset,2}(\eta,\d_v\gamma)), \d_ub\rangle+\\
&+(-1)^{1+|v|+|\eta||u|+|\eta||v| +n} \langle\qbg_{0,1}(\qg_{\emptyset,2}(\eta,\d_u\gamma)), \d_vb\rangle = \\
=
(-1)&^{|\eta||u|+|\eta||v|+n}
\big(
-\qbg_{-1,2}( \qg_{\emptyset,2}( \eta,\d_u\gamma),\d_v\gamma)
+(-1)^{|u||v|}\qbg_{-1,2}( \qg_{\emptyset,2}(\eta, \d_v\gamma),\d_u\gamma)+\\
&\quad +
(-1)^{|u|+|u||v|} \langle\qbg_{0,1}(\qg_{\emptyset,2}(\eta,\d_v\gamma)), \d_ub\rangle
-(-1)^{|v|} \langle\qbg_{0,1}(\qg_{\emptyset,2}(\eta,\d_u\gamma)), \d_vb\rangle
\big)\\
=
(-1)&^{|\eta||v|+n}
\d_u\qbg_{-1,1}( \qg_{\emptyset,2}(\eta, \d_v\gamma))
-
(-1)^{|\eta||u|+|u||v|+n}
\d_v\qbg_{-1,1}(\qg_{\emptyset,2}(\eta,\d_u\gamma)).\\
\shortintertext{By Lemma~\ref{lm:csymm}, we continue}
=&
(-1)^{n}
\d_u\qbg_{-1,1}( \qg_{\emptyset,2}(\d_v\gamma,\eta))
-
(-1)^{|u||v|+n}
\d_v\qbg_{-1,1}(\qg_{\emptyset,2}(\d_u\gamma,\eta)).
\end{align*}
Second, consider the summands of equations~~\eqref{eq:def36},~\eqref{eq:def35}, and~\eqref{eq:def34}, that are products of $\qbg_{k,l}$ and $\qbg_{k,l;\chi}$:
\begin{align*}
(-1&)^{|\eta|+|v|+|v||\eta|}
\langle\qbg_{0,1}(\d_u\gamma), \qbg_{0,2;1_0,2}(\eta,\d_v\gamma)\rangle+\\
&\hspace{10em} +
(-1)^{|\eta|+|u|+|v|+ |\eta||v|} \langle\qbg_{1,0}(\d_ub), \qbg_{0,2;1_0,2}(\eta,\d_v\gamma)\rangle+\\
&+
(-1)^{1+|\eta||u|+|\eta||v|+|u||v|}
\langle\qbg_{1,1;2_{0,1},1}(\d_vb;\eta), \qbg_{0,1}(\d_u\gamma)\rangle+\\
&\hspace{10em}
+
(-1)^{|\eta|+|u|+|v|+|v||\eta|} \langle\qbg_{1,0}(\d_ub), \qbg_{1,1;2_{0,1},1}(\d_vb;\eta)\rangle+\\
&+
(-1)^{|u||\eta|+|v||\eta|}
\langle\qbg_{0,2;1_0,2}(\eta,\d_u\gamma), \qbg_{0,1}(\d_v\gamma)\rangle+\\
&\hspace{10em}
+
(-1)^{1+|\eta|+|u|+|v|+|\eta||u|+|u||v|}
\langle\qbg_{1,0}(\d_vb),\qbg_{0,2;1_0,2}(\eta,\d_u\gamma)\rangle+\\
&+
(-1)^{|\eta||u|+|\eta||v|}
\langle\qbg_{1,1;2_{0,1},1}(\d_ub;\eta), \qbg_{0,1}(\d_v\gamma)\rangle+\\
&\hspace{10em}+
(-1)^{|v|+|\eta||u|+|\eta||v|}
\langle\qbg_{1,1;2_{0,1},1}(\d_ub;\eta), \qbg_{1,0}(\d_vb)\rangle=
\shortintertext{by the symmetry property of the inner product, equation~\eqref{eq:psgn},}
=&
(-1)^{1+|\eta||u|+|\eta||v|+|u||v|}\big(
\langle \qbg_{0,2;1_0,2}(\eta,\d_v\gamma), \qbg_{0,1}(\d_u\gamma)\rangle +
(-1)^{|u|}
\langle \qbg_{0,2;1_0,2}(\eta,\d_v\gamma), \qbg_{1,0}(\d_ub)\rangle\big)+\\
&\hspace{-10.25pt}+
(-1)^{1+|\eta||u|+|\eta||v|+|u||v|}
\big(
\langle\qbg_{1,1;2_{0,1},1}(\d_vb;\eta), \qbg_{0,1}(\d_u\gamma)\rangle
+
(-1)^{|u|}
\langle\qbg_{1,1;2_{0,1},1}(\d_vb;\eta),
\qbg_{1,0}(\d_ub)\rangle\big)+\\
&+
(-1)^{|u||\eta|+|v||\eta|}
\big(
\langle\qbg_{0,2;1_0,2}(\eta,\d_u\gamma), \qbg_{0,1}(\d_v\gamma)\rangle
+
(-1)^{|v|}
\langle\qbg_{0,2;1_0,2}(\eta,\d_u\gamma), \qbg_{1,0}(\d_vb)\rangle\big)+\\
&+
(-1)^{|\eta||u|+|\eta||v|}
\big(
\langle\qbg_{1,1;2_{0,1},1}(\d_ub;\eta), \qbg_{0,1}(\d_v\gamma)\rangle
+
(-1)^{|v|}
\langle\qbg_{1,1;2_{0,1},1}(\d_ub;\eta), \qbg_{1,0}(\d_vb)\rangle\big).
\shortintertext{
We proceed with an algebraic manipulation represented graphically in Figure~\ref{pic:add}. By Lemmas~\ref{lm:qlinear},~\ref{cl:cyclic}, and~\ref{cl:symmetry}, taking the derivative of the Maurer-Cartan equation, we obtain
$\qbg_{0,1}(\d_u\gamma)+(-1)^{|u|}\qbg_{1,0}(\d_ub) = \d_u\qbg_{0,0} = \d_u c \cdot 1,$ and thus}
=&
(-1)^{1+|\eta||u|+|\eta||v|+|u||v|}
\langle \qbg_{0,2;1_0,2}(\eta,\d_v\gamma), \d_uc\cdot 1\rangle +\\
&+
(-1)^{1+|\eta||u|+|\eta||v|+|u||v|}
\langle\qbg_{1,1;2_{0,1},1}(\d_vb;\eta), \d_uc\cdot 1\rangle+\\
&+
(-1)^{|\eta||u|+|\eta||v|}
\langle\qbg_{0,2;1_0,2}(\eta,\d_u\gamma), \d_vc\cdot 1\rangle+\\
&+
(-1)^{|\eta||u|+|\eta||v|}
\langle\qbg_{1,1;2_{0,1},1}(\d_ub;\eta), \d_vc\cdot 1\rangle.\\
\shortintertext{Using the bilinearity of the pairing, Lemma~\ref{lm:qlinear}, we continue}
=&
(-1)^{1+|v||\eta|}
\d_uc\cdot
\langle \qbg_{0,2;1_0,2}(\eta,\d_v\gamma), 1\rangle +\\
&+
(-1)^{1+|v||\eta|}
\d_u c\cdot
\langle\qbg_{1,1;2_{0,1},1}(\d_vb;\eta), 1\rangle+\\
&+
(-1)^{|u||\eta|+|u||v|}
\d_vc\cdot
\langle\qbg_{0,2;1_0,2}(\eta,\d_u\gamma), 1\rangle+\\
&+
(-1)^{|\eta||u|+|u||v|}
 \d_vc\cdot
\langle\qbg_{1,1;2_{0,1},1}(\d_ub;\eta), 1\rangle, \\
\shortintertext{so by the geodesic unit properties, Lemma~\ref{lm:geodunit-1} and Corollary~\ref{cor:geodunit},}
=&
(-1)^{1+|\eta||v|+n}
\d_uc\cdot
\qbg_{-1,2}(\eta,\d_v\gamma) +\\
&+
(-1)^{1+|v|+|\eta||v|+n}
\d_u c\cdot
\langle\qbg_{0,1}(\eta), \d_vb\rangle+\\
&+
(-1)^{|u||\eta|+|u||v|+n}
\d_vc\cdot
\qbg_{-1,2}(\eta,\d_u\gamma)+\\
&+
(-1)^{|u|+|\eta||u|+|u||v|+n}
 \d_vc\cdot
\langle \qbg_{0,1}(\eta),\d_ub\rangle,\\
\shortintertext{and by equation~\eqref{eq:dqbg},}
=&
(-1)^n\big(
-
\d_u c\cdot \d_v
\qbg_{-1,1}(\eta)
+(-1)^{|u||v|}
\d_v c\cdot\d_u
\qbg_{-1,1}(\eta)
\big).
\end{align*}

\begin{figure}[ht]
\centering
\includegraphics[width=15cm]{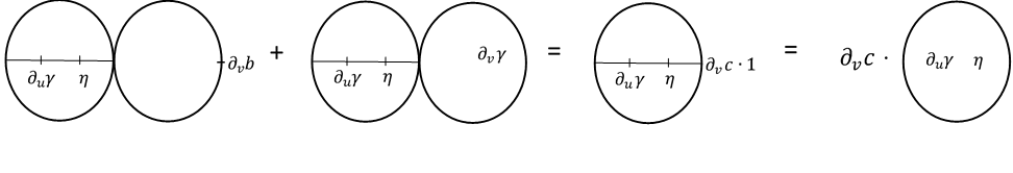}
\caption{Combining contributions from two boundary strata to form a product of derivatives.}\label{pic:add}
\end{figure}

In total, equations~\eqref{eq:def36},~\eqref{eq:def35}, and~\eqref{eq:def34}, add up to
\begin{align*}
0=&
(-1)^{1+|\eta|+|v| +|\eta||v|}
\qbg_{-1,3;0,3}(\d_u\gamma,d\eta,\d_v\gamma)+\\
&+
(-1)^{1+|\eta|+|u|+|v|+|\eta||u|+|\eta||v|+|u||v|}
\langle\qbg_{0,2;1_0,2}(d\eta,\d_v\gamma),\d_ub\rangle+\\
&+
(-1)^{|\eta|+|u|+|v|+|\eta||u|+|\eta||v|}
\langle\qbg_{0,2;1_0,2}(d\eta,\d_u\gamma),\d_vb\rangle+\\
&+
(-1)^{1+|\eta|+|u|+|v|+|\eta||u|+|\eta||v|+|u||v|}
\langle\qbg_{1,1;2_{0,1},1}(\d_vb; d\eta),\d_ub\rangle+\\
&+
(-1)^{1+|\eta|+|u|+|v|+|v||\eta|}
\i(\qg_{\emptyset,3;(0,1,2,3)\in(1,\infty)} (\d_u\gamma, \eta,\d_v\gamma))+\\
&+(-1)^{n}
\big(
\d_u\qbg_{-1,1}(\qg_{\emptyset,2}(\d_v\gamma,\eta))
-
(-1)^{|u||v|}
\d_v\qbg_{-1,1}(\qg_{\emptyset,2}(\d_u\gamma,\eta))\big)+\\
&+(-1)^{n}
\big(
-\d_uc\cdot \d_v\qbg_{-1,1}(\eta)
+(-1)^{|u||v|}\d_vc\cdot\d_u\qbg_{-1,1}(\eta)\big),
\end{align*}
which is equivalent to equation~\eqref{eq:secc}, as required.

\end{proof}

\subsection{Relative potential}\label{ssec:psi}

In equation~\eqref{eq:psidfn}, we defined
\[
\psi(\gamma,b):=(\qg_{\emptyset,0},(-1)^{n+1}\qbg_{-1,0}) \in \cone.
\]

\begin{lm}
	$\dcone\psi(\gamma,b)=0.$
\end{lm}
\begin{proof}
By definition,
\[
\dcone\psi(\gamma,b)
=(d\qg_{\emptyset,0},\i(\qg_{\emptyset,0}))
=(0,\i(\qg_{\emptyset,0})).
\]
We claim that $\i(\qg_{\emptyset,0})=0$.
To see this, apply the deformed version of Proposition~\ref{q-1_rel}
to the case $l=0$:
\begin{multline*}
0=
\frac{1}{2}
\langle\qbg_{0,0},\qbg_{0,0}\rangle
+(-1)^{n+1}\i( \qg_{\emptyset,0})
=\frac{1}{2}c^2
\langle 1,1\rangle
+(-1)^{n+1}\i( \qg_{\emptyset,0})
=(-1)^{n+1}\i( \qg_{\emptyset,0}).
\end{multline*}
\end{proof}

As in Section~\ref{sssec:psi}, define the relative potential to be the cohomology class
\[
\Psi(\gamma,b)=[\psi(\gamma,b)]
\in H^*(\cone).
\]

\begin{proof}[Proof of Theorem~\ref{thm:psiinv}]
Let $(\gt,\bt)$ be a pseudoisotopy from $(\gamma,b)$ to $(\gamma',b')$.
We show that $(\q^{\gamma'}_{\emptyset,0},(-1)^{n+1}\q^{\gamma',b'}_{-1,0})$ is cohomologous to $(\qg_{\emptyset,0},(-1)^{n+1}\qbg_{-1,0})$, by considering their difference.

To compute the first component, apply Stokes' theorem, Proposition~\ref{stokes}, to $f=p_X\circ\evt_0.$ This gives
\[
\q^{\gamma'}_{\emptyset,0}-\qg_{\emptyset,0}=
-d((p_X)_*\qt^{\tilde\gamma}_{\emptyset,0}).
\]

We now compare the difference in the second component.
Use the analog of Proposition~\ref{cl:qt_-1} for the deformed $\qt_{-1,0}$ operator:
\begin{multline*}
0=d\qt_{-1,0}^{\gt,\bt}+
\frac{1}{2}
\ll\qt^{\gt,\bt}_{0,0},\qt^{\gt,\bt}_{0,0}\gg
- (p_I)_*i^*\qt^{\gt}_{\emptyset,0}=\\
=
d\qt_{-1,0}^{\gt,\bt}+
\frac{1}{2}
\ll \ct\cdot 1,\ct\cdot 1\gg
- (p_I)_*i^*\qt^{\gt}_{\emptyset,0}
=
d\qt_{-1,0}^{\gt,\bt}
- (p_I)_*i^*\qt^{\gt}_{\emptyset,0}.
\end{multline*}
Pushing the expression along $pt:I\to pt$ and using Lemma~\ref{lm:fibprodsgn}, we have
\[
\int_Id\qt_{-1,0}^{\gt,\bt}=
pt_*(p_I)_*i^*\qt^{\tilde\gamma}_{\emptyset,0}=
pt_*(p_L)_*i^*\qt^{\tilde\gamma}_{\emptyset,0}
=(-1)^n pt_*i^*(p_X)_*\qt^{\tilde\gamma}_{\emptyset,0}
=-\i((p_X)_*\qt^{\tilde\gamma}_{\emptyset,0}).
\]
On the other hand, by Stokes' theorem, Proposition~\ref{stokes}, we have
\[
\int_Id\qt_{-1,0}^{\gt,\bt}=
(-1)^{|\qt_{-1,0}^{\gt,\bt}|}(\q_{-1,0}^{\gamma',b'}-\qbg_{-1,0})
=
(-1)^{n+1}(\q_{-1,0}^{\gamma',b'}-\qbg_{-1,0}).
\]
So,
\[
(-1)^{n+1}\big(\q_{-1,0}^{\gamma',b'}-\qbg_{-1,0}\big)
=
-\i((p_X)_*\qt^{\tilde\gamma}_{\emptyset,0}).
\]

In total, we have
\[
\Psi(\gamma',b')-\Psi(\gamma,b)= [-\dcone((p_X)_*\qt^{\tilde\gamma}_{\emptyset,0},0)]=0.
\]
\end{proof}

\begin{lm}\label{lm:dpsi}
Assume $b$ is separated and write $\int_L b=f$.
Then the relative potential is related to $\memh$ via
\[
\d_u\psi=\memh(\d_u\gamma,0),\qquad \forall u\in W,
\]
and
\[
\d_s\psi=\dmf\cdot\memh(0,1), \qquad \forall s\in S.
\]
Consequently,
\[
\d_u\Psi=\nund([(\d_u\gamma,0)])\quad\mbox{and} \quad
\d_s\Psi=\dmf\cdot\nund([(0,1)]).
\]
\end{lm}
\begin{proof}
To prove the first identity, note that the separateness assumption implies
$\d_u\int_L b=0$. Therefore, by Lemmas~\ref{lm:csymm} and~\ref{cl:cyclic} and the Maurer-Cartan equation,
\begin{multline*}
\d_u\psi(\gamma,b)=
(\qg_{\emptyset,1}(\d_u\gamma),(-1)^{n+1}\qbg_{-1,1}(\d_u\gamma)) +(0,(-1)^{n+1}\langle\qbg_{0,0},\d_ub\rangle)=\\
\hspace{7em}=(\qg_{\emptyset,1}(\d_u\gamma), (-1)^{n+1}\qbg_{-1,1}(\d_u\gamma)) +(0,(-1)^{n+1}c\cdot\langle 1,\d_ub\rangle)=\\
=(\qg_{\emptyset,1}(\d_u\gamma),(-1)^{n+1}\qbg_{-1,1}(\d_u\gamma))
=\memh(\d_u\gamma,0).
\end{multline*}
To prove the second identity,
compute
\begin{multline*}
\d_s\psi(\gamma,b)=
(0,(-1)^{n+1}\langle\qbg_{0,0},\d_sb\rangle)
=(-1)^{n+1}(0,c\cdot\langle 1,\d_sb\rangle)=\\
=(-1)^{n+1}(0,c\cdot(-1)^{n}\d_s\int_Lb)
=\dmf\cdot(0,-c)
=\dmf\cdot\memh(0,1).
\end{multline*}
\end{proof}

\subsection{Enhanced superpotential}\label{ssec:omegabar}
We begin by giving a full account of the conditions imposed on the map $P$ discussed in Section~\ref{sssec:esp}. There is a natural map of complexes $\Ah(X,L) \to \cone$ given by $\eta \mapsto (\eta,0).$ Denote by $a : \Hh^*(X,L;\R) \to H^*(\cone)$ the induced map on cohomology.
Consider the commutative diagram of long exact sequences,
\begin{equation}\label{eq:tris}
\xymatrix{
H^*(\cone) \ar[rr]^\pi && H^*(X;\Qh_W) \ar[dl]_{[1]}^{\i} \\
 & R_W[-n-1] \ar[ul]^x  \\
\Hh^*(X,L;\Qh_W) \ar'[r]^(.7){\rho_\Qh}[rr] \ar@{^{(}->}[uu]^a && H^*(X;\Qh_W) \ar[dl]^{\i_\Qh}\ar[uu]^\wr \\
 & \Qh_W[-n] \ar[ul]_{[1]}^{y_\Qh}\ar@{^{(}->}[uu]_(.3){[1]},
}
\end{equation}
where $a$ is injective by the five lemma. Observe that for this diagram to commute, we need the map $x$ to be given at chain level by $r \mapsto (0,-r)$. There is a canonical chain map $\cone \to R_W/Q_W$ given by $(\eta,\xi) \mapsto [-\xi].$ Let
\[
\overline P : H^*(\cone) \to R_W/Q_W
\]
denote the induced map on cohomology. Let $q : R_W \to R_W/Q_W$ denote the quotient map, and let $\bar q : \Coker \i \to R_W/Q_W$ be the induced map. We obtain the following diagram with exact rows and columns.
\begin{equation}\label{eq:ybar}
\xymatrix{
        & 0 & 0 \\
0 \ar[r] & R_W/Q_W \ar[u] \ar[r]^\sim & R_W/Q_W \ar[u] \ar[r] & 0 \\
0 \ar[r] & \Coker \i \ar[u]^{\bar q} \ar[r]^(.45){\bar x} & H^*(\cone) \ar[u]^{\overline P} \ar[r]^(.55){\pi} & \Ker \i \ar[u] \ar[r] & 0 \\
0 \ar[r] & \Coker \i_\Qh \ar[r]^(.4){\bar y_\Qh}\ar[u]^{\bar a}  & \Hh^*(X,L;\Qh_W) \ar[r]^(.6){\rho_\Qh}\ar[u]^a & \Ker \i_\Qh \ar[u]^\wr \ar[r] & 0 \\
& 0 \ar[u] & 0 \ar[u] & 0 \ar[u]
}
\end{equation}
We choose
\[
P :  H^*(\cone) \lrarr \Coker \i,
\]
a left inverse to the map $\bar x$ from the diagram~\eqref{eq:ybar} satisfying the following two conditions. The first condition is that
\begin{equation}\label{eq:bP}
\bar q \circ P = \overline P.
\end{equation}
Note that if $[L]\ne 0$, then $\bar q$ is an isomorphism, so this determines $P$ completely.
Condition~\eqref{eq:bP} and the exactness of diagram~\eqref{eq:ybar} imply that there exists a unique $P_\Qh : \Hh^*(X,L;\Qh_W)\to \Coker \i_\Qh$ such that the following diagram commutes.
\begin{equation}\label{eq:PQ}
\xymatrix{
\Coker \i  & H^*(\cone) \ar[l]_P \\
 \Coker \i_\Qh \ar[u]^{\bar a}  & \Hh^*(X,L;\Qh_W) \ar[u]^a \ar[l]_(.55){P_\Qh} \\
}
\end{equation}
The second condition is that there exists $P_\R : \Hh^*(X,L;\R) \to \Coker \i_\R,$ such that
\begin{equation}\label{eq:PR}
P_\Qh = P_\R \otimes \Id_\Qh.
\end{equation}
Recall the exact sequence~\eqref{eq:yrho}. Denote by $\bar y : \Coker \i_\R \to \Hh^*(X,L;\R)$ the induced map.
\begin{lm}\label{lm:PRr}
$P_Q \circ \bar y_Q = \Id$ and $P_\R\circ \bar y = \Id.$
\end{lm}
\begin{proof}
By the exactness of diagram~\eqref{eq:ybar}, the map $\bar a$ is injective. So, to prove $P_Q \circ \bar y_Q = \Id$, it suffices to prove that $\bar a \circ P_Q \circ \bar y_Q = \bar a.$ By commutativity of diagrams~\eqref{eq:PQ} and~\eqref{eq:ybar}, we obtain
\[
\bar a \circ P_Q \circ \bar y_Q = P \circ a \circ \bar y_Q = P \circ \bar x \circ \bar a = \bar a.
\]
To see $P_\R \circ \bar y = \Id,$ observe that $\bar y_Q = \bar y \otimes \Id_\Qh.$
\end{proof}

\begin{lm}\label{lm:PRP}
Let $l : \Hh^*(X,L;\R) \to \Coker \i_\R$ satisfy $l \circ \bar y = \Id.$
There exists a unique choice of $P:  H^*(\cone) \to \Coker \i$ satisfying conditions~\eqref{eq:bP} and~\eqref{eq:PR} such that $l = P_\R.$ Moreover, $\Ker P = a(\Ker P_Q).$
\end{lm}
\begin{proof}
Let $l_\Qh = l \otimes \Id_\Qh.$ Then $l_\Qh \circ \bar y_\Qh = \Id,$ so the splitting lemma implies that \[
\Hh^*(X,L;\Qh_W) \simeq \Im \bar y_\Qh \oplus \Ker l_\Qh, \qquad \rho_\Qh|_{\Ker l_\Qh}: \Ker l_\Qh \overset{\sim}{\lrarr} H^*(X;\Qh).
\]
Diagram~\eqref{eq:ybar} gives $\pi \circ a = \rho_\Qh,$ so it follows that $\pi \circ a|_{\Ker l_\Qh} : \Ker l_\Qh \overset{\sim}{\to} H^*(X;\Qh).$ Thus,
\[
\pi|_{a(\Ker l_\Qh)} : a(\Ker l_\Qh) \overset{\sim}{\lrarr} H^*(X;\Qh)
\]
and the splitting lemma implies that
\[
H^*(\cone) \simeq \im \bar x \oplus a(\Ker l_\Qh).
\]
Take $P$ to be the unique left inverse of $\bar x$ such that $\Ker P = a(\Ker l_\Qh).$ Diagram~\eqref{eq:ybar} and the above splittings imply condition~\eqref{eq:bP} and $l_Q = P_Q.$ Condition~\eqref{eq:PR} follows.
\end{proof}

\begin{rem}\label{rem:Pgeom}
As mentioned, if $[L]\ne 0$ then $P$ is uniquely determined. If $[L]=0$, Lemma~\ref{lm:PRP} shows that $P$ is uniquely determined by
$P_{\R}:\Hh^*(X,L;\R)\to \R[-n-1].$ Geometrically, we can interpret $P_\R$ as integration over a homology class $V\in H_{n+1}(X,L;\R)$. Considering the boundary map $\d: H_{n+1}(X,L;\R) \to H_{n}(L;\R),$ the condition $P_{\R}\circ\bar{y}=\Id$ corresponds to $\d V =[L]$.
\end{rem}

In Section~\ref{sssec:esp}, we defined the enhanced superpotential by
\[
\Ob:=P \Psi
\]
and we defined the superpotential $\Omega$ by
\begin{equation}\label{eq:omconds}
q(\Omega) = \bar q(\Ob),\qquad D(\Omega) =0.
\end{equation}
\begin{lm}\label{lm:OP}
The enhanced superpotential $\Ob$ satisfies $\bar q(\Ob)=\overline P(\Psi)$. In particular, $\Omega$ is independent of the choice of $P$.
\end{lm}
\begin{proof}
By condition~\eqref{eq:bP} we have $\bar q(\Ob) = \bar q \circ P(\Psi) = \overline P (\Psi).$ Since the map $\overline P$ is canonical, the conditions~\eqref{eq:omconds} are independent of the choice of $P$.
\end{proof}

Following~\cite{ST2}, we write
\begin{equation}\label{eq:oh}
\Oh = (-1)^{n}\qbg_{-1,0}.
\end{equation}
\begin{lm}\label{lm:spc}
We have $\Omega = \Oh - D \Oh$.
\end{lm}
\begin{proof}
We verify the two conditions defining $\Omega.$ First, Lemma~\ref{lm:OP} gives
\[
\bar q(\Ob) = \overline P(\Psi) = q(\Oh - D \Oh).
\]
Second,
\[
D(\Oh - D(\Oh)) = D \Oh - D \Oh = 0.
\]
\end{proof}

\begin{rem}\label{rem:qder}
In~\cite{ST2} we defined
\begin{equation}\label{eq:ogwid}
\ogw_{\beta,k}([\gamma_{i_1}],\ldots,[\gamma_{i_l}])
:=
\text{ the coefficient of $T^{\beta}$ in }\d_{t_{i_1}}\cdots\d_{t_{i_l}}\d_s^k\Omega|_{s=0,t_j=0}.
\end{equation}
By Lemma~\ref{lm:spc} and the condition $q(\Omega)=\bar q(\Ob)$, this definition coincides with the definition of equation~\eqref{eq:ogw} in the current paper.
However, the analog of equation~\eqref{eq:ogwid} for $\ogwb$, with $\Ob$ instead of $\Omega$, is true only if $[L]=0$.
\end{rem}

Denote by
\[
p:R_W\lrarr \Coker(\i)
\]
the quotient map, and let $f = \int_L b.$
\begin{lm}\label{lm:dsomega}
Let $s\in S$.
\begin{enumerate}
\item
$
\d_s\Ob=p(\dmf\cdot c).
$
\item
$
\d_s \Omega = \dmf \cdot c.
$
\end{enumerate}
\end{lm}
\begin{proof}
By Lemma~\ref{lm:dpsi},
\[
\d_s\Ob= P(\dmf\cdot\nund[(0,1)])= P(0,-\dmf\cdot c)=P(\bar x(p(\dmf \cdot c)))= p(\dmf\cdot c).
\]
Lemma~\ref{lm:spc} and the Maurer-Cartan equation~\eqref{eq:bdchdfn} give
\[
\d_s \Omega = \d_s(\Oh - D \Oh) = \d_s \Oh
= (-1)^{n}\langle \qbg_{0,0},\d_s b \rangle
= (-1)^{n}\langle c\cdot 1, \d_s b \rangle
= c \cdot \d_s \left(\int_L b\right)
= \dmf \cdot c.
\]
\end{proof}

\begin{rem}\label{rem:plbc}
Recall from Definition~\ref{dfn:bdpair} that a bounding chain $b$ is called point-like if $S$ is $1$-dimensional and $f = \int_L b$ is a coordinate function on $S.$ In light of Lemma~\ref{lm:dsomega}, we will now argue that, generically, all open Gromov-Witten invariants can be obtained from point-like bounding chains. Indeed, consider a general bounding chain $b$ with $S$ of arbitrary dimension.
For generic $f,$ the point $0 \in S$ is regular. So, after a formal change of coordinates, we may assume that $f$ is one of the coordinate functions on $S.$ Then, Lemma~\ref{lm:dsomega} shows that the superpotentials $\Omega$ and $\Ob$ depend only on $f$ and not the other coordinate functions. Replacing $S$ with the one dimensional subspace on which the other coordinate functions vanish, and replacing $b$ with its restriction to this subspace, we obtain a point-like bounding chain. The superpotentials $\Omega$ and $\Ob$ associated with this point-like bounding chain contain the same information as the superpotentials associated to the original $b.$
\end{rem}

Let $U \subset H^*(X;\R)$ be a subspace, and let $(\gamma_W,b)$ be a bounding pair over $W := \rho^{-1}(U) \subset \Hh^*(X,L)$.
Recall that in Section~\ref{sssec:wallcross} we define $\g_\diamond:=y(1)$. By abuse of notation, denote by $\d_\diamond: R_W\to R_W$
the derivation corresponding to $\g_\diamond$.
\begin{prop}\label{prop:wallseries}
Suppose $[L]=0$ and $b$ is separated. Then
\[
\d_s\Ob=-\d_sf\cdot \d_\diamond\Ob,
\qquad
\forall s\in S.
\]
\end{prop}

\begin{proof}
Let $\gamma_\diamond\in \Ah^*(X,L;\R)$ be a representative of $\g_\diamond$.
It follows from the definition of $\g_\diamond$ that in $H^*(\cone),$ we have $[(\gamma_\diamond,0)]=[(0,-1)].$ By Lemma~\ref{lm:dpsi} we then have
\[
\d_s\Ob
=P(\d_s\Psi)
=P(\d_sf\cdot\nund[(0,1)])
=-P(\d_sf\cdot\nund[(\gamma_\diamond,0)])
=-P(\d_sf\cdot \d_\diamond\Psi)
=-\d_sf\cdot \d_\diamond\Ob.
\]
\end{proof}

\begin{proof}[Proof of Theorem~\ref{prop:wallcross}]
By assumption, $\d_sf=1.$ Therefore Proposition~\ref{prop:wallseries} gives $\d_s\Ob=- \d_\diamond\Ob$. Since $[L]=0$ and thus $\Coker \i = R_W$, this immediately implies the required equality of invariants.
\end{proof}

\subsection{Enhanced axioms}
The following axioms for the invariants $\ogwb$ are useful in combination with the open WDVV equations for carrying out recursive computations. See Section~\ref{comp}. Similar axioms for the invariants $\ogw$ were proved in~\cite{ST2}.
In the following, we assume that the subspace $W \subset \Hh^*(X,L;\R)$ and the bounding pair~$(\gamma_W,b)$ are as in Theorem~\ref{thm:esphere} or Theorem~\ref{thm:cases}. In particular, writing $W_i \subset W$ for the degree $i$ homogeneous part, by Lemma~5.11 of~\cite{ST2}, there is a natural inclusion
\[
W_2 \hookrightarrow H^2(X,L;\R).
\]
So, for $A \in W_2,$ and $\beta \in H_2(X,L;\Z),$ the pairing $\int_\beta A$ is well-defined.
\begin{prop}
\label{axioms}
The invariants $\ogwb$ of $(X,L)$ satisfy the following axioms.
Let $A_j \in W$ for $j = 1,\ldots,l$.
\begin{enumerate}
	\item (Degree)
		$\ogwb_{\beta,k}(A_1,\ldots,A_l)=0$ unless
		\begin{equation}\label{ax_deg}
		n-3+\mu(\beta)+k+2l = kn+ \sum_{j=1}^l |A_j|.
		\end{equation}
	\item (Unit / Fundamental class)
		\begin{equation}\label{ax_unit}
		\ogwb_{\beta,k}(1,A_{1},\ldots,A_{l-1})=
		\begin{cases}
			-1, & (\beta,k,l)=(\beta_0,1,1),\\
            P_\R(A_1), & (\beta,k,l) = (\beta_0,0,2),\\
			0, & \text{otherwise}.
		\end{cases}
		\end{equation}
	\item (Zero)
		\begin{equation}\label{ax_zero}
		\ogwb_{\beta_0,k}(A_1,\ldots,A_l)=
		\begin{cases}
		-1, & (k,l)=(1,1)\text{ and } A_1=1,\\
        P_\R(A_1 \smallsmile A_2), & (k,l) = (0,2),\\
		0, & \text{otherwise}.
		\end{cases}
		\end{equation}
	\item (Divisor)
		If $|A_l| = 2,$ then
		\begin{equation}\label{ax_divisor}
		\ogwb_{\beta,k}(A_1,\ldots,A_{l})=\int_\beta A_l \cdot\ogwb_{\beta,k}(A_1,\ldots,A_{l-1}).
		\end{equation}
\end{enumerate}
\end{prop}
\begin{proof}
When $[L] \neq 0,$ all values of $\ogwb$ that are defined coincide with the values of $\ogw$ and $P_\R = 0.$ So, the axioms proved for $\ogw$ in~\cite{ST2} imply the axioms for $\ogwb.$ Thus, we may assume that $[L] = 0$ and, as mentioned in Remark~\ref{rem:qder}, we can compute the invariants $\ogwb$ by taking derivatives of $\Ob.$ We first prove the unit axiom~\eqref{ax_unit} in detail. By assumption, the unit element of the cohomology $1 \in \Hh^*(X,L;\R)$ belongs to $W$ and we denote by $\d_1$ the corresponding directional derivative. It is shown in~\cite{ST2} that $\d_1 \Oh = - T^{\beta_0}s.$ It follows from Lemma~\ref{lm:cunit} that $\d_1\q^{\gamma_W}_{\emptyset,0} = T^{\beta_0}\gamma_W.$ Thus
\[
\d_1 \Psi(\gamma_W,b) = \d_1[(\q^{\gamma_W}_{\emptyset,0},-\Oh)] = T^{\beta_0}[(\gamma_W,s)] = T^{\beta_0}(a(\Gamma_W) - x(s)).
\]
So, diagram~\eqref{eq:PQ} implies
\[
\d_1 \Ob = \d_1 P \Psi = P \d_1\Psi = T^{\beta_0}(P(a(\Gamma_W)) - P(x(s))) = T^{\beta_0}(P_\Qh(\Gamma_W) - s).
\]
The unit axiom follows by taking derivatives and using equation~\eqref{eq:PR}. The remaining axioms follow by a similar combination of the arguments in~\cite{ST2} with Lemmas~\ref{lm:cdeg},~\ref{lm:czero} and~\ref{lm:cdiv}.
\end{proof}

\section{Relative quantum cohomology and open WDVV}\label{sec:rqcowdvv}
\subsection{Preliminaries}\label{ssec:prelim}
Throughout Section~\ref{sec:rqcowdvv}, we operate under assumptions~\ref{assump:sfa} and~\ref{assump:b}. Recalling the definition of $P_\R$ from equation~\eqref{eq:PR}, we take
\begin{equation}\label{eq:W'}
W' = \Ker(P_\R|_W) \subseteq W.
\end{equation}
Lemma~\ref{lm:PRr} implies that $W = W' \oplus \Im y$ and $\rho|_{W'} : W' \to U$ is injective.
Denote by $\kappa : W \to W$
the projection to $W'$ along $\im y$ and denote by $\kappa^* : \Qh_W \to \Qh_W$ the induced homomorphism. Define $\gamma_{W'} \in \mI_W\Ah^*(X,L;\Qh_W )$ by $\gamma_{W'} = \kappa^* \gamma_W.$ Since $\rho|_{W'}$ is injective, there exists a closed $\gamma_U \in \mI_U A^*(X;\Qh_U)$ such that $[\gamma_U]=\g_U \in U\otimes \Qh_U$ and
\begin{equation}\label{eq:gUW'}
\rho^*\gamma_U=\gamma_{W'}.
\end{equation}

\begin{lm}\label{lm:rhostar}
For $\eta_1,\ldots,\eta_l \in A^*(X;\Qh_W),$ we have
\[
[\rho^*(\q_{\emptyset,l}^{\gamma_U}(\eta_1,\ldots,\eta_l))] =
[\q_{\emptyset,l}^{\gamma_W}(\eta_1,\ldots,\eta_l)] \in H^*(X;\Qh_W).
\]
\end{lm}
\begin{proof}
By equation~\eqref{eq:gUW'}, we have
\[
\rho^*(\q_{\emptyset,l}^{\gamma_U}(\eta_1,\ldots,\eta_l)) = \q_{\emptyset,l}^{\rho^*\gamma_U}(\eta_1,\ldots,\eta_l) = \q_{\emptyset,l}^{\gamma_{W'}}(\eta_1,\ldots,\eta_l).
\]
By definition, $\gamma_W$ and $\gamma_{W'},$ and consequently $\q^{\gamma_{W}}_{\emptyset,l}(\eta_1,\ldots,\eta_l)$ and $\q_{\emptyset,l}^{\gamma_{W'}}(\eta_1,\ldots,\eta_l),$ are differential form valued formal functions on $W$ that agree when restricted to $W' \subset W.$
Since $W = W' \oplus \Im y$, it is enough to
consider dependence of $[\q^{\gamma_{W}}_{\emptyset,l}(\eta_1,\ldots,\eta_l)]$ on $\g_\diamond=y(1)\in \Hh^*(X,L).$
Since $[\d_\diamond\gamma_W] = \d_\diamond \Gamma_W = \g_\diamond \in \Ker(\rho)$, there  is $\eta\in A^*(X;\R)$ such that $\d_\diamond\gamma_W=d\eta$.
By Lemma~\ref{lm:qecm}, we have
\begin{equation}\label{eq:ddiamW}
\d_\diamond\q^{\gamma_{W}}_{\emptyset,l}(\eta_1,\ldots,\eta_l)
=
\q^{\gamma_{W}}_{\emptyset,l+1}(\d_\diamond\gamma_W,\eta_1,\ldots,\eta_l)
=
d\q^{\gamma_{W}}_{\emptyset,l+1}(\eta,\eta_1,\ldots,\eta_l).
\end{equation}
Moreover,
\begin{equation}\label{eq:ddiamW'}
\d_\diamond\q^{\gamma_{W'}}_{\emptyset,l}(\eta_1,\ldots,\eta_l)
=
\q^{\gamma_{W'}}_{\emptyset,l+1}(\d_\diamond\gamma_{W'},\eta_1,\ldots,\eta_l)
=0.
\end{equation}
Thus, we have $[\d_\diamond\q^{\gamma_{W}}_{\emptyset,l}(\eta_1,\ldots,\eta_l)]
=
[\d_\diamond\q^{\gamma_{W'}}_{\emptyset,l}(\eta_1,\ldots,\eta_l)]$, and therefore, $[\q^{\gamma_{W}}_{\emptyset,l}(\eta_1,\ldots,\eta_l)]
=
[\q^{\gamma_{W'}}_{\emptyset,l}(\eta_1,\ldots,\eta_l)]$.
\end{proof}

Throughout Section~\ref{sec:rqcowdvv}, we write
\begin{equation}\label{eq:V}
V:=\pi^{-1}(U\otimes \Qh_W)\subset \cone.
\end{equation}
\begin{lm}\label{lm:nunendo}
For all $u \in W$ we have
\[
\d_u\nund(V) \subset V.
\]
\end{lm}
\begin{proof}
It suffices to show that $\pi(\d_u\nund(V)) \subset U\otimes \Qh_W.$ Indeed, let $v\in V$ and let $(\eta,\xi) \in \cone$ with $[(\eta,\xi)] = v.$  By Lemma~\ref{lm:rhostar}, we calculate
\begin{equation}\label{eq:dupi}
\pi(\d_u \nund(v)) = [\d_u \q^{\gamma_W}_{\emptyset,1}(\eta)] = [\q^{\gamma_W}_{\emptyset,2}(\eta,\d_u\gamma_W)]
=[\rho^* \q^{\gamma_U}_{\emptyset,2}(\eta,\d_u\gamma_W)] = \rho^*(\pi(v) \star_U \rho(u)).
\end{equation}
Assumption~\ref{assump:sfa} implies that
\begin{equation}\label{eq:inU}
\rho^*(\pi(v) \star_U \rho(u)) \in U\otimes \Qh_W.
\end{equation}
The lemma follows by combining~\eqref{eq:dupi} and~\eqref{eq:inU}.
\end{proof}

\subsection{Relative quantum product}

Recall that
in Section~\ref{sssec:memintro} we defined
$\Qhh_W=\Qh_W\otimes_{\L_c}\L$ and
\[
\coneq:=(\cone/x(sR_W))\otimes_{\Qh_W}\Qhh_W,
\qquad
QH_U(X,L) := W \otimes \Qhh_W.
\]
As noted, there is a natural isomorphism $\Hh^*(X,L;\Qhh_W)\simeq H^*(\coneq)$. In particular, we can think of $QH_U(X,L)$ as a subspace of $H^*(\coneq)$.

For $\eta\in \cone$ and $r\in R_W$, we have
$\memh(\eta+x(s\cdot r))=\memh(\eta)-c\cdot x(s\cdot r)$.
Thus, $\memh$ induces a map
$
\memhh: \coneq\lrarr \coneq
$.

\begin{lm}\label{cor:assocnunh}
The map $\memhh$ inherits the properties of $\memh$. Specifically, \begin{enumerate}
\item\label{it:nunha}
$\memhh$ is a chain map and is invariant under gauge equivalence,
\item\label{it:nunhb}
$\d_u\memhh\circ \d_v\memhh$ is chain homotopic to $(-1)^{|u||v|}\d_v\memhh\circ\d_u\memhh$ for all vector fields $u,v\in QH_U(X,L)$.
\end{enumerate}
\end{lm}
\begin{proof}
Statement~\ref{it:nunha} follows from Lemmas~\ref{lm:nunchain} and~\ref{lm:memhinvt}.
Statement~\ref{it:nunhb} follows from Theorem~\ref{lm:assoc}.
\end{proof}

Denote by $\nundh$ the map induced by $\memhh$ on cohomology.
\begin{lm}\label{lm:nunhendo}
$\d_u\nundh(QH_U(X,L))\subset QH_U(X,L)$ for all $u\in QH_U(X,L)$.
\end{lm}
\begin{proof}
Recall that $\Im(x)=\Ker(\pi)$.
Thus, we have the following commutative diagram.
\[
\xymatrix{
H^*(\cone)\ar[rr]^{h}\ar[dr]^{\pi}& & H^*(\cone/x(sR_W))\ar[dl]_{\tilde\pi}\\
& H^*(X;Q_W)&
}
\]
By Lemma~\ref{lm:nunendo},
we have
\[
(\d_u\nundh)(\tilde{\pi}^{-1}(U\otimes Q_W)) =(\d_u\nundh)(h(V))=h((\d_u\nund)(V))\subset h(V)=\tilde{\pi}^{-1}(U\otimes Q_W).
\]
Tensoring with $\Qhh_W$, we get the required result.
\end{proof}

Keeping in mind Lemma~\ref{lm:nunhendo}, define
\[
\mem:QH_U(X,L)\otimes QH_U(X,L) \lrarr QH_U(X,L)
\]
by
\[
\mem(u,v)=\d_u\nundh(v).
\]

\begin{lm}\label{lm:memform}
\[
\mem(u,v)=
[\qg_{\emptyset,2}(u,v)]+ (-1)^{n}(\d_{u}\d_{v}\qbg_{-1,0} - \d_{\d_u v} \qbg_{-1,0})\cdot \g_\diamond.
\]
\end{lm}

\begin{proof}
By assumption~\ref{assump:b}, we have $\d_v\big(\int_Lb\big)=0$.
Again by assumption~\ref{assump:b}, we have $[\partial_u \gamma_W] = u, \; [\partial_v \gamma_W] = v.$ So, by Lemmas~\ref{cl:symmetry} and~\ref{cl:cyclic} and the Maurer-Cartan equation,
\[
\d_v\q^{\gamma_W,b}_{-1,0} =\q^{\gamma_W,b}_{-1,1}(\d_v\gamma_W)
+\langle\q_{0,0}^{\gamma_W,b},\d_vb\rangle
=\q^{\gamma_W,b}_{-1,1}(\d_v\gamma_W)+c\cdot\langle 1,\d_vb\rangle
=\q^{\gamma_W,b}_{-1,1}(\d_v\gamma_W) = \q^{\gamma_W,b}_{-1,1}(v).
\]
Consequently,
\[
(\d_u\q^{\gamma_W,b}_{-1,1})(v)
= \d_u(\q^{\gamma_W,b}_{-1,1}(v)) - \q^{\gamma_W,b}_{-1,1}(\d_uv)
=\d_{u}\d_{v}\qbg_{-1,0} - \d_{\d_u v} \qbg_{-1,0}.
\]
Therefore, $\mem(u,v)$ is given at chain level by
\begin{multline*}
\d_u\memhh(v)
=(\q^{\gamma_W}_{\emptyset,2}(\d_u\gamma_W,v),
(-1)^{n+1}\d_u\q^{\gamma_W,b}_{-1,1}(v))=\\
=(\q^{\gamma_W}_{\emptyset,2}(u,v), (-1)^{n+1}(\d_u\d_v\q^{\gamma_W,b}_{-1,0}-\d_{\d_u v}\q^{\gamma_W,b}_{-1,0})).
\end{multline*}
It follows from the definition of $\g_\diamond$ that in $H^*(\cone)$, and therefore in $H^*(\coneq)$, we have $[(\gamma_\diamond,0)]=[(0,-1)].$
The desired formula follows.
\end{proof}

\begin{lm}\label{lm:memcommut}
\leavevmode
\begin{enumerate}
\item\label{it:gc}
For vector fields $u,v\in QH_U(X,L)$, we have $\mem(u,v) = (-1)^{|u||v|}\mem(u,v).$
\item\label{it:dp}
The map $\mem$ has degree zero. That is, $|\mem(u,v)| = |u| + |v|.$
\end{enumerate}
\end{lm}

\begin{proof}
For any $f \in \Qhh_W,$ by definition of the graded Lie bracket,
\[
\d_u \d_v f - (-1)^{|u||v|} \d_v \d_u f = \d_{[u,v]}f.
\]
On the other hand, since the canonical connection on affine space is symmetric, we have
\[
\d_{\d_u v}f - (-1)^{|u||v|} \d_{\d_v u}f = \d_{[u,v]}f.
\]
Thus $\d_u \d_v f - \d_{\d_u v} f$ is graded symmetric in $u,v.$ Taking $f = \q^{\gamma_W,b}_{-1,0},$ and using also
Lemma~\ref{lm:csymm}, we see that the right-hand side in Lemma~\ref{lm:memform} is graded symmetric in $u,v,$  which implies part~\ref{it:gc}. Keeping in mind the shifted grading of $R[-n-1]$ in the definition of $\cone$, part~\ref{it:dp} follows from Lemmas~\ref{lm:cdeg} and~\ref{deg_str_map}.
\end{proof}

\begin{lm}\label{lm:memassoc}
For vector fields $u,v,w,$ we have $\mem(\mem(u,v),w) = \mem(u,\mem(v,w)).$
\end{lm}

\begin{proof}
By Lemma~\ref{lm:memcommut} and Lemma~\ref{cor:assocnunh},
\begin{multline*}
\mem(\mem(u,v),w)
=(-1)^{(|u|+|v|)|w|}\mem(w,\mem(u,v))=\\
=(-1)^{(|u|+|v|)|w|}\d_w\nundh(\mem(u,v))
=(-1)^{(|u|+|v|)|w|}\d_w\nundh(\d_u\nundh(v))
=(-1)^{|v||w|}\d_u\nundh(\d_w\nundh(v))= \\
=(-1)^{|v||w|}\mem(u,\mem(w,v))
=\mem(u,\mem(v,w)).
\end{multline*}
\end{proof}

\begin{proof}[Proof of Theorem~\ref{thm:mem}]
Graded commutativity and associativity are given by Lemmas~\ref{lm:memcommut}\ref{it:gc} and~\ref{lm:memassoc}, respectively.
Invariance under gauge equivalence follows from Lemma~\ref{cor:assocnunh}\ref{it:nunha}.
\end{proof}

\subsection{Open WDVV}\label{ssec:OWDVV}
In the following, we make some calculations that will be useful in the proof of Theorem~\ref{thm:OWDVV}. As in Section~\ref{sssec:owdvv}, abbreviate
\[
\Phi = \Phi_U \in Q_U, \qquad \Ob = \Ob(\gamma_W,b) \in R_W.
\]

\begin{lm}\label{lm:memofx}
Let $r\in \Coker(\i)$.
\begin{enumerate}
\item
$\nund(\xb(r))=-c\cdot \xb(r),$
\item
$\d_u\nund(\xb(r))=-\d_u c\cdot \xb(r),$ \: for $u\in W\oplus S.$
\end{enumerate}
\end{lm}
\begin{proof}
Let $r'\in R_W$ be a representative of the class of $r$ in $\Coker(\i).$ Then
\[
\nund(\xb(r))=\nund(x(r')) =[(0,c\cdot r')]
=x(-cr')=\xb(-p(cr'))
=
-c\cdot\xb(r),
\]
where the last equality uses the fact that $c\in \Qh_W$ in the case $[L]\ne 0$ (see Theorem~\ref{prop:formforc}),
and holds trivially when $[L]=0$.
Consequently,
\[
\d_u\nund(\xb(r))=-\d_uc\cdot \xb(r).
\]
\end{proof}

\begin{lm}\label{lm:phipsi}
$\pi(\Psi(\gamma_W,b)) = \rho^*(\nabla\Phi_U).$
\end{lm}
\begin{proof}
Recall from Section~\ref{sssec:gwp} that
\[
\nabla \Phi_U = [\q^{\gamma_U}_{\emptyset,0}] \in H^*(X;Q_U).
\]
So, it follows from Lemma~\ref{lm:rhostar} that
\[
\rho^* \nabla\Phi_U = [\rho^*\q_{\emptyset,0}^{\gamma_U}] = [\q_{\emptyset,0}^{\gamma_{W}}] = \pi(\Psi(\gamma_W,b)).
\]
\end{proof}
\begin{lm}\label{lm:rhodphi}
For
$l\in I_U$, we have
$
\rho^*(\d_l\Phi)
=
\int_X\q^{\gamma_W}_{\emptyset,0}\wedge\d_l\gamma_U
$.
\end{lm}
\begin{proof}
By definition of the gradient, we have
\[
\rho^*(\d_l\Phi) = \rho^*\langle \nabla\Phi,\g_l \rangle_X =\langle \rho^*\nabla\Phi,\g_l \rangle_X.
\]
Moreover,
\[
\int_X\q^{\gamma_W}_{\emptyset,0}\wedge\d_l\gamma_U = \langle \pi(\Psi(\gamma_W,b)),\g_l \rangle_X.
\]
The claimed equality follows from Lemma~\ref{lm:phipsi}.
\end{proof}

Recall the definitions of $W'$ and $V$ from equations~\eqref{eq:W'} and~\eqref{eq:V} respectively.

\begin{lm}\label{lm:inj}\label{it:ds}
$\Im\bar x \oplus a(W'\otimes Q_W) = V.$
\end{lm}
\begin{proof}
Since $P : H^*(\cone) \to \Coker \i$ splits the short exact sequence
\[
0 \lrarr \Coker \i \overset{\bar x}{\lrarr} H^*(\cone) \overset{\pi}{\lrarr} \Ker \i \to 0,
\]
we obtain
\begin{equation}\label{eq:dsHC}
\Ker P \oplus \Im \bar x = H^*(\cone).
\end{equation}
Thus, since $\Im \bar x \subset V,$ it suffices to show that $V \cap \Ker P = a(W' \otimes \Qh_W).$

By Lemma~\ref{lm:PRP}, we have $\Ker P = a(\Ker P_\Qh).$ Thus,
\[
V \cap \Ker P = \pi^{-1}(U\otimes\Qh_W) \cap a(\Ker P_\Qh).
\]
On the other hand,
\[
\qquad W' \otimes \Qh_W = \rho_\Qh^{-1}(U\otimes \Qh_W) \cap \Ker P_\Qh.
\]
Diagram~\eqref{eq:tris} gives $\pi \circ a = \rho_\Qh.$ So, for $v = a(w),$ we have
\[
w \in \rho_\Qh^{-1}(U\otimes \Qh_W) \Leftrightarrow \rho_\Qh(w) \in U\otimes \Qh_W \Leftrightarrow \pi\circ a(w) \in U\otimes \Qh_W \Leftrightarrow v \in \pi^{-1}(U\otimes \Qh_W).
\]
Therefore, $V \cap \Ker P = a(W' \otimes \Qh_W)$ and the lemma follows.
\end{proof}

Recall from Section~\ref{sssec:owdvv} that $\Delta_j\in U$, $j\in I_U$, and $\g_j\in W'$, $j\in I_{W'}$ are bases such that $\rho(\g_j)=\Delta_j$.
Write $\Upsilon_j:=a(\g_j)\in H^*(\cone)$, for $j\in I_{W'}$, and $\Upsilon:=a(\g_W)$.

\begin{lm}\label{lm:memofg}
Let $u\in W$ and $v\in W\oplus S$, let $v_W$ be the projection of $v$ to $W,$ and let $\bar{u} = \rho(u),\bar{v} = \rho(v_W).$ We have
\[
\d_v\nund(\d_u\Upsilon)=
\xb(\d_v\d_u\Ob)+\sum_{\substack{m\in I_{W'}\\ l\in I_U}}\rho^*\d_{\bar v}\d_{\bar u}\d_l\Phi \cdot g^{lm}\cdot\Upsilon_m.
\]
\end{lm}
\begin{proof}

By Lemmas~\ref{lm:nunendo} and~\ref{it:ds}, we obtain a unique decomposition
\begin{equation}\label{eq:dmemdecomp}
\d_v\nund(\d_u\Upsilon)
=
\xb(r)+\sum_{\;\;\; m \in I_{W'}}r^m\Upsilon_m, \qquad r \in R_W, \quad r^j \in \Qh_W.
\end{equation}
To compute $r,$ recall that $P\circ \xb=\Id$ and $\Upsilon_m\in\Ker P$. By Lemma~\ref{lm:dpsi}, we have
\begin{align*}
r&= P(\xb(r))
=P(\d_v\nund(\d_u\Upsilon))=P(\d_v\d_u\Psi)
=\d_v\d_uP\Psi=\d_v\d_u\Ob.
\end{align*}
To find $r^m$, take $l\in I_U$ and compute
\begin{align*}
\langle\pi(\d_v\nund(\d_u\Upsilon)),\Delta_l\rangle_X &=
\langle \sum_{m\in I_{W'}} r^m\pi(\Upsilon_m),\Delta_l\rangle_X\\
=&
\sum_{m\in I_{W'}} r^m\langle\Delta_m,\Delta_l\rangle_X\\
=&
\sum_{m\in I_{W'}}r^m g_{ml}.
\end{align*}
On the other hand,
\begin{align*}
\langle\pi(\d_v\nund(\d_u\Upsilon)),\Delta_l\rangle_X &=
\int_X \d_v\q^{\gamma_W}_{\emptyset,1}(\pi(\d_u\Upsilon))\wedge\d_l\gamma_U\\
=&
\int_X\d_v\d_u\q^{\gamma_W}_{\emptyset,0}\wedge\d_l\gamma_U\\
=&
\d_v\d_u\int_X\q^{\gamma_W}_{\emptyset,0}\wedge\d_l\gamma_U,\\
\shortintertext{and by Lemma~\ref{lm:rhodphi},}
=&\d_v\d_u\rho^*(\d_l\Phi) \\
=&\rho^*(\d_{\bar v}\d_{\bar u}\d_l\Phi).
\end{align*}
So,
\[
\sum_{m\in I_{W'}}r^m g_{ml}=
\rho^*(\d_{\bar v}\d_{\bar u}\d_l\Phi).
\]
Multiplying by the inverse matrix $(g^{ml})$, we get
\[
r^m=\sum_{l\in I_U}\rho^*(\d_{\bar v}\d_{\bar u}\d_l\Phi) \cdot g^{lm}.
\]
\end{proof}

\begin{proof}[Proof of Theorem~\ref{thm:OWDVV}]
We deduce the result from
Theorem~\ref{lm:assoc}, as follows.
By Lemmas~\ref{lm:memofg} and~\ref{lm:memofx}, compute
\begin{align}
\d_u\nund(\d_v&\nund(\d_w\Upsilon))=
\d_u\nund\big(\xb(\d_v\d_w\Ob) +\sum_{m\in I_{W'}, l\in I_U}\rho^*\d_{\bar v}\d_{\bar w}\d_l\Phi \cdot g^{lm}\cdot\Upsilon_m\big)\notag
\\
=&\d_u\nund(\xb(\d_v\d_w\Ob)) +\sum_{m\in I_{W'}, l\in I_U}(-1)^{(|v|+|w|+|\Delta_l|)|u|}\rho^*\d_{\bar v}\d_{\bar w}\d_l\Phi \cdot g^{lm}\cdot\d_u\nund(\Upsilon_m)\notag\\
=&\xb(-\d_u c\cdot\d_v\d_w\Ob)
+ \notag\\
& + \sum_{m\in I_{W'}, l\in I_U}(-1)^{(|v|+|w|+|\Delta_l|)|u|}\rho^*\d_{\bar v}\d_{\bar w}\d_l\Phi \cdot g^{lm}
\big(\xb(\d_u\d_m\Ob)
+\sum_{\substack{i\in I_{W'}\\ j\in I_U}}\rho^*\d_{\bar u}\d_m\d_j\Phi\cdot g^{ji}\cdot\Upsilon_i\big)\notag\\
=&
\xb\big(-\d_uc\cdot\d_v\d_w\Ob
+
\sum_{m\in I_{W'}, l\in I_U}(-1)^{(|v|+|w|+|\Delta_l|)|u|}\rho^*\d_{\bar v}\d_{\bar w}\d_l\Phi  \cdot g^{lm}\cdot
\d_u\d_m\Ob\big)+ \notag \\
&\hspace{5em}+
\sum_{m,i\in I_{W'}, l,j\in I_U}(-1)^{(|v|+|w|+|\Delta_l|)|u|}\rho^*\d_{\bar v}\d_{\bar w}\d_l\Phi\cdot g^{lm}\cdot
\rho^*\d_{\bar u}\d_m\d_j\Phi\cdot g^{ji}\cdot\Upsilon_i \notag\\
=&
\xb\big(-\d_uc\cdot\d_v\d_w\Ob
+
\sum_{m\in I_{W'}, l\in I_U}
(-1)^{(|v|+|w|+|\Delta_l|)|\Delta_m|}  \d_u\d_m\Ob
\cdot g^{lm}\cdot
\rho^*\d_{\bar v}\d_{\bar w}\d_l\Phi
\big)+ \notag \\
&\hspace{5em}+
\sum_{m,i\in I_{W'}, l,j\in I_U}(-1)^{(|v|+|w|+|\Delta_l|)|u|}\rho^*\d_{\bar v}\d_{\bar w}\d_l\Phi\cdot g^{lm}\cdot
\rho^*\d_{\bar u}\d_m\d_j\Phi\cdot g^{ji}\cdot\Upsilon_i.
\notag\\
\shortintertext{Since $g^{lm} = (-1)^{|l||m|}g^{ml}$ and $g^{lm} = 0$ unless $|m| \equiv |l| \pmod 2$, we continue}
=&
(-1)^{|v||w|}\xb\big(-\d_u c \cdot\d_w\d_ v\Ob
+
\sum_{l \in I_{W'}, m \in I_U}\d_u\d_l\Ob\cdot g^{lm}\cdot
\rho^*\d_m\d_{\bar w}\d_{\bar v}\Phi\big)+ \notag \\
&\hspace{5em}+
(-1)^{(|v|+|w|)|u|}\sum_{m,i\in I_{W'}, l,j\in I_U}\rho^*\d_{\bar v}\d_{\bar w}\d_l\Phi\cdot g^{lm}\cdot
\rho^*\d_m\d_{\bar u}\d_j\Phi\cdot g^{ji}\cdot\Upsilon_i.
\label{eq:pe}
\end{align}
Similarly,
\begin{align}
\d_v\nund(\d_u&\nund(\d_w\Upsilon))= \notag \\
& =  \xb\big(-\d_vc\cdot\d_u\d_w\Ob
+
\sum_{m\in I_{W'}, l\in I_U}(-1)^{(|u|+|w|+|\Delta_l|)|v|}\rho^*\d_{\bar u}\d_{\bar w}\d_l\Phi \cdot g^{lm}\cdot
\d_v\d_m\Ob\big)+ \notag \\
&\hspace{5em}+
\sum_{m,i\in I_{W'}, l,j\in I_U}(-1)^{(|u|+|w|+|\Delta_l|)|v|}\rho^*\d_{\bar u}\d_{\bar w}\d_l\Phi \cdot g^{lm}\cdot
\rho^*\d_{\bar v}\d_m\d_j\Phi \cdot g^{ji}\cdot\Upsilon_i \notag\\
& =  (-1)^{|v|(|u| + |w|)}\xb\big(-\d_u\d_w\Ob\cdot\d_vc
+
\sum_{m\in I_{W'}, l\in I_U}\rho^*\d_{\bar u}\d_{\bar w}\d_l\Phi \cdot g^{lm}\cdot
\d_m\d_v\Ob\big)+ \notag \\
&\hspace{5em}+
(-1)^{|v|(|u| + |w|)}\sum_{m,i\in I_{W'}, l,j\in I_U}\rho^*\d_{\bar u}\d_{\bar w}\d_l\Phi \cdot g^{lm}\cdot
\rho^*\d_m\d_{\bar v}\d_j\Phi \cdot g^{ji}\cdot\Upsilon_i. \label{eq:pe'}
\end{align}
By Theorem~\ref{lm:assoc}, quantities~\eqref{eq:pe} and~\eqref{eq:pe'} are equal after multiplying one by $(-1)^{|u||v|}.$
To get the open WDVV equation~\eqref{1st}, apply $P$.
\end{proof}

\begin{rem}
From equation~\eqref{eq:pe} we can obtain the standard WDVV equation for closed genus zero Gromov-Witten invariants by applying $\pi$ and pairing via $\langle \;,\;\rangle_X$ with $\g_e$.
\end{rem}

\begin{proof}[Proof of Corollary~\ref{cor:owdvv}]

Use Lemma~\ref{lm:dsomega} as follows.
If $[L]=0,$ then $c=\d_s\Ob$ and the equations follow.
If $[L]\ne 0,$ then $c=\d_s\Omega$. Let $z:R_W/\Qh_W\to R_W$ be the unique right-inverse of $q$ such that $D\circ z=0$. In particular, $z(\Ob)=\Omega$. Observe that if $u \in W$, then for all $g \in R_W/\Qh_W$ we have $\d_u (z(g)) = z( \d_u g).$ The desired equation then follows from applying $z$ to equation~\eqref{1st}.
\end{proof}

\section{Computations for projective space}\label{comp}
The objective of this section is to describe a recursive process for the computation of our invariants for $(X,L) = (\P^n,\RP^n)$ $n$ odd.
As in Section~\ref{sssec:cpn},
take $\omega=\oFS$ the Fubini-Study form, $J=J_0$ the standard complex structure,
and $\sly=H_2(X,L;\Z)$. Equip $(X,L)$ with a relative spin structure.
Thus, Theorem~\ref{thm:esphere} holds, and we have a bounding pair $(\gamma,b)$ over $W=\Hh^*(X,L;\R)$ with $b$ point-like.
Abbreviate $\g_j = [\omega^j] \in \Hh^*(X,L;\R)$ and $\g_j = [\omega^j] \in H^*(X;\R).$ Observe that $[L]=0\in H_n(X;Z).$ So, together with $\g_\diamond=y(1)$, the classes $\g_j$ form a basis of $\Hh^*(X,L;\R)$. Denote by $t_j\in \R[[W]]$ the coordinates on $W$ corresponding to $\g_j$ for $j=0,\ldots,n,\diamond$. Let $s\in S$ be the coordinate on $S$ with $\int_L b=s$.
Identify $H_2(X,L;\Z)$ with $\Z$ so that the non-negative integers correspond to classes $\beta \in H_2(X,L;\Z)$ with $\omega(\beta) \geq 0.$ By Lemma~\ref{lm:PRP}, choose $P:  H^*(\cone) \to \Coker \i$ to be the unique left inverse to $\bar x$ satisfying conditions~\eqref{eq:bP} and~\eqref{eq:PR} such that $P_\R(\Gamma_j) = 0$ for $j = 0,\ldots,n.$ Thus, by definition~\eqref{eq:W'} we have $W' = \spa\{\g_j\}_{j = 0}^n.$

\begin{lm}\label{lm:init}
The relative spin structure on $(X,L)$ can be chosen so that $\ogwb_{1,2}=2.$
\end{lm}

\begin{proof}
For any space, denote by $pt$ the map from it to a point.
Recall Remark~\ref{rem:qder}. Since $\gamma \in \mI_W\Ah^*(X,L;\Qh_W )$ and $b \in \mJ_{W}A^*(L;R_W)$,
Lemmas~\ref{cl:cyclic},~\ref{lm:qzero}, and Proposition~\ref{prop:iof}~\ref{prop:pushpull},
imply that the coefficient of $T^1$ in $-\d_s^2\Ob|_{s=t_j=0}$ is
\begin{align}
(\d_s^2\q^{\beta_1;\gamma,b}_{-1,0})|_{s=t_j=0}
= &
(\d_s^2\q^{\beta_1;0,b}_{-1,0})|_{s=t_j=0} \notag\\
=&
\langle\q^{\beta_1}_{1,0}(\d_s b),\d_s b\rangle \notag\\
=&
(-1)^{n+1} pt_*((evb_0)_*(evb_1^*\d_s b)\wedge \d_s b) \notag \\
=&
(-1)^{1+n\cdot\rdim evb_0} pt_*(\d_s b \wedge (evb_0)_*(evb_1^*\d_s b)) \notag \\
=&
pt_*((evb_0)^*\d_s b\wedge (evb_1^*\d_s b)). \label{eq:ds2}
\end{align}

Let $Y = L \times L \setminus \triangle$ and let $Z = (evb_0 \times evb_1)^{-1}(Y) \subset \M_{2,0}(1).$ It is easy to see that
\[
g = (evb_0 \times evb_1)|_Z : Z \to Y
\]
is a covering map. Indeed, the fiber of $g$ over a point $(x_1,x_2) \in Y$ can be identified with the pair of oriented lines in $\R P^n$ passing through $x_1$ and $x_2.$  We claim that for an appropriate choice of relative $Pin$ structure, the degree of $g$ is $-2.$ Indeed, the two points in the preimage of a point in $Y$ are conjugate disks of degree $1$ with two marked points. Proposition~5.1 of~\cite{SolomonThesis} shows that the covering transformation that interchanges these two disks is orientation preserving, so the degree is $\pm 2.$ Finally, Lemma 2.10 of~\cite{SolomonThesis} shows that one can choose the relative spin structure on $L$ so as to make the degree $-2.$

For $i = 1,2,$ let $p_i: L \times L \to L$ be the projections. Since $Y \subset L \times L$ and $Z \subset \M_{2,0}(\beta)$ are open dense subsets, using that $b$ is point-like, we obtain
\begin{multline*}
pt_*((evb_0)^*\d_s b\wedge (evb_1)^*\d_s b) = \int_Zg^*((p_1^* \d_s b \wedge p_2^* \d_s b)|_Y) = \\
= -2 \int_Y p_1^* \d_sb \wedge p_2^* \d_s b
= -2 \int_{L\times L} p_1^* \d_sb \wedge p_2^* \d_s b
=-2\left (\int_L \d_s b \right)^2 = -2.
\end{multline*}
Combining this calculation with equation~\eqref{eq:ds2} we obtain the desired result.
\end{proof}

\begin{proof}[Proof of Theorem~\ref{recursion}]
Use the axioms of $\GW$ given in~\cite[Section 2]{KontsevichManin} and~\cite[Chapter 7]{MS}, and the axioms of $\ogwb$ given in Proposition~\ref{axioms}.
As in Remark~\ref{rem:qder}, the invariants $\ogwb$ are given by derivatives of $\Ob$.
Note that, with the identification $H_2(\P^n;\Z)\simeq \Z$, we have
\[
w_\s(\hat\beta)\equiv \frac{n+1}{2}\cdot \hat\beta \pmod 2.
\]
As shown in Lemma 5.11 of~\cite{ST2}, the natural map $H^2(X,L;\R) \to \Hh^*(X,L;\R)$ is an isomorphism. So, the pairing of $\Gamma_1 = [\omega] \in \Hh^*(X,L;\R)$ with $\beta \in \sly = H_2(X,L;\Z) \simeq \Z$ is well-defined and the result is
\[
\int_\beta \Gamma_1 = \int_\beta \omega = \beta/2.
\]
We use this implicitly each time we invoke the divisor axiom~\eqref{ax_divisor} in the following argument.

The value $\ogwb_{1,2}=2$ is computed in Lemma~\ref{lm:init}.
Equating the coefficients of $T^1$
in equation~\eqref{eq:cor2} with $v=\g_{\frac{n-1}{2}},w=\g_1,$ and using the zero axiom~\eqref{ax_zero}
yields
\[
(-1)^{w_\s(0)}\GW_{0}(\Delta_\frac{n-1}{2},\Delta_1,\Delta_\frac{n-1}{2})\cdot \ogwb_{1,1}(\g_{\frac{n+1}{2}})=0,
\]
so
\[
\ogwb_{1,1}(\g_{\frac{n+1}{2}})=0.
\]
Equating the coefficients of $T^{2}$ in equation~\eqref{eq:cor2} with $v=\g_1,w=\g_n,$ evaluated at $s=t_j=0$ yields
\[
(-1)^{w_\s(1)}\GW_1(\Delta_1,\Delta_n,\Delta_n)\cdot \ogwb_{0,1}(\g_0)-\ogwb_{1,0}(\g_1,\g_n)\cdot \ogwb_{1,2}=0,
\]
so by the zero axiom~\eqref{ax_zero}, the divisor axiom~\eqref{ax_divisor}, and Lemma~\ref{lm:init},
\begin{equation}\label{eq:ogw10n}
\ogwb_{1,0}(\g_n)=(-1)^{\frac{n+3}{2}}.
\end{equation}

For convenience, we use $[X](F)$ to denote the  coefficient of $X$ in the power series $F$.
To prove recursion~\ref{reduce_smallest},
apply $\d_s^k\d_I=\d_s^k\d_{j_3}\cdots \d_{j_l}$ to equation~\eqref{eq:cor1} with $v=\g_{j_1-1},w=\g_1,u=\g_{j_2}$, evaluate at $s=t_j=0$, and consider the coefficients of $T^\beta$.
Using the zero axiom for $\GW$, we can single out instances of $\ogwb_\beta$ and compute
\begin{multline*}
[T^{\beta}](\d_s^k\d_I(\d_{v}\d_{w}\d_j\Phi\cdot\d_{n-j}\d_{u}\Ob)|_{s=t_j=0})
=
\ogwb_{\beta,k}(\g_{j_1},\ldots,\g_{j_l})
+\\+
\sum_{\substack{\varpi(\hat\beta)+\beta_1=\beta\\ I_1\sqcup I_2= I}}\sum_{i=0}^n
(-1)^{\frac{(n+1)\hat\beta}{2}}\GW_{\hat\beta}(\Delta_1,\Delta_{j_1-1},\Delta_{I_1},\Delta_i) \ogwb_{\beta_1,k}(\g_{n-i},\g_{j_2},\g_{I_2}),
\end{multline*}
\begin{multline*}
[T^{\beta}](\d_s^k\d_I(\d_{u}\d_s\Ob\cdot\d_{v}\d_{w}\Ob)|_{s=t_j=0})
=\\=
\sum_{\substack{\beta_1+\beta_2=\beta\\k_1+k_2=k\\ I_1\sqcup I_2=I}} \binom{k}{k_1}\ogwb_{\beta_1,{k_1}}(\g_1,\g_{j_1-1},\g_{I_1}) \ogwb_{\beta_2,{k_2+1}}(\g_{j_2},\g_{I_2}),
\end{multline*}
\begin{multline*}
[T^{\beta}](\d_s^k\d_I (\d_{u}\d_{w}\d_j\Phi\cdot\d_{n-j}\d_{v}\Ob)|_{s=t_j=0})
=
\ogwb_{\beta,k}(\g_{j_1-1},\g_{j_2+1},\g_{j_3},\ldots,\g_{j_l})
+\\
+\sum_{\substack{\varpi(\hat\beta)+\beta_1=\beta\\ I_1\sqcup I_2= I}}\sum_{i=0}^n (-1)^{\frac{(n+1)\hat\beta}{2}} \GW_{\hat\beta}(\Delta_1,\Delta_{j_2},\Delta_{I_1},\Delta_i) \ogwb_{\beta_1,k}(\g_{n-i},\g_{j_1-1},\g_{I_2}),
\end{multline*}
\begin{multline*}
[T^{\beta}](\d_s^k\d_I (\d_{v}\d_s\Ob\cdot\d_{u}\d_{w}\Ob)|_{s=t_j=0})
=\\=
\sum_{\substack{\beta_1+\beta_2=\beta\\k_1+k_2=k\\I_1\sqcup I_2=I}}\binom{k}{k_1}\ogwb_{\beta_1,{k_1}}(\g_1,\g_{j_2},\g_{I_1}) \ogwb_{\beta_2,{k_2+1}}(\g_{j_1-1},\g_{I_2}).
\end{multline*}
Substituting the expressions in~\eqref{eq:cor1} gives the required recursion.

Recursion~\ref{reduce_bd} follows from
applying $\d_s\d_I=\d_s^{k-2}\d_{j_1}\cdots \d_{j_l}$ to equation~\eqref{eq:cor2} with $v=\g_1,w=\g_n$, evaluating at $s=t_j=0$, and considering the coefficients of $T^{\beta+1}$.
We find that
\begin{multline*}
[T^{\beta+1}](\d_s^{k-2}\d_I (\d_{v}\d_{w}\d_j\Phi\cdot\d_{n-j}\d_s\Ob)|_{s=t_j=0})=\\
=\sum_{\substack{\varpi(\hat\beta)+\beta_1=\beta+1\\I_1\sqcup I_2=I}}\sum_{i=0}^n (-1)^{\frac{(n+1)\hat\beta}{2}} \GW_{\hat\beta}(\Delta_1,\Delta_{n},\Delta_{I_1},\Delta_i) \ogwb_{\beta_1,{k-1}}(\g_{n-i},\g_{I_2}),
\end{multline*}
\begin{multline*}
[T^{\beta+1}](\d_s^{k-2}\d_I (\d_s^2\Ob\cdot\d_{v}\d_{w}\Ob)|_{s=t_j=0})=
\ogwb_{\beta,k}(\g_{j_1},\ldots,\g_{j_l})\ogwb_{1,0}(\g_1,\g_n)+\\
+\sum_{\substack{\beta_1+\beta_2=\beta+1\\ 2\le\beta_1\le \beta\\ k_1+k_2=k-2\\I_1\sqcup I_2=I}}\binom{k-2}{k_1}\ogwb_{\beta_2,{k_2+2}}(\g_{I_2})\ogwb_{\beta_1,{k_1}}(\g_1,\g_{n},\g_{I_1}),
\end{multline*}
\begin{multline*}
[T^{\beta+1}](\d_s^{k-2}\d_I (\d_{v}\d_s\Ob\cdot \d_s\d_{w}\Ob)|_{s=t_j=0})=\\
=\sum_{\substack{\beta_1+\beta_2=\beta+1\\ 1\le\beta_1\le \beta\\ k_1+k_2=k-2\\I_1\sqcup I_2=I}}\binom{k-2}{k_1}\ogwb_{\beta_1,{k_1+1}}(\g_1,\g_{I_1}) \ogwb_{\beta_2,{k_2+1}}(\g_{n},\g_{I_2}).
\end{multline*}
By the computation~\eqref{eq:ogw10n} and the divisor axiom~\eqref{ax_divisor},
\[
\ogwb_{\beta,k}(\g_{j_1},\ldots,\g_{j_l})\ogwb_{1,0}(\g_1,\g_n)
=(-1)^{\frac{n+3}{2}} \frac{1}{2}\ogwb_{\beta,k}(\g_{j_1},\ldots,\g_{j_l}).
\]
This recovers the second recursion.
\end{proof}

\begin{proof}[Proof of Corollary~\ref{cor:computable}]

By Theorem~\ref{prop:wallcross} invariants with interior constraints in $\g_\diamond$ are computable in terms of invariants with interior constraints of the form $\g_j=[\omega^j]$.
Further, by the unit~\eqref{ax_unit} and divisor~\eqref{ax_divisor} axioms, we may assume that $|\g_j|>2$.
It follows from the degree axiom~\eqref{ax_deg} that for any $\beta$ there are only finitely many values of $k,l,$ for which there may be nonzero invariants with constraints of the above type.
Thus, we give a process for computing $\ogwb_{\beta,k}(\g_{i_1},\ldots,\g_{i_l})$ which is inductive on $(\beta,k,l)$ with respect to the lexicographical order on $\Z_{\ge 0}^{\oplus 3}$.

For $\beta=0$, all values are given by the zero axiom~\eqref{ax_zero}. For $(\beta,k,l)$ with $\beta=1$ and $l\le 1$, all possible values have been computed explicitly in Theorem~\ref{recursion}. Indeed, assume for convenience that interior constraints are written in ascending degree order. By the degree axiom~\eqref{ax_deg},
\begin{align*}
\beta=1\Lrarr&\; n-3+n+1+k+2l=kn+\sum_{j=1}^l|\g_{i_l}|\\
\Lrarr&\; 0=(k-2)(n-1)+\sum_{j=1}^l(|\g_{i_l}|-2).
\end{align*}
Since $|\g_{i_j}|>2$, equality cannot occur when $k>2$. For $k=2$, equality holds if and only if $l=0$, for $k=1$ if and only if $l=1$
and $|\g_{i_1}|=n+1$, and for $k=0$ if and only if $l=1$ and $|\g_{i_1}|=2n$.

In the following, we often use the zero axiom~\eqref{ax_zero} without mention to deduce the vanishing of open Gromov-Witten invariants with $\beta = 0.$ For this purpose it is important that $W'$ is closed under the cup product so that for $A_1,A_2 \in W',$ we have $P_\R(A_1 \smallsmile A_2) = 0.$

Consider a triple $(\beta,k,l)$ with $l\ge 2$. By Theorem~\ref{thm:recursion}\ref{reduce_smallest} we can express the invariant as a combination of invariants that either have degree smaller than $\beta$, or have at most the same amount of interior constraints as the original invariant but with a smaller minimal degree. Proceed to reduce the degree of the smallest constraint until you arrive at a divisor, then eliminate this constraint by the divisor axiom~\eqref{ax_divisor}.
In the process, summands of degree $\beta$ do not increase the value of $k$.
Thus, the invariant is reduced to invariants with data of smaller lexicographical order, known by induction.

Consider a triple $(\beta,k,l)$ with $l\le 1$.
For $\beta=0,1,$ the values have been computed above. For $\beta>1$, the degree axiom~\eqref{ax_deg} implies that $k\ge 2.$
Using Theorem~\ref{thm:recursion}\ref{reduce_bd}, express the required invariant as sums of invariants that are either of smaller degree or have equal degree and less boundary marked points. Either way, we get invariants with data of smaller lexicographical order, known by induction.
\end{proof}

\begin{cor}\label{cor:integer}
All the invariants of $(\P^n,\RP^n)$ are of the form $\frac{m}{2^r}$ with $m,r\in\Z$.
\end{cor}
\begin{proof}
This is immediate from the recursive process noting that the initial conditions are integer, and the only contribution to the denominators comes from the divisor axiom. Therefore, the denominators consist of powers of $2$.
\end{proof}

\appendix
\section{The real setting}\label{app:a}
The objective of this section is to prove Theorem~\ref{thm:cases}. In particular, we operate under the assumptions of Theorem~\ref{thm:cases} throughout the section.

For any nodal Riemann surface with boundary $\Sigma$,
denote by $\overline\Sigma$ be the conjugate surface, as in Section~\ref{ssec:geoddfn}.
Denote by $\psi_\Sigma: \overline \Sigma \to\Sigma$ the anti-holomorphic map given by the identity map on points.
Denote by $\tilde\phi:\M_{l+1}(\beta)\to \M_{l+1}(\phi_*\beta)$ the map induced
by $\phi$, namely,
\[
\tilde\phi\,[u : \Sigma\to X,\vec{w}=(w_0,w_1,\ldots,w_l)]
:=
[\phi\circ u\circ \psi_\Sigma, (\psi^{-1}_\Sigma(w_0),\psi^{-1}_\Sigma(w_1),\ldots, \psi^{-1}_\Sigma(w_l))].
\]

\begin{lm}\label{lm:sgnphitilde}
We have $sgn(\phi)\equiv n\pmod 2$ and $sgn(\tilde\phi)\equiv  n+c_1(\beta)+l\pmod 2$.
\end{lm}

\begin{proof}
Since both $X$ and $\M_{l+1}(\beta)$ are complex orbifolds, and therefore admit a canonical orientation, the sign of the involution on each of them is simply half the dimension. So,
\[
sgn(\phi)=\frac{1}{2}\cdot 2n=n,
\]
and
\[
sgn(\tilde\phi)=\frac{1}{2}(2n+2c_1(\beta)+2l+2-6)
\equiv n+c_1(\beta)+l\pmod 2.
\]
\end{proof}

\begin{lm}\label{lm:qpinvt}
Let $\zeta_1,\ldots,\zeta_l\in  A^{even}(X)$ be homogeneous forms such that $\phi^*\zeta_j = (-1)^{|\zeta_j|/2}\zeta_j$. Then
\[
\phi^*\q^{\phi_*\beta}_{\emptyset,l}(\zeta_1,\ldots,\zeta_l)
=
(-1)^{|\q^\beta_{\emptyset,l}(\zeta_1,\ldots,\zeta_l)|/2}
\q^{\beta}_{\emptyset,l}(\zeta_1,\ldots,\zeta_l).
\]
\end{lm}

\begin{proof}
Denote by $ev_j$ (resp. $ev_j^\phi$) the evaluation maps on $\M_{l+1}(\beta)$ (resp. $\M_{l+1}(\phi_*\beta)$).
By assumption on $\zeta_j$, we have
\begin{align*}
(-1)^{\sum_{j=1}^l|\zeta_j|/2}
\q^\beta_{\emptyset,l}(\zeta_1,\ldots,\zeta_l)
=&
(-1)^{\sum_{j=1}^l|\zeta_j|/2}
(ev_0)_*\big(\bigwedge_{j=1}^lev_j^*\zeta_j\big)\\
=&
(ev_0)_*\big(\bigwedge_{j=1}^lev_j^*\phi^*\zeta_j\big)\\
=&
(ev_0)_*\big(\bigwedge_{j=1}^l(ev_j^\phi\circ \tilde\phi)^*\zeta_j\big)\\
=&
(ev_0)_*\tilde\phi^* \big(\bigwedge_{j=1}^l(ev_j^\phi)^*\zeta_j\big).\\
\shortintertext{By Lemma~\ref{lm:pulltopush},}
=&
(-1)^{sgn(\tilde\phi)}
(ev_0)_*(\tilde\phi^{-1})_* \big(\bigwedge_{j=1}^l(ev_j^\phi)^*\zeta_j\big)\\
=&
(-1)^{sgn(\tilde\phi)}
\phi_*(ev_0^\phi)_*\big(\bigwedge_{j=1}^l(ev_j^\phi)^*\zeta_j\big)\\
=&
(-1)^{sgn(\tilde\phi)+sgn(\phi)}
\phi^*(ev_0^\phi)_*\big(\bigwedge_{j=1}^l(ev_j^\phi)^*\zeta_j\big),\\
\shortintertext{and by Lemma~\ref{lm:sgnphitilde},}
=&
(-1)^{c_1(\beta)+l}
\phi^*\q^{\phi_*\beta}_{\emptyset,l}(\zeta_1,\ldots,\zeta_l).
\end{align*}
The desired conclusion follows by Lemma~\ref{lm:cdeg}.
\end{proof}

\begin{cor}\label{cor:realfrob}
Let $U = H^{even}_\phi(X).$
Then $U \otimes \Qh_U \subset QH_U(X)$ is a Frobenius subalgebra.
\end{cor}
\begin{proof}
For short, write $V:=U\otimes \Qh_U$.
First, $V$ is closed under $\star_U$. Indeed, let $\eta,\zeta\in A^*(X;Q_U)$ be representatives of classes in $V$ such that $\phi^*\eta = (-1)^{|\eta|/2}\eta,\phi^*\zeta = (-1)^{|\zeta|/2}\zeta.$  We have
\[
[\eta]\star_U[\zeta] = \sum_{\beta \in H_2(X;\Z), \; l \geq 0} \frac{T^{\varpi(\beta)}}{l!}\q_{\emptyset,l+2}^\beta(\eta \otimes \zeta \otimes \gamma_U^{\otimes l}).
\]
Recall that we have chosen $\sly =H_2(X,L;\Z)/\Im(\Id+\phi_*),$ so $\varpi(\beta) = \varpi(\phi_*\beta).$ Thus, Lemma~\ref{lm:qpinvt} implies that $[\eta]\star_U[\zeta] \in V.$

Second, the bilinear form $\langle \cdot , \cdot \rangle_X$ is nondegenerate on $V$.
Indeed, let $V'$ denote the direct sum over $k$ of the $(-1)^{k + 1}$-eigenspaces of $\phi^*$ acting on $H^{2k}(X)\otimes \Qh_U,$ so $V \oplus V' = H^{even}(X) \otimes \Qh_U.$
Since the bilinear form $\langle \cdot,\cdot \rangle_X$ is non-degenerate on $H^{even}(X)\otimes \Qh_U,$ it suffices to show that $V$ and $V'$ are $\langle \cdot,\cdot \rangle_X$-orthogonal. Let $\eta \in V$ and $\zeta \in V'$ be homogeneous.
In order for $\langle\eta,\zeta\rangle_X$ to be nonzero, we need $|\eta|+|\zeta|=2n$.
In addition, recall that $sgn(\phi)=n$. Therefore, by Lemma~\ref{lm:pulltopush},
\begin{multline*}
\langle\eta,\zeta\rangle_X
=
(-1)^{|\zeta|}pt_*(\eta\wedge\zeta)
=
(-1)^{|\zeta|+sgn(\phi)}pt_*\phi^*(\eta\wedge\zeta)=\\
=
(-1)^{|\zeta|+n+|\eta|/2+|\zeta|/2+1}pt_*(\eta\wedge\zeta)
=
(-1)^{|\zeta|+n+n+1}pt_*(\eta\wedge\zeta)
=
-\langle\eta,\zeta\rangle_X,
\end{multline*}
so $\langle\eta,\zeta\rangle_X=0$.
\end{proof}

\begin{proof}[Proof of Theorem~\ref{thm:cases}]
Part~\ref{it:realthma} is given by Corollary~\ref{cor:realfrob}.
Part~\ref{it:realthmc} is given by~\cite[Theorem~3]{ST2}.
It remains to verify part~\ref{it:realthmb}, namely, that $\Ker(\rho|_{H^{even}})\subset \Hh^{even}_{\phi}(X,L;\R)$.
Consider the long exact sequence~\eqref{eq:yrho}. Since $\Ker(\rho) = \Im(y),$ it suffices to show that $y(1) \in \Hh^{even}_{\phi}(X,L;\R).$
Let $\phi^*$ act on $\R[-n]$ by
the identity.
This action makes $\i$ into a $\phi^*$-equivariant map. Indeed,
for $\eta\in A^*(X;\R)$, we have
\begin{multline*}
\i_\R(\phi^*(\eta))=
(-1)^{n+|\eta|}\int_L i^*(\phi^*\eta)
=
(-1)^{n+|\eta|}\int_L(\phi\circ i)^*\eta
= \\
=(-1)^{n+|\eta|}\int_L i^*\eta
=
\i_\R(\eta)
=
\phi^*(\i_\R(\eta)).
\end{multline*}
By the naturality of the long exact sequence~\eqref{eq:yrho},
we conclude that $y$ is $\phi^*$-equivariant, and therefore $\phi^*y(1) =y(1)$. Since $n\equiv 3\pmod 4$ and $|y(1)|=n+1$, this means
$y(1)\in \Hh^{even}_\phi(X,L;\R)$.
\end{proof}

\bibliography{../../bibliography_exp}
\bibliographystyle{../../amsabbrvcnobysame}

\end{document}